\documentclass[11pt]{amsart}
\usepackage{amsmath,latexsym,amsfonts,amssymb,amsthm}
\usepackage{geometry}
\usepackage{fullpage}
\usepackage{mathrsfs}
\usepackage{graphicx,color}
\usepackage{tikz-cd}
\usepackage[all,cmtip]{xy}
\usepackage{enumitem}
\usepackage{verbatim}
\usepackage{bbm}
\usepackage{bm}
\usepackage{mathtools}
\usepackage[colorlinks=true, linkcolor=blue!70!black, urlcolor=purple, citecolor=blue!70!black]{hyperref}
\usepackage{tabu}
\usepackage{appendix}
\usepackage{wrapfig}
\usepackage{tablefootnote}
\usepackage{multirow}
\usepackage{array}

\numberwithin{equation}{section}
\newtheorem{theorem}{Theorem}[section]
\newtheorem{prop}[theorem]{Proposition}
\newtheorem{lemma}[theorem]{Lemma}

\newtheorem{thm}[theorem]{Theorem}

\newtheorem{cor}[theorem]{Corollary}

\theoremstyle{definition}
\newtheorem{corollary}[theorem]{Corollary}
\newtheorem{definition}[theorem]{Definition}
 
 \newtheorem{defn}[theorem]{Definition}

  \newtheorem{property}[theorem]{Property}

\newtheorem{remark}[theorem]{Remark}
\newtheorem{example}[theorem]{Example}

\newtheorem{defn-prop}[theorem]{Definition-Proposition}

\newcommand{\bC}{\mathbb{C}}

\newcommand{\bQ}{\mathbb{Q}}
\newcommand{\bR}{\mathbb{R}}
\newcommand{\bF}{\mathbb{F}}
\newcommand{\bG}{\mathbb{G}}
\newcommand{\calY}{\mathcal{Y}}
\newcommand{\calL}{\mathcal{L}}

\newcommand{\calX}{\mathcal{X}}

\newcommand{\PGL}{\mathrm{PGL}}
\newcommand{\GL}{\mathrm{GL}}

\newcommand{\Proj}{\mathrm{Proj}}

\newcommand{\KSBA}{{\rm KSBA}}

\newcommand{\bP}{\mathbb{P}}
\newcommand{\bA}{\mathbb{A}}

\newcommand{\calD}{\mathcal{D}}
\newcommand{\calO}{\mathcal{O}}
\newcommand{\cO}{\mathcal{O}}

\newcommand{\calM}{\mathcal{M}}

\newcommand{\calU}{\mathcal{U}}

\usepackage{graphicx}
\newcommand{\bZ}{\mathbb{Z}}

\newcommand{\oM}{\overline M}

\newcommand{\cX}{\mathcal X}

\newcommand{\cU}{\mathcal{U}}

\newcommand{\Supp}{\textrm{Supp}}

\newcommand{\hvol}{\widehat{\mathrm{vol}}}

\newcommand{\ord}{\mathrm{ord}}
\newcommand{\Val}{\mathrm{Val}}

\newcommand{\vol}{\mathrm{vol}}

\newcommand{\ind}{\mathrm{ind}}

\newcommand{\klt}{\mathrm{klt}}
\newcommand{\kst}{\mathrm{kst}}

\newcommand{\lct}{\mathrm{lct}}

\newcommand{\cM}{\mathcal M}

\newcommand{\cL}{\mathcal L}

\newcommand{\Fut}{\mathrm{Fut}}
\newcommand{\fm}{\mathfrak{m}}

\newcommand{\calC}{\mathcal C}

\newcommand{\K}{\mathrm{K}}

\begin{document}

\title{A guide to wall crossing for moduli of varieties}
\author{Kristin DeVleming}
\date{\today}

\begin{abstract}
    There have been major developments in the theory of moduli of varieties in the past decade, essentially settling the construction of moduli spaces of log canonically polarized slc pairs (c.f. \cite{KolNewBook}) and moduli spaces of K-polystable log Fano pairs (c.f. \cite{Chenyang}).  Given the construction of these moduli spaces of pairs $(X, D)$, it is natural to ask how the moduli spaces vary as the coefficients of $D$ are perturbed.  This phenomenon is known as \textit{wall crossing}, the theory of which has been developed in several important cases in the past five years.  This semi-expository article is an introduction to moduli of varieties and wall crossing, capturing a portion of the theory developed in \cite{ADLpub, Zhou, ABIP, MZ, BABWILD, BL}.  It also introduces tools and techniques used in explicit computations and examples, applying them in new examples.
\end{abstract}

\maketitle

\section{Introduction}

In the past decade, there has been a host of developments in the theory of moduli of algebraic varieties.  A moduli functor of canonically polarized varieties and pairs admitting a projective coarse moduli space has been settled upon (see \cite{KolNewBook}) and K-stability, an algebro-geometric notion originating in differential geometry, has been used to develop a good moduli theory of Fano varieties and log Fano pairs (see \cite{Chenyang}).  These different moduli theories can be used to provide compactifications of moduli spaces of smooth varieties of a given type.  

However, when compactifying moduli spaces of smooth varieties, there often exist \textit{several} choices of compactifications.  For example, if $M_g$ is the moduli space of smooth genus $g$ curves, there exists the celebrated Deligne-Mumford compactification $\oM_g$ parametrizing stable genus $g$ curves.  Several alternative compactifications exist: for example, taking the Jacobian of each curve yields a map $M_g \to A_g$, the moduli space of abelian varieties, and the closure of $M_g$ inside the Satake compactification of $A_g$ provides an alternative compactification.  In genus $2$, we could view a smooth genus $g$ curve as a double cover of $\bP^1$ branched over $6$ points, and could instead compactify the $M_2$ using Geometric Invariant Theory (GIT) for binary sextics.  If we add marked points and consider the moduli space of smooth genus $g$ curves with $n$ marked points $M_{g,n}$, we have even more flexibility: in the Deligne-Mumford compactification $\oM_{g,n}$, we parametrize stable curves with $n$ distinct points (that are distinct from any nodes of the curve).  We could form alternative compactifications by allowing the markings to come together, e.g. allowing a subset of the $n$ points to be equal. 

Each of these compactifications produces a moduli space birational to $M_g$ or $M_{g,n}$.  Given these birational varieties, it is natural to ask if there exists a morphism between them or a way to resolve the rational map. In the case of $M_{g,n}$, Hassett proves in \cite{Hassettcurves} that there are explicit birational morphisms from the moduli space $\oM_{g,n}$ to the moduli spaces where more and more points are allowed to collide. The aim of this article is to discuss in general the theory of `wall crossing' in moduli of varieties, which is a study of varying compactifications of moduli spaces and describing the morphisms between them.

In what follows, we will primarily consider moduli spaces of pairs $(X,D)$, where $X$ will be a Fano variety\footnote{This assumption is unnecessary to phrase the question, but will be used in this survey to connect several different moduli theories.} and $D \in |-rK_X|$ for some $r \in \bQ$ is an ample divisor on $X$.  Through general theory, for each $c \in (0,1]$, there exists a moduli stack parametrizing pairs $(X,cD)$ and their degenerations satisfying certain stability conditions.  As $c$ varies, \textit{how do the moduli stacks change}?  What values of $c$ cause changes in stability?  Are there explicit descriptions of these critical values of $c$ and associated moduli stacks?  In a precise way, as $c$ increases from $0$ to $1$, the change in stability conditions will come from interchanging the singularities of $D$ and the singularities of $X$: as $c$ increases, $D$ must become `less' singular while $X$ may become `more' singular.  This article will explore these phenomenona.

Wall crossing appears in many other contexts--e.g. moduli of vector bundles or sheaves--but the present writing will focus only on the case of moduli of varieties.

In what follows, we outline the proofs of the following results.  We will primarily be concerned with the case that the general point of the moduli stack $[(X,cD)]$ is such that $X$ is a smooth Fano variety and $D \in |-rK_X|$ for fixed $r \in \bQ \cap (0,1]$ is an ample divisor on $X$.  Both theorems hold in more generality; for full statements and generalizations, see \cite{ABIP, MZ, ADLpub, Zhou}. 

\begin{theorem}[\cite{ADLpub}]\label{thm:introKwallcrossing}
Let $\cM^\K_{n,v,r,c}$ (respectively, $M^\K_{n,v,r,c}$) be the proper K-moduli stack (respectively, projective K-moduli space) of K-semistable (respectively, K-polystable) log Fano pairs $(X,cD)$ such that $X$ is a smoothable klt Fano variety of $\dim X = n$, $(-K_X)^n = v$, and $D \in |-rK_X|$.  Then, there exist rational numbers \[0=c_0<c_1<c_2<\cdots<c_k=\min\{1,r^{-1}\}\] such that $c$-K-(poly/semi)stability conditions do not change for $c\in (c_i,c_{i+1})$. For each $1\leq i\leq k-1$ and $0<\epsilon\ll 1$, we have open immersions
 \[
 \cM^\K_{n,v,r,c_i+\epsilon}\xhookrightarrow{}
 \cM^\K_{n,v,r,c_i}\xhookleftarrow{}\cM^\K_{n,v,r,c_i-\epsilon}
 \]
 which induce projective morphisms of good moduli spaces
 \[
 M^\K_{n,v,r,c_i+\epsilon}\xrightarrow{}
 M^\K_{n,v,r,c_i}\xleftarrow{}M^\K_{n,v,r,c_i-\epsilon}.
 \]
\end{theorem}

\begin{theorem}[\cite{ABIP}]\label{thm:introKSBAwallcrossing}
Let $\cM^{\KSBA}_{n,V,r,c}$ (respectively, $M^{\KSBA}_{n,V,r,c}$) be the proper KSBA moduli stack (respectively, projective KSBA moduli space) of stable pairs $(X,cD)$ of $\dim X = n$, $(K_X+cD)^n = V$, with general point satisfying $(X,cD)$ klt and $D \in |-rK_X|$.  Then, there exist rational numbers \[r^{-1}=c_0<c_1<c_2<\cdots<c_k=1\] such that $c$-stability conditions do not change for $c\in (c_i,c_{i+1})$.  Let $(\cM^{\KSBA}_{n,V,r,c})^\nu$ (respectively, $(M^{\KSBA}_{n,V,r,c})^\nu$) denote the normalization.  For each $1\leq i\leq k-1$ and $0<\epsilon\ll 1$, we have morphisms
 \[
 (\cM^{\KSBA}_{n,V,r,c_i+\epsilon})^\nu\xrightarrow{}
 (\cM^{\KSBA}_{n,V,r,c_i})^\nu\xleftarrow{\cong}(\cM^{\KSBA}_{n,V,r,c_i-\epsilon})^\nu
 \]
which induce projective morphisms of coarse moduli spaces
 \[
 (M^{\KSBA}_{n,V,r,c_i+\epsilon})^\nu\xrightarrow{}
 (M^{\KSBA}_{n,V,r,c_i})^\nu\xleftarrow{\cong}(M^{\KSBA}_{n,V,r,c_i-\epsilon})^\nu.
 \]
\end{theorem}

The previous two theorems explain wall crossing when the pairs $(X,cD)$ are either log canonically polarized or log Fano.  When $\dim X = 2$, there is also a \textit{boundary polarized Calabi Yau} wall crossing developed in \cite{BABWILD, BL}: 

\begin{theorem}[\cite{BABWILD, BL}]\label{thm:introbpcywallcrossing}
In what follows, let $X$ be a smoothable slc Fano variety of $\dim X = 2$, $(-K_X)^2= v$, and $D \in |-rK_X|$. 
Let $\cM^\K_{2,v,r,r^{-1} - \epsilon}$ (respectively, $M^\K_{2,v,r,r^{-1} - \epsilon}$) be the K-moduli stack (respectively, K-moduli space) of log Fano pairs $(X,(r^{-1} - \epsilon)D)$.  Let $\cM^\KSBA_{2,v,r,r^{-1} + \epsilon}$ (respectively, $M^\KSBA_{2,v,r,r^{-1} + \epsilon}$) be the KSBA moduli stack (respectively, KSBA moduli space) of stable pairs $(X,(r^{-1} + \epsilon)D)$.  Then, there exists an Artin stack locally of finite type $\cM^{\rm CY}_{2,v,r}$ and an asymptotically good moduli space $M^{\rm CY}_{2,v,r}$ parametrizing slc log Calabi Yau pairs $(X,r^{-1}D)$.  Furthermore, there exist open immersions
 \[
 \cM^\KSBA_{2,v,r,r^{-1}+\epsilon}\xhookrightarrow{}
 \cM^{\rm CY}_{2,v,r}\xhookleftarrow{}\cM^\K_{2,v,r,r^{-1}-\epsilon}
 \]
 which induce projective morphisms
  \[
 M^\KSBA_{2,v,r,r^{-1}+\epsilon}\xrightarrow{}
 M^{\rm CY}_{2,v,r}\xleftarrow{}M^\K_{2,v,r,r^{-1}-\epsilon}.
 \]
\end{theorem}

The rest of this article will apply these results in several examples, first explaining the wall crossing for the moduli space of pairs compactifying the locus of pairs $(\bP^2, cD)$ where $D$ is a smooth plane quartic curve.   We will then study the compactification of the locus of surface pairs more generally and use the theory of wall crossing to produce KSBA stable surface pairs with arbitrarily many components.  

\begin{theorem}\label{introthm:manycomps}
    If $(X,cD)$ is a smoothable pair of dimension 2 parametrized by a KSBA moduli space such that $D$ has an isolated $A_n$ singularity for some $n$ at a smooth point of $X$, $D$ is at worst nodal elsewhere, and $K_X + D$ is ample, then in the associated KSBA moduli space of pairs with coefficient $c = 1$, there exists a surface with at least $k+n-1$ components, where $k$ is the number of irreducible components of $X$.
\end{theorem}

Applying this to the KSBA moduli space of pairs compactifying the locus of curves in $\bP^2$, we produce $\bQ$-Gorenstein degenerations of $\bP^2$ and nodal limits of plane curves with many components.  

\begin{theorem}\label{introthm:degensP2}
    In the KSBA moduli space compactifying the locus of pairs $(\bP^2, D)$ where $D \in |\calO(d)|$ is a degree $d \ge 6$ plane curve, there exists a surface with at least $\frac{d^2}{2} -1$ components if $d$ is even and at least $\frac{d(d-1)}{2} -1$ components if $d$ is odd. 
\end{theorem}

We also explore wall crossing for of quintic plane curves, providing an explicit description of the first KSBA wall crossing. 

\begin{theorem}\label{introthm:quintics}
    The first wall crossing in the KSBA moduli space compactifying the locus of pairs $(\bP^2, cD)$ where $D \in |\calO(5)|$ occurs at $c = \frac{11}{18}$, and there is an explicit description of objects parametrized by this component of $M^{\KSBA}$ for all $c \in (\frac{3}{5},\frac{5}{8})$.  
\end{theorem}

Combined with the results of \cite{ADLpub, BABWILD, hacking}, this gives a complete description of the moduli space compactifying the locus of pairs $(\bP^2, cD)$ where $D$ is a quintic curve for each $c \in (0, \frac{5}{8})$.

We also briefly discuss the case of non-proportional wall crossing, when $D \notin |-rK_X|$ for any $r \in \bQ$, and prove by direct computation: 

\begin{theorem}\label{introthm:F1}
    The K-moduli space compactifying the locus of pairs $(\bF_1, c D)$, where $D \in |2s + 4f|$, is nonempty if and only if $c \in [ \frac{1}{10}(4 - \sqrt{6}), \frac{1}{2}).$  When $c = \frac{1}{10}(4 - \sqrt{6})$, the moduli space is a single point.  If $D$ is smooth and meets the negative section of $\bF_1$ transversely, $(\bF_1, cD)$ is K-stable for every $c \in ( \frac{1}{10}(4 - \sqrt{6}), \frac{1}{2})$.
\end{theorem}

This shows the proportionality assumption is crucial for the walls to occur at rational numbers.  The necessity of the proportionality condition was already known (see, e.g. \cite{218}) but in this case the ambient space $X = \bF_1$ is K-unstable without the presence of a divisor.

\subsection*{Organization} The paper is organized as follows.  In \S \ref{sec:preliminaries}, we provide background on KSBA and K-stable pairs, defining the relevant moduli functors.  We also state versions of the KSBA moduli theorem and K-moduli theorem that will be used in this survey.  In \S \ref{sec:wallcrossing}, we introduce the theory of wall crossing and outline the proofs of the wall crossing theorems in the introduction.  In \S \ref{sec:quartics}, we study moduli spaces of plane quartic curves and the associated wall crossing.  The aim of this section is to introduce techniques relevant to explicit computation and classification in general while applying them in this example.  In \S \ref{sec:higherdegree}, we summarize wall crossing for moduli of plane quintics and compute the first KSBA wall crossing.  We conclude this section with several results on plane curves of degree $d \ge 6$, proving the results in the introduction.  Finally, in \S \ref{sec:nonprop}, we offer concluding remarks on generalizations beyond the proportional case and study moduli of curves on the Hirzebruch surface $\bF_1$. 

\subsection*{Acknowledgments} I am indebted to Jarod Alper, Victoria Hoskins, J\'anos Koll\'ar, S\'andor Kov\'acs, Radu Laza, James M\textsuperscript{c}Kernan, and Chenyang Xu, from whom I learned much of the theory of moduli of varieties in different contexts.  The results in this paper and my own perspective on the subject also owe a great deal of thanks to my collaborators cited within: Kenneth Ascher, Dori Bejleri, Harold Blum, Giovanni Inchiostro, Lena Ji, Patrick Kennedy-Hunt, Yuchen Liu, Ming Hao Quek, David Stapleton, and Xiaowei Wang. Finally, I gratefully acknowledge support from the National Science Foundation through NSF DMS Grant 2302163.

\tableofcontents

\section{Preliminaries}\label{sec:preliminaries}

\subsection{Singularities of the minimal model program, stable pairs, and K-stability}

The content in this section is taken primarily from \cite{KM98} and \cite{Chenyang}.

\subsubsection{Conventions}
Throughout, we work over $\bC$. A scheme  $X$ is \emph{demi-normal} if it is $S_2$ and its codimension 1 points are either regular or nodes \cite[Definition 5.1]{Kol13}.

\subsubsection{Singularities of pairs}
A \emph{pair} $(X,\Delta)$ consists of a demi-normal scheme $X$ that is essentially of finite type over a field and an effective $\bQ$-divisor $\Delta$ on $X$ such that  $\Supp(\Delta)$ does not contain a codimension 1 singular point of $X$ and
 $K_{X}+\Delta$ is $\bQ$-Cartier.
The pair $(X,\Delta)$ is  \emph{projective} if $X$ is projective.

\begin{definition}
    A log resolution of a pair $(X,\Delta)$ is a pair $(Y,\Delta_Y)$ such that $\pi: Y \to X$ is a resolution $Y$ of $X$ and, if $\{E_i\}$ are the exceptional divisors of $\pi$, $\mathrm{Supp} \ \pi^{-1}_* \Delta \cup \{E_i\}$ is simple normal crossing.
\end{definition}

\begin{definition}{\cite[Definition 2.28]{KM98}}
    If $(X,\Delta)$ is a pair with $X$ normal and $K_X +\Delta$ is $\bQ$-Cartier, write $\Delta = \sum b_j \Delta_j$ where $\{ \Delta_j\}$ are prime divisors and $b_j \in \bQ^{\ge 0}$.  Let $\pi: Y \to X$ be a log resolution of singularities and $\Delta_Y = \sum b_j \tilde{\Delta}_j$, where $\tilde{\Delta}_j$ is the strict transform of $\Delta_j$ on $Y$.  Let $\{ E_i \}$ be the components of the exceptional locus.  Then, for some rational numbers $a_{X,\Delta}(E_i)$, we can write
    \[ K_Y + \Delta_Y = \pi^* (K_X+\Delta) + \sum_{E_i} a_{X,\Delta}(E_i) E_i .\]

    We define the \textit{discrepancy} of the pair $(X,\Delta)$ to be \[\mathrm{discrep } (X,\Delta) = \inf_{E/X: E \text{ exceptional }} a_{X,\Delta}(E).\]  Note the infimum is taken over all exceptional divisors in any log resolution of $(X,\Delta)$.

    The pair $(X,\Delta)$ is said to have:
        \begin{itemize}
            \item \textit{terminal singularities} if $\mathrm{discrep } (X,\Delta) > 0 $,
            \item \textit{canonical singularities} if $\mathrm{discrep } (X,\Delta) \ge 0 $,
            \item \textit{Kawamata log terminal singularities} (klt) if $\mathrm{discrep } (X,\Delta) > -1 $ and $b_j < 1$ for each $j$, 
            \item \textit{purely log terminal singularities} (plt) if $\mathrm{discrep } (X,\Delta) > -1 $, and
            \item \textit{log canonical singularities} (lc) if $\mathrm{discrep } (X,\Delta) \ge -1 $.
        \end{itemize}
\end{definition}

The discrepancies can be computed on one log resolution, provided $\Delta_Y$ is smooth (i.e. any intersections among the $D_j$ are separated): 

\begin{theorem}{\cite[Corollary 2.32]{KM98}}\label{t:discrepofpair}
    Let $\pi: (Y,\Delta_Y) \to (X,\Delta)$ be a log resolution such that $\Delta_Y$ is smooth, where $\Delta_Y$ is the strict transform of $\Delta$ on $Y$.  If $a_{X,\Delta}(E_i) \ge -1$ for every exceptional divisor $E_i$ of $\pi$, then $\mathrm{ discrep } (X,\Delta) = \min \{ \min_i \{a_{X, \Delta}(E_i)\} , \min_j \{ 1- b_j \}, 1 \} $.  
\end{theorem}

Given a non-necessarily log canonical pair, we also need the following: 

\begin{definition}\label{defn:lc}
    If $\Delta \subset X$ is a nonzero $\bQ$-Cartier divisor on a normal, log canonical variety $X$ with $\bQ$-Cartier canonical divisor, the \textit{log canonical threshold} of $\Delta$ is the maximum nonnegative number $c$ such that $(X,c\Delta)$ is log canonical.  This is denoted $\lct(X,\Delta)$ or simply $\lct(\Delta)$ if $X$ is clear from context. By Theorem \ref{t:discrepofpair}, we see that $c \le 1$. 
\end{definition}

Next, suppose $X$ is non-normal.  The \emph{normalization} of a pair $(X,\Delta)$ is the possibly disconnected pair
$(\overline{X},\overline{\Delta}+\overline{G})$,
where 
$\pi:\overline{X}\to X$ is the normalization morphism, $\overline{G} \subset \overline{X}$ the conductor divisor on $\overline{X}$, and $\overline{\Delta}$  the divisorial part of $\pi^{-1}(\Delta)$, 
and satisfies
$K_{\overline{X}}+\overline{G}+\overline{\Delta}=\pi^*(K_X+\Delta)$ \cite[5.7]{Kol13}.
This construction induces a generically fixed point free involution $\tau:\overline{G}^n\to \overline{G}^n$  that fixes ${\rm Diff}_{\overline{G}^n}(\overline{\Delta})$; see  \cite[5.2 \& 5.11]{Kol13} for details, and allows us to define a non-normal version of the previous list of singularities.

\begin{definition}{\cite[Definition 5.10]{Kol13}}
    A pair $(X,\Delta)$ is \emph{semi log canonical} (slc) if its normalization $(\overline{X},\overline{\Delta}+\overline{G})$ is log canonical.
\end{definition}

\begin{defn}\label{d:FanoCYKSBA}
A projective slc\footnote{Note that we allow the pairs to be slc in this definition, while common conventions assume that a log Fano pair is klt and a CY pair is lc.} pair  $(X,\Delta)$ is called
\begin{enumerate}
	\item \emph{log Fano} if $-K_X-\Delta$ is ample,
	\item  \emph{CY} if $K_{X}+\Delta \sim_{\bQ} 0$, and 
	\item \emph{canonically polarized} if $K_{X}+\Delta$ is ample.
\end{enumerate}

Let $\mathbf{a} = (a_1, \dots, a_r) \in (\bQ\cap (0,1])^r$ be a fixed rational vector.  We say a pair $(X, \Delta)$ as above is \textit{marked} by $\mathbf{a}$ if $\Delta = \sum a_i \Delta_i$ where each $\Delta_i$ is an effective $\bZ$-divisor. 

The \emph{index} of a CY pair $(X,\Delta)$ is the smallest positive integer $N$ such that $N(K_X+\Delta)\sim0$.
\end{defn}

It will turn out that slc pairs give rise to a well-behaved moduli theory of canonically polarized pairs.  In the log Fano case, however, we need a finer stability notion to have a well-behaved moduli theory, called \textit{K-stability.}

\subsubsection{K-stability via the Futaki invariant}

Most of the material in this section comes from \cite{Chenyang}. 

\begin{defn}[\cite{Tian, Donaldson}]
Let $(X,\Delta=\sum_{i=1}^k b_i \Delta_i)$ be a projective log pair such that $X$ is normal. Let $L$ be an ample line bundle on $X$.
\begin{enumerate}[label=(\alph*)]
\item A \emph{test configuration} of index $r$ $(\cX;\cL_r)/\bA^1$ of $(X;L)$ consists of the following data:
\begin{itemize}
 \item a variety $\cX$ together with a flat projective morphism $\pi:\cX\to \bA^1$;
 \item a $\pi$-ample line bundle $\cL_r$ on $\cX$;
 \item a $\bG_m$-action on $(\cX;\cL_r)$ such that $\pi$ is $\bG_m$-equivariant with respect to the standard action of $\bG_m$ on $\bA^1$ via multiplication;
 \item $(\cX\setminus\cX_0;\cL_r|_{\cX\setminus\cX_0})$
 is $\bG_m$-equivariantly isomorphic to $(X;L^r)\times(\bA^1\setminus\{0\})$.
\end{itemize}
 
\item A \emph{test configuration} $(\cX,\mathcal{D};\cL_r)/\bA^1$ of $(X,\Delta;L)$ consists of the following data:
\begin{itemize}
\item a test configuration $(\cX;\cL_r)/\bA^1$ of $(X;L)$;
\item a formal sum $\mathcal{D}=\sum_{i=1}^k c_i \mathcal{D}_i$ of codimension one closed integral subschemes $\mathcal{D}_i$ of $\cX$ such that $\mathcal{D}_i$ is the Zariski closure of $\Delta_i\times(\bA^1\setminus\{0\})$ under the identification between $\cX\setminus\cX_0$ and $X\times(\bA^1\setminus\{0\})$.
\end{itemize}
\item A test configuration $(\cX,\mathcal{D};\cL_r)/\bA^1$ is called a \emph{normal} test configuration if $\cX$ is normal. 
A normal test configuration is called a \emph{product} test configuration if \[
(\cX,\mathcal{D};\cL_r)\cong(X\times\bA^1,\Delta\times\bA^1;pr_1^* L^r\otimes\cO_{\cX}(k\cX_0))
\] for some $k\in\bZ$. A product test configuration is called a \emph{trivial} test configuration if the above isomorphism is $\bG_m$-equivariant with respect to the trivial $\bG_m$-action on $X$ and the standard $\bG_m$-action on $\bA^1$ via multiplication. 
\item 
Let $(X,\Delta)$ be a klt log Fano pair. Let $L$ be an ample $\bQ$ line bundle on $X$ such that $L\sim_{\bQ}-(K_X+\Delta)$. A normal test configuration $(\cX,\mathcal{D};\cL_r)/\bA^1$ is called a \emph{special test configuration} if $\cL_r\sim_{\bQ}-r(K_{\cX/\bA^1}+\mathcal{D})$ and $(\cX,\mathcal{D}+\cX_0)$ is plt. In this case, we say that $(X,\Delta)$ \emph{specially degenerates to} $(\cX_0,\mathcal{D}_0)$ which is necessarily a klt log Fano pair.
\item Given a test configuration $(\cX,\mathcal{D};\cL_r)/\bA^1$ and the induced isomorphism \[(\cX\setminus\cX_0, \mathcal{D} \setminus \mathcal{D}_0;\cL_r|_{\cX\setminus\cX_0}) \cong (X,\Delta;L^r)\times(\bA^1\setminus\{0\}),\] we may $\bG_m$-equviariantly glue this along the trivial test configuration $(X,\Delta; L^r) \times \bA^1 \to \bA^1 $ along $\bA^{1} \setminus \{ 0\}$ to produce the \textit{$\infty$-trivial} compactification $\overline{\pi}: (\overline{\cX}, \overline{\mathcal{D}}; \overline{\cL}_r) \to\bP^1$.
\end{enumerate}
\end{defn}

Associated to any test configuration, we can define a \textit{Futaki invariant}, coming from the weight of the $\bG_m$-action.  We provide a definition only in the case of special test configurations due to \cite{LX14}.  

\begin{definition}
    For any log Fano pair $(X,\Delta)$ and special test configuration $(\calX, \mathcal{D}; \calL_r)$, the \textit{Futaki invariant} is 
    \[ \Fut(\calX, \mathcal{D}; \calL_r) = - \frac{1}{2(-K_X-\Delta)^n(n+1)} \left(-K_{\overline{\calX}/\bP^1} - \overline{\mathcal{D}}\right)^{n+1} ,\]
    where $(\overline{\calX},\overline{\mathcal{D}})$ is the $\infty$-trivial compactification.
\end{definition}

\begin{definition}
     Let $(X,\Delta)$ be a log Fano pair.  Then, $(X,\Delta)$ is said to be:
\begin{enumerate}[label=(\roman*)]
    \item \emph{K-semistable} if $\Fut(\cX,\mathcal{D};\cL_r)\geq 0$ for any special test configuration of any index $r$;
    
    \item  \emph{K-stable} if it is K-semistable and $\Fut(\cX,\mathcal{D};\cL_r)=0$ for a special test configuration $(\cX,\mathcal{D};\cL_r)/\bA^1$ if and only if it is a trivial test configuration; and
 
\item \emph{K-polystable} if it is K-semistable and $\Fut(\cX,\mathcal{D};\cL_r)=0$ for a special test configuration $(\cX,\mathcal{D};\cL_r)/\bA^1$ if and only if it is a product test configuration.
\end{enumerate}
\end{definition}

In the previous definition, we may assume that $(X, \Delta)$ is klt by the following result of Odaka \cite{Odaka13}: 

\begin{theorem}\label{thm:kssisklt}
    Suppose $(X, \Delta)$ is a K-semistable log Fano pair.  Then, $(X, \Delta)$ is klt. 
\end{theorem}

These conditions can be viewed as a sort of `asymptotic' GIT condition (and in some cases, is equivalent to GIT, see Theorem \ref{thm:k=git}). Indeed, a degeneration along a one-parameter subgroup in GIT will give rise to a test configuration, and the Futaki invariant (under suitable circumstances) will be related to the Hilbert-Mumford weight.  The key difference, however, is that in GIT, one-parameter subgroups degenerate $X$ inside a \textit{fixed} projective space, and in K-stability, the central fiber of the test configuration does not need to embed in the same ambient projective space as $X$.  For computational purposes, it is often convenient to have an alternative definition of K-stability, given below. 

\subsubsection{K-stability via the $\delta$ invariant}

\begin{definition}
Let $(X,\Delta)$ be a klt log Fano pair of dimension $n$ and $E$ a prime divisor over $X$.  Let $\mu: Y \to X$ be any morphism such that $E \subset Y$.

Let $A_{X,\Delta}(E)$ be the \textit{log discrepancy} of the divisor $E$, defined as \[ A_{X,\Delta}(E) = 1 + a_{X,\Delta}(E),\]
where $a_{X,\Delta}(E)$ is the discrepancy as defined above. 

Define $S_{X,\Delta}(E)$ to be 
\[ S_{X,\Delta}(E) = \frac{1}{(-K_X-\Delta)^n} \int_0^\infty \vol(\mu^*(-K_{X} -  \Delta) - tE) dt ,\] where the \textit{volume} of divisor $D$ on a normal variety $X$ of dimension $n$ is 
    \[ \vol(D) = \lim_{m\to \infty} \frac{h^0(X,mD)}{m^n/n!}.\] 
This does not depend on choice of $\mu$ and $Y$, so we often write \[ S_{X,\Delta}(E) = \frac{1}{(-K_X-\Delta)^n} \int_0^\infty \vol(-K_X - \Delta- tE) dt .\]

The $\delta$-invariant of $E$ is \[\delta_{X,\Delta}{E} = \frac{A_{X,\Delta}(E)}{S_{X,\Delta}(E)}. \]
\end{definition}

\begin{remark}
In the definition of $S_{X,\Delta}(E)$, we need to compute an improper integral.  However, $\vol(\mu^*(-K_X - \Delta) - tE) > 0$ if and only if $-\mu^*(-K_X - \Delta) - tE$ is big, and hence the volume is only non-zero if $t \in [0, \tau ]$ where $\tau$ is the \textit{pseudo-effective threshold}, which is finite. Therefore, we could instead write 
\[ S_{X,\Delta}(E) = \frac{1}{(-K_X- \Delta)^n} \int_0^\tau \vol(-K_X - \Delta- tE) dt .\]
\end{remark}

With this definition, we can relate the K-(semi/poly)stability of $X$ to the birational geometry of $X$.  The following theorem is known as the \textit{valuative criterion} for K-(semi/poly)stability. For the full strength of the theorem as stated, we need the equivalence of uniform K-stability and K-stability, see \cite{LXZ}.

\begin{theorem}[\cite{FuValuative,LiValuative}]\label{thm:valcriteria}
The log Fano pair $(X, \Delta)$ is K-semistable (resp. stable) if and only if $\delta_{X,\Delta}(E) \ge 1$ (resp. $> 1$) for all prime divisors $E$ over $X$.
\end{theorem}

This theorem will prove essential in wall crossing computations.

\subsection{Families of stable pairs and moduli functors}

Most of the content in this section comes from \cite{KolNewBook}, with some background written in \cite{BABWILD}.  Defining the correct notion of a family of pairs, needed to construct moduli stacks, is a subtle and technical point.  We only present a definition over reduced bases.

To define a family of pairs, we must first define a family of divisors \cite[Definition 4.68]{KolNewBook}. 

\begin{defn}[Relative Mumford divisor]
Let $f:X\to B$ be a flat finite type morphism with $S_2$ fibers of pure dimension $n$. 
A subscheme $D\subset X$ is a \emph{relative Mumford divisor} if there is an open set $U\subset X$ such that 
\begin{enumerate}
\item ${\rm codim}_{X_b}(X_b\setminus U_b) \geq 2$ for each $s\in S$, 
\item $D\vert_U$ is a relative Cartier divisor,
\item $D$ is the closure of $D\vert_U$, and
\item $X_b$ is smooth at the generic points of $D_b$ for every $b\in B$.
\end{enumerate}
\end{defn}
	
If $D\subset X$ is a relative Mumford divisor for $f:X\to B$ and $B'\to B$ is a morphism, then the \emph{divisorial pullback} $D_{B'}$  on $X_{B'}: = X\times_B B'$ is the relative Mumford divisor defined to be the closure of the pullback of $D\vert_U$ to $U_{B'}$. Note that $D_b$ as in (4) always denotes the divisorial pullback.  A key property of Mumford divisors is that (a) it makes sense to `restrict' relative Mumford divisors to fibers of a family of varieties and that (b) different notions of `restriction' coincide (see \cite[Proposition 2.79]{KolNewBook}), making them well suited to the study of moduli of varieties.

\begin{defn}\label{def:mumford}	
A \emph{family of slc (resp. lc, klt) pairs} $(X,\Delta) \to B$  over a reduced Noetherian scheme $B$
is a flat finite type morphism $X\to B$ with $S_2$ fibers and a $\bQ$-divisor $\Delta$ on $X$   satisfying
\begin{enumerate}
\item each prime component of $\Delta$ is a relative Mumford divisor,
\item $K_{X/B} +\Delta$ is $\bQ$-Cartier, and
\item  $(X_b,\Delta_b)$ is an slc (resp. lc, klt) pair for all points $b\in B$.
\end{enumerate}

As in Definition \ref{d:FanoCYKSBA}, if $\mathbf{a} = (a_1, \dots, a_r) \in (\bQ\cap (0,1])^r$ is a fixed rational vector, we say $(X, \Delta)$ is marked by $\mathbf{a}$ if $\Delta = \sum a_i \Delta_i$ where each $\Delta_i$ is a $\bZ$-divisor. 
\end{defn}

Using this definition, we can define a family of the classes of pairs appearing in Definition \ref{d:FanoCYKSBA}.

\begin{defn}
A  family of slc pairs $(X,\Delta)\to B$ over a reduced Noetherian scheme $B$ with $X\to B$ projective is called 
\begin{enumerate}
 \item \emph{a family of KSBA stable pairs} if $K_{X/B} + \Delta$ is relatively ample,
    \item \emph{a family of CY pairs} if $K_{X/B} + \Delta \sim_{\bQ,B} 0$, and
    \item \emph{a family of K-semistable log Fano pairs} if $-(K_{X/B}+\Delta)$ is relatively ample and $(X_b, \Delta_b)$ is a K-semistable pair for all points $b \in B$.
\end{enumerate}
\end{defn}

The correct definition of family over an arbitrary base scheme requires great care and technical definitions.  We refer the interested reader to \cite[Sections 6 - 8]{KolNewBook} for details.  From this definition, fixing the numerical data of a marking $\mathbf{a} = (a_1, \dots, a_r) \in (\bQ\cap (0,1])^r$, a positive integer $n$, and a positive rational number $V$, we define a moduli stack $\calM_{n,V,\mathbf{a}}$ whose objects over a base scheme $B$ are families of stable pairs of relative dimension $n$ and volume $V$.  Two fundamental theorems from the past decade provide moduli stacks and projective good moduli spaces of KSBA and K-semistable pairs.  For the technical details of the proofs, we refer the reader to \cite{KolNewBook} and \cite{Chenyang}.

\begin{theorem}[K-Moduli Theorem]\label{kmodthm}
    For each fixed $\mathbf{a} = (a_1, \dots, a_r) \in (\bQ\cap (0,1))^r$, $n$, and $V$, there exists a proper Artin stack $\cM^{\K}_{n,V,\mathbf{a}}$ and projective good moduli space $M^{\K}_{n,V,\mathbf{a}}$ parametrizing, respectively, K-semistable and K-polystable pairs $(X, \Delta = \sum_{i=1}^r a_i D_i)$ such that $\dim X = n$ and $(-K_X - \Delta)^n = V$.
\end{theorem}

\begin{theorem}[KSBA Moduli Theorem]\label{ksbamodthm}
    For each fixed $\mathbf{a} = (a_1, \dots, a_r) \in (\bQ\cap (0,1])^r$, $n$, and $V$, there exists a proper Deligne-Mumford stack $\cM^{\KSBA}_{n,V,\mathbf{a}}$ and projective coarse moduli space $M^{\KSBA}_{n,V,\mathbf{a}}$ parametrizing KSBA stable pairs $(X, \Delta = \sum_{i=1}^r a_i D_i)$ such that $\dim X = n$ and $(K_X + \Delta)^n = V$.
\end{theorem}

In this survey, we will be primarily interested in the following setting: suppose the general point of either moduli stack is a pair $(X,cD)$ such that $X$ is a $\bQ$-Gorenstein Fano variety and $D \in |- rK_X|$ for some $r \in \bQ$, $r > 0$, for $c \in \bQ \cap (0, 1]$.  Because $K_X + cD = K_X - cr K_X = (1- cr)K_X $, for $c  < r^{-1}$, the pair $(X,cD)$ is log Fano, and for $c > r^{-1}$, the pair $(X,cD)$ is canonically polarized.  In either case, the previous theorems provide stacks and moduli spaces compactifying the locus of such pairs.  Furthermore, if $n = \dim X$ and $v = \vol(X) = (-K_X)^n$, then $(K_X  + cD)^n = ((1-cr)K_X)^n = ((cr -1)(-K_X))^n = (cr-1)^n v$ and similarly $(-K_X -cD)^n = (1-cr)^n v$ and thus the volume $V$ of the pair $(X,cD)$ is determined by $c, r,$ and $v$.  Therefore, in the setting above where $\mathbf{a} = c$, for each $c \in \bQ \cap (0,1]$, there exists a proper Artin stack $\cM_{n,v,r,c}$ and good moduli space parametrizing either KSBA stable pairs (if $c > r^{-1}$) or K-semistable and K-polystable pairs (if $c < \min\{1, r^{-1}\}$) with invariants $(K_X + cD)^n = (cr-1)^n v$.  In the K-moduli stack, we can actually assume that every pair $(X,cD)$ satisfies $D \in |-rK_X|$ by the following remark.  (This does not hold in the KSBA setting; see e.g. Example \ref{ex:octic}).

\begin{remark}\label{rmk:fiberwiseDinKX}
    If $\pi: \calX \to B$ is a family of normal varieties over a smooth curve germ $0 \in B$ and $\calL$ is a line bundle on $\calX$ such that $\calL|_{B-\{0\}} \sim_{B - \{0\}} \calO_{B - \{ 0\}}$, then $\calL \sim_B \calO_B$ as $\calL$ can only differ from $\calO_B$ by a multiple of $\calO(X) = \pi^*\calO(1)$ (c.f. \cite[Proposition 6.5]{hartshorne}).  Therefore, if $(\calX, \calD)$ is a family of pairs with $\pi: \calX \to B$ a family of normal varieties and $rK_X + D \sim 0$ on the general fiber, i.e. $D \in |-rK_X|$, then $rK_{\calX} + \calD \sim_B 0$ so $rK_X + D \sim 0$ on each fiber.  Now supposing in addition that $K_{\calX} + c\calD$ is $\bQ$-Cartier for some $c \ne r^{-1}$ (which holds, in particular, if $K_{\calX} + c\calD$ is relatively ample or anti-ample), then each of $K_{\calX}$ and $\calD$ are a linear combination of two $\bQ$-Cartier divisors so are $\bQ$-Cartier.  As $rK_X + D \sim 0$ on each fiber, this implies $D \in |-rK_X|$ on each fiber. 
\end{remark}

The fundamental question in this paper is \textit{how do these moduli spaces vary with $c$}?

\section{Wall crossing for moduli of pairs}\label{sec:wallcrossing}

\subsection{K-moduli of log Fano pairs}

The first result in wall crossing is that the K-moduli stacks and spaces defined above admit a wall-crossing framework with respect to the parameter $c$. This was first studied in \cite{ADLpub} and follow-up work by \cite{Zhou}.

From the previous section, for choice of marking $c \in \bQ$, invariants $n \in\bZ^{> 0}$, $v \in \bQ^{>0}$, and $r \in \bQ^{>0}$, there exists a proper Artin stack $\calM^K_{n,v,r,c}$ and projective good moduli space $M^K_{n,v,r,c}$ parametrizing K-semistable and K-polystable log Fano pairs $(X,cD)$ such that $\dim X = n$ and $(1-cr)^nv = (-K_X - cD)^n$.  Restricting to the components of this stack parametrizing \textit{smoothable} log Fano pairs\footnote{This assumption is unnecessary but will be made for clarity of exposition.} such that $D \in |-rK_X|$ on the generic fiber, by Remark \ref{rmk:fiberwiseDinKX}, we may assume that this stack and space parametrizes pairs $(X,cD)$ such that $K_X$ and $D$ are both $\bQ$-Cartier and the relation $D \in |-rK_X|$ holds for every such pair.  We will therefore describe the K-moduli stacks and spaces in the following way: 

\textit{For choice of marking $c \in \bQ$, invariants $n \in\bZ^{> 0}$, $v \in \bQ^{>0}$, and $r \in \bQ^{>0}$, there exists a proper Artin stack $\calM^K_{n,v,r,c}$ and projective good moduli space $M^K_{n,v,r,c}$ parametrizing smoothable K-semistable and K-polystable log Fano pairs $(X,cD)$ such that $\dim X = n$ and $(-K_X)^n = v$, and $D \in |-rK_X|$.} These stacks and spaces depend on the coefficient $c$ and satisfy a wall crossing framework as given by the following theorem.

\begin{theorem}
Let $\cM^\K_{n,v,r,c}$ (respectively, $M^\K_{n,v,r,c}$) be the K-moduli stack (respectively, K-moduli space) of log Fano pairs $(X,cD)$ such that $X$ is a smoothable slc Fano variety of $\dim X = n$, $\vol(X) = v$, and $D \in |-rK_X|$.  Then, there exist rational numbers \[0=c_0<c_1<c_2<\cdots<c_k=\min\{1,r^{-1}\}\] such that $c$-K-(poly/semi)stability conditions do not change for $c\in (c_i,c_{i+1})$. For each $1\leq i\leq k-1$ and $0<\epsilon\ll 1$, we have open immersions
 \[
 \cM^\K_{n,v,r,c_i+\epsilon}\xhookrightarrow{}
 \cM^\K_{n,v,r,c_i}\xhookleftarrow{}\cM^\K_{n,v,r,c_i-\epsilon}
 \]
 which induce projective morphisms
 \[
 M^\K_{n,v,r,c_i+\epsilon}\xrightarrow{}
 M^\K_{n,v,r,c_i}\xleftarrow{}M^\K_{n,v,r,c_i-\epsilon}.
 \]
\end{theorem}

The idea of the proof is as follows: for each choice of coefficient $c$, a moduli stack/space exists as claimed, so if suffices to understand the behavior of the walls $c_i$.  The critical values $c_i$ can only arise where a pair $(X,cD)$ is K-semistable for some $c$ but $(X,c'D)$ is K-unstable for some $c' = c \pm \epsilon$.  We wish to show such critical values lie in a finite set of rational numbers.  We approach this by defining the \textit{upper} and \textit{lower K-semistable thresholds}: 

\begin{definition}
    Let $(X,D)$ be a pair such that $X$ is a $\bQ$-Gorenstein Fano variety.  The \textit{upper K-semistable threshold of} $D$ is $\kst_+(D) = \sup \{ c \mid (X,cD) \text{ is K-semistable} \}$.  The \textit{lower K-semistable threshold of} $D$ is $\kst_-(D) = \inf \{ c \mid (X,cD) \text{ is K-semistable} \}$. 
\end{definition}

The idea of the proof is to prove that the set  
\[ \{(X,D) \mid (X,cD) \text{ is smoothable and K-semistable for some } c \in (0,\min\{1,r^{-1}\}) \}\]
is log bounded, so all such elements are contained in a universal family over a Hilbert scheme.  Then, one shows that the K-semistable thresholds are constructible and semicontinuous, and attained at rational numbers.  Applying this to the universal family over the given Hilbert scheme will provide the finitely many rational walls.  The original method of proof in \cite{ADLpub} used analysis and differential geometry, but with advances in the algebro-geometric theory of K-stability, we can simplify the proof in a similar fashion to \cite{Zhou}. 

\begin{proof}[Proof sketch]
    Recall that K-semistability of a pair $(X,cD)$ can be computed in terms of the $\delta$-invariant: 
    \[ \delta(X,cD) = \inf_{E/X}\delta_{X,cD}(E) = \inf_{E/X}\frac{A_{X,cD}(E)}{S_{X,cD}(E)}. \]
    By definition of $A_{X,cD}(E)$, and because $D \in |-rK_X|$  using the definition of $S_{X,cD}(E)$, we can compute 
    \[ A_{X,cD}(E) = A_{X}(E) - c \ord_D(E)\quad \text{  and  } \quad S_{X,cD}(E) = (1-cr) S_{X}(E).\] Assume there exists some $c$ such that $(X,c'D)$ is K-semistable for all $c' \in (c-\epsilon, c]$ but unstable for $c' > c$.  Necessarily, $(X,cD)$ must be K-semistable but not K-stable (see, e.g. \cite[Proposition 3.18]{ADLpub}).  As $(X,c'D)$ is K-unstable for each $c' > c$, by \cite[Theorem 1.2]{LXZ}, there exists a divisor $E_n$ such that  
    \[ \delta(X,(c+\tfrac{1}{n}D)) = \delta_{X,(c+\frac{1}{n})D}(E_n) = \frac{A_{X,(c+\frac{1}{n})D}(E_n)}{S_{X,(c+\frac{1}{n})D}(E_n)} < 1 \]  and furthermore, this divisor $E_n$ induces a special test configuration of $(X,cD)$ to some $(X_n, cD_n)$ with 
    \[ \delta(X_n,(c+\tfrac{1}{n}D_n)) = \delta(X,(c+\tfrac{1}{n}D)). \]
    As $(X,cD)$ was assumed to be K-semistable, we must have
    \[ \delta_{X,cD}(E_n) = \frac{ A_{X,cD}(E_n) }{S_{X,cD}(E_n) }\ge  1 .\]

    We claim that $\{X_n\}_{n \ge n_0}$ must be bounded for $n_0$ sufficiently large.  First, using \cite[Lemma 2.6]{BJ}, for any divisor $F$ over $X$, we have $\ord_D(F) \le a S_{X}(F)$ for some constant $a$ depending only on $r$ and the dimension of $X$, so 
    \[ \frac{A_{X,(c+\tfrac{1}{n})D}(F)}{S_{X,(c+\tfrac{1}{n})D}(F)} = \frac{A_{X,cD}(F) - \tfrac{1}{n}\ord_D(F)}{(1-(c+\frac{1}{n})r)S_{X}(F)} = \frac{A_{X,cD}(F) - \tfrac{1}{n}\ord_D(F)}{(1-(c+\frac{1}{n})r)\tfrac{(1-cr)}{(1-cr)}S_{X}(F)}  \ge \frac{1-cr}{(1-(c+\tfrac{1}{n})r)} - \frac{\tfrac{a}{n}}{1-(c+\tfrac{1}{n})r} \]
    where the inequality follows from the K-semistability of $(X,cD)$ and the inequality $\ord_D(F) \le a S_{X}(F)$. This implies \[ \delta(X, (c + \tfrac{1}{n})D) \ge \frac{1-cr}{(1-(c+\tfrac{1}{n})r)} - \frac{\tfrac{a}{n}}{1-(c+\tfrac{1}{n})r} .\]
    As 
    \[ \delta(X_n) \ge (1 - (c + \tfrac{1}{n})r) \delta(X_n, (c + \tfrac{1}{n})D_n) = (1 - (c + \tfrac{1}{n})r) \delta(X, (c + \tfrac{1}{n})D)\]
    we have 
    \[ \delta(X_n) \ge 1 - cr - \tfrac{a}{n}.\]
    Choosing a fixed $n_0 \gg 0$ proves $\{ \delta(X_n) \mid n \ge n_0\}$ is bounded away from $0$ and therefore by \cite{Jiang} form a bounded family.  As $D_n \in |-rK_{X_n}|$ is a fixed rational multiple of the anticanonical divisor, the set $\{ (X_n, D_n)\}$ is therefore log bounded and hence by \cite[Theorem 1.1]{BLXOpenness}, $\delta(X_n, cD_n)$ takes only finitely many values.  As $n \to \infty$, $\delta(X,(c+\tfrac{1}{n}D))$ must approach a value $\ge 1$, so we conclude  $\delta(X_n, cD_n) \ge 1$ for all $n \gg 1$ and in fact must equal $1$ as $(X,cD)$ was strictly semistable.  As $E_n$ induced a special test configuration $(X,cD) \rightsquigarrow (X_n, cD_n)$, we conclude (see, e.g. \cite[Lemma 3.4]{Zhou}) 
    \[ A_{X}(E_n) - c \ord_D(E_n) - (1-cr) S_X(E_n) = 0.\]  By \cite{BLZ}, $S_X(E_n)$ is rational, so $c = \frac{ A_X(E_n) - S_X(E_n)}{\ord_D(E_n) - r S_X(E_n)} $ is rational as all terms in this expression are rational. This argument will imply the rationality of $\kst_+(D)$ and a similar argument holds for $\kst_-(D)$.  Next, one can show that, for a family $(X, \Delta) \to B$, these thresholds are constructible functions on $B$.  This follows from the constructiblity of $\delta$ in \cite{BLXOpenness} and ACC for log canonical thresholds from \cite{HMX}. 

    Finally, we show boundedness. First, we have boundedness of the set
    \[ \{(X,D) \mid (X,cD) \text{ is K-semistable for some } c \in (0,\min\{1,r^{-1}\} - \epsilon) \}\] for any fixed $\epsilon > 0$, which follows from the main result in \cite{Birkar}.  Then, we prove there is a fixed $\epsilon_0 > 0$ such that a pair $(X,cD)$ is K-semistable for some $c \in (0, \min\{1,r^{-1}\})$ if and only if $(X,cD)$ is K-semistable for some $c \in (0, \min\{1,r^{-1}\} - \epsilon_0)$.  This will imply boundedness of the set  
    \[ \{(X,D) \mid (X,cD) \text{ is smoothable and K-semistable for some } c \in (0,\min\{1,r^{-1}\}) \}.\]
      
    For ease of notation, assume $r^{-1} \le 1$. First, by \cite[Lemma 3.4]{ADLpub} or \cite[Theorem 1.8]{LXZ}, the set of pairs $(X,D)$ such that $X$ is a smooth Fano variety satisfies this property for some fixed $\epsilon_1 > 0$.  Now, assume $(X, cD)$ is smoothable and K-semistable for some $c \in (0,1)$.  By the smoothability assumption, we may find a family $(\calX, \calD) \to B$ such that $(\calX_0, c\calD_0) = (X, cD)$ and $(\calX_b, c\calD_b)$ is smooth for $b \ne 0$.  For $c > r^{-1} - \epsilon_1$, we have the K-semistability of $(\calX_b, c' \calD_b)$  for any $c' \in (r^{-1} - \epsilon_1, r^{-1})$.  By properness of the K-moduli space, there exists a K-polystable limit $(X', c'D')$ of the family $(\calX_b, c' \calD_b)$. By ACC for log canonical thresholds, there exists some fixed $\epsilon_2$ such that for any such $(X',D')$, $(X',(r^{-1} - \epsilon_2)D')$ is klt if and only if $(X',r^{-1}D')$ is log canonical.  Therefore, assuming $\epsilon_3 \le \min(\epsilon_1, \epsilon_2)$, for any $c' \in (r^{-1} - \epsilon_3, r^{-1})$, we have K-polystability of $(X',c'D')$ by interpolation (see Proposition \ref{prop:k-interpolation}).  Let $\epsilon_0 = \frac{\epsilon_3}{2}$.  Then, for any $c \in [r^{-1} - \epsilon_0, r^{-1})$, the K-semistability of $(X,cD)$ and properness of the K-moduli space implies that $(X,cD)$ specially degenerates to $(X',cD')$.  By openness of K-semistability, this implies $(X,cD)$ is K-semistable for every $c \in [r^{-1} - \epsilon_0, 1)$ and in particular $(X,(r^{-1} - \epsilon_0)D)$ is K-semistable.  
    
    For the converse, note that the K-semistability of $(X,(r^{-1} - \epsilon_0)D)$ implies it is klt, and $0 \le \epsilon_0 \le \epsilon_2$ implies $(X, r^{-1} D)$ is log canonical.  Therefore, by interpolation, we conclude $(X,cD)$ is K-semistable for any $c \in[r^{-1}-\epsilon_0, r^{-1})$.

    Finally, applying the rationality and constructibility of $\kst_+$ and $\kst_-$ to to the universal family over the Hilbert scheme parametrizing this bounded set of pairs, we conclude there are finitely many rational walls.  The claimed morphisms of stacks then follow by construction on the universal family over each moduli stack. 
\end{proof}

\subsection{Wall crossing for canonically polarized pairs}

From the previous section, for each $n \in \bZ^{>0}$, $v \in \bQ^{>0}$, and $r, c \in \bQ^{>0}$ such that $c \in (r^{-1}, 1)$, there exists a proper Deligne-Mumford stack and projective moduli space of KSBA stable pairs $(X,cD)$ such that $\dim X = n$, $(K_X + cD)^n = (cr-1)^nv$.  We will restrict to the components of this moduli stack and space such that the general point $(X,cD)$ is a klt pair for all $c \in (r^{-1},1)$ and $D \in |-rK_X|$.  The aim of this section is to sketch the proof of the following theorem. 

\begin{theorem}
Let $\cM^{\KSBA}_{n,v,r,c}$ (respectively, $M^{\KSBA}_{n,v,r,c}$) be the KSBA moduli stack (respectively, KSBA moduli space) of stable pairs $(X,cD)$ such that $\dim X = n$, $(K_X+cD)^n = (cr-1)^nv$, and the general point is klt for all $c \in (r^{-1},1)$ and $D \in |-rK_X|$.  Then, there exist rational numbers \[r^{-1}<c_0<c_1<c_2<\cdots<c_k=1\] such that $c$-stability conditions do not change for $c\in (c_i,c_{i+1})$.  Let $(\cM^{\KSBA}_{n,v,r,c})^\nu$ (respectively, $(M^{\KSBA}_{n,v,r,c})^\nu$) denote the normalization.  For each $1\leq i\leq k-1$ and $0<\epsilon\ll 1$, we have morphisms
 \[
 (\cM^{\KSBA}_{n,v,r,c_i+\epsilon})^\nu\xrightarrow{}
 (\cM^{\KSBA}_{n,v,r,c_i})^\nu\xleftarrow{\cong}(\cM^{\KSBA}_{n,v,r,c_i-\epsilon})^\nu
 \]
and
 \[
 (M^{\KSBA}_{n,v,r,c_i+\epsilon})^\nu\xrightarrow{}
 (M^{\KSBA}_{n,v,r,c_i})^\nu\xleftarrow{\cong}(M^{\KSBA}_{n,v,r,c_i-\epsilon})^\nu.
 \]
\end{theorem}

Note that the assumption that a general pair $(X,D)$ of this form is klt for all $c \in (r^{-1,},1)$, which is implied if $(X,D)$ is dlt, occurs frequently in practice, e.g. if $\cM$ is a component of a moduli space compactifying log smooth pairs. To give an idea of the proof, suppose $(X,cD)$ is stable for some $c$ but $(X,c'D)$ is unstable for $c' = c\pm \epsilon$. Then, either $(X, c'D)$ has worse than slc singularities, so perturbing the coefficient made the singularities worse, or $K_X + c'D$ is not ample, so perturbing the coefficient changed the positivity of this divisor.  Assuming $X$ is $\bQ$-factorial (or at least $K_X$ is $\bQ$-Cartier), the first condition occurs when we cross the log canonical threshold of $D$, which is a rational number, and the second condition occurs when we cross the nef threshold of $(X,D)$, i.e. the coefficient for which $K_X + c'D$ is nef but not ample, which is also rational.  These rational critical values of $c$ will be the walls in the moduli problem.  We then use techniques in the MMP--taking canonical models--to produce the wall crossing.

\begin{proof}[Proof sketch.]
    Because the general point $(X,cD)$ is klt for all $c \in (r^{-1},1)$, for any such $c$ and any stable pair $(X,cD)$, it admits a partial smoothing $(\calX, c\calD) \to B$ with central fiber $(\calX_0, c\calD_0) = (X,cD)$ and general fiber $(\calX_b, c\calD_b)$ klt.  Because the general fiber is klt and $\calD_b \in |-rK_{\calX_b}|$, the family $(\calX^\circ, \calD^\circ) \to B^\circ$, where $B^\circ = B - \{ 0\}$, is a stable family in $\cM^{\KSBA}_{n,v,r,1}$.  By properness of the KSBA moduli space, (up to base change) this extends to a stable family $(\calX', \calD') \to B$.  We may take the $c$-canonical model of this family, defined as 
     $(\calX'^{can}, c  \calD'^{can}) \to B$, where 
    \[ \calX'^{can}  = \Proj_B(\oplus_mf_*\calO(m(K_{\calX'} + c\calD')))\]
    and $\calD'^{can}$ is the pushforward of $\calD'$.  This produces a KSBA stable family that agrees with $(\calX, c\calD)$ away from $0 \in B$.  The authors prove that this canonical model is in fact the family of canonical models of the fibers.  Uniqueness of canonical models implies that $(\calX, c\calD)$ must in fact be $(\calX'^{can}, c  \calD'^{can})$. 

    Conversely, starting with any family $(\calX, \calD) \to B$ in $\cM^{\KSBA}_{n,v,r,1}$, we may take the $c$-canonical model to produce a family in $\cM^{\KSBA}_{n,v,r,c}$.  This correspondence demonstrates that the moduli spaces $\cM^{\KSBA}_{n,v,r,c}$ are themselves canonical models of the universal family over the projective moduli space of pairs $M^{\KSBA}_{n,v,r,1}$. The statement that there are only finitely many critical values of $c_i \in \bQ$ at which the moduli stacks/spaces change then follows from \cite[Corollary 1.1.5]{BCHM} applied to this family.
    
    To construct the required morphisms, the fundamental observation is, given a family of stable pairs $f:(\calX, (c_i + \epsilon) \calD) \to B$ over a base $B$ for some $\epsilon > 0$, the family $(\calX, c_i  \calD) \to B$ is a \textit{locally} KSBA stable family.  Each fiber has slc singularities (as this property is preserved by decreasing the coefficient), however the relative canonical divisor $K_{\calX/B} + c_i \calD$ may no longer be relatively ample over $B$.  To find the associated stable family, we simply take the canonical model $(\calX^{can}, c_i  \calD^{can}) \to B$.  This is now a stable family and is in fact a family of canonical models of the fibers. This will induce a morphism $(\cM^{\KSBA}_{n,V,r,c_i+\epsilon})^\nu\xrightarrow{}
 (\cM^{\KSBA}_{n,V,r,c_i})^\nu$.  

 For the other direction, the main observation is again that given a family of stable pairs $(\calX, (c_i - \epsilon) \calD) \to B$ where $c_i$ is a wall and $\epsilon \ll 1$, by \cite{KM98} $(\calX, c_i \calD) \to B$ is a locally stable family.  If $K_{\calX/B} + c_i \calD$ is not relatively ample over $B$, we can take its canonical model as above.  In this case, as $K_{\calX/B} + (c_i-\epsilon) \calD$ was relatively ample, this implies there is some component of $\calD$ on which $K_{\calX/B} + c_i \calD$ must be negative.  Using this negativity, the authors prove that, for any canonical model $(\calX^{can}, c_i  \calD^{can})$, there are only finitely many possible preimages $(\calX, c_i \calD)$ (i.e. the morphism $\cM^{\KSBA}_{n,V,r,c_i}\xleftarrow{}\cM^{\KSBA}_{n,V,r,c_i-\epsilon}$ is quasi-finite).  Furthermore, from the assumption that the general point $(X,cD)$ is klt and stable for all $c \in (r^{-1},1)$, it follows that these spaces are birational. By taking normalization and using Zariski's Main Theorem, we obtain the desired isomorphism $(\cM^{\KSBA}_{n,V,r,c_i})^\nu\xleftarrow{\cong}(\cM^{\KSBA}_{n,V,r,c_i-\epsilon})^\nu$.  Finally, the statement on the coarse moduli spaces follows from the morphisms on the moduli stacks and the Keel-Mori theorem. 
\end{proof}

\begin{remark}
    We cannot require that $D \in|-rK_X|$ on every fiber in this moduli space in general; see the following example for a construction of a pair where this is not satisfied.  By Remark \ref{rmk:fiberwiseDinKX}, we may assume this condition if $X$ is normal, and in fact, if $c = r^{-1} + \epsilon$ for $\epsilon \ll 1$ we may also assume this condition.  In \cite{hacking, devleming,KX19}, these `minimally' canonically polarized spaces are studied and it is shown that in fact the condition $r K_X + D \sim 0$ and $K_X, D$ are both $\bQ$-Cartier hold for all pairs $(X,cD)$ parametrized by this moduli stack and space. 
\end{remark}

This next example comes from \cite{DeVSingh}.

\begin{example}\label{ex:octic}
    Let $P_{8,c} \subset M^{\KSBA}_{2,9,\frac{8}{3},c}$ be the component of the moduli space whose general point represents $(\bP^2, c C)$ where $C \in |\calO(8)|$ is an octic plane curve and $c \in (\frac{3}{8},1)$.  By \cite{Orevkov}, there exists an octic plane curve $C_0$ with a single cuspidal singularity of the form $x^3 = y^{22}$.  The log canonical threshold of this curve is $\frac{1}{3}+\frac{1}{22} = \frac{25}{66} > \frac{3}{8}$.  Therefore, for any $c \in (\frac{3}{8}, \frac{25}{66}]$, $(\bP^2, c C_0)$ is a KSBA stable pair and a closed point in the moduli space $P_{8,c}$.  However, when $c > \frac{25}{66}$, this is unstable.  Taking a generic smoothing of the pair $(\bP^2, C_0)$ over $\bA^1$, i.e. a family $(\bP^2 \times \bA^1_t,\calC) \to \bA^1_t $, after base change we perform a $(22,3,1)$  blow up in the coordinates $(x,y,t)$ at the singular point of the central fiber.  The new central fiber of the family is the union of two surfaces: the $(22,3)$ weighted blow up $X$ of $\bP^2$, and the weighted projective exceptional divisor $Y =\bP(1,3,22)$.  These are glued along $E_X \subset X$, the exceptional divisor of the $(22,3)$ weighted blow up of $\bP^2$, which passes through two singular points of types $\frac{1}{3}(1,2)$ and $\frac{1}{22}(1,19)$, and $L \subset Y$ a section of $\calO_{Y}(1)$ which passes through the two singular points of $Y$.  The canonical divisor $K_{X \cup Y}$ restricts to $K_X + E_X$ on $X$ and $K_{Y} + L$, which has weighted degree $-25$ on $Y = \bP(1,3,22)$.
    
    The strict transform of the family of curves $\calC$ on the central fiber is a nodal union of two smooth curves $D_X + D_Y$, one on each component of the central fiber.  In $X$, by direct computation, $D_X$ is the strict transform of the singular octic curve and is a rational curve with self-intersection $-2$, and on $Y = \bP(1,3,22)$, $D_Y$ is a smooth curve of weighted degree 66. Because $\rho(X) = 2$ and $X$ has two rational curves with negative self-intersection ($E_X$ and $D_X$), these curves generate the Mori cone.  Therefore, we can compute directly using that $E_X^2 = - \frac{1}{66}$, $D_Y^2 = -2$, and $E_Y \cdot D_Y = 1$ that $K_{X} + E_X + c D_X$ is ample if and only if $c \in (\frac{25}{66}, \frac{1}{2})$, and $K_Y + L + cD_Y$ is ample if and only if $c > \frac{25}{66}$.  Letting $D = D_X + D_Y \subset X\cup Y$, this implies $(X\cup Y, c(D_X + D_Y))$ is a stable pair if and only if $c \in (\frac{25}{66},\frac{1}{2})$.  This therefore appears in the KSBA moduli space `replacing'\footnote{By `replacing', we mean that the log canonical model when $c < \frac{25}{66}$ of the pair $(X\cup Y, c(D_X + D_Y))$ is the pair $(\bP^2, cC_0)$.} the pair $(\bP^2, cC_0)$ when $c \in (\frac{25}{66}, \frac{1}{2})$.  
    
    At $c = \frac{1}{2}$, the divisor $D_X$ becomes $K_X+E_X + \frac{1}{2}D_X$ trivial, and taking the log canonical model when $c = \frac{1}{2}$ contracts $D_X$, resulting in the pair $(\overline{X} \cup Y, \frac{1}{2}D_Y)$ where $X \to \overline{X}$ is the contraction of the $-2$ curve $D_X$.  Because $K_{X \cup Y} \cdot D_X = (K_X + E_X) \cdot D_X \ne 0$, the canonical divisor on the pair $(\overline{X} \cup Y, \frac{1}{2}D_Y)$ is \textit{not} $\bQ$-Cartier (and neither is $D_Y$), although $K_{\overline{X} \cup Y} + \frac{1}{2} D_Y$ is $\bQ$-Cartier.  The pair $(\overline{X} \cup Y, \frac{1}{2}D_Y)$ is slc and $K_{\overline{X} \cup Y} + \frac{1}{2} D_Y$ is ample, and therefore this is a KSBA stable pair.  However, as the divisors $K_{\overline{X} \cup Y}$ and $D_Y$ are not individually $\bQ$-Cartier, $(\overline{X} \cup Y, cD_Y)$ is not slc for any $c \ne \frac{1}{2}$.  

    To find the associated KSBA stable pair when $c > \frac{1}{2}$, we must perform the flip of the curve $D_X$ in the family constructed when  $c \in (\frac{25}{66}, \frac{1}{2})$.  By direct computation, the new family has central fiber equal to $\overline{X} \cup \tilde{Y}$, where $\tilde{Y}$ is the $(1,2)$-weighted blow-up of the intersection point $L \cap D_Y$ on $Y$.  The strict transform of the given family of curves on $\overline{X} \cup \tilde{Y}$ is just $\tilde{D}_Y$, the strict transform of $D_Y$ in $Y$.  We claim that $(\overline{X} \cup \tilde{Y}, c \tilde{D}_Y)$ is a KSBA stable family for any $c \in (\frac{1}{2},1]$.  It is clear that this has slc singularities, so it suffices to check the ampleness of the log canonical divisor.  The relevant computation is to prove that $K_{\overline{X} \cup \tilde{Y}}$ is positive on $\overline{E}_X$, the image of $E_X$ in the contraction $X \to \overline{X}$.  Because $\overline{E}_X$ passes through three singularities of types $\frac{1}{3}(1,2)$, $\frac{1}{22}(1,19)$, and $\frac{1}{2}(1,1)$, and therefore we compute 
    \[ K_{\overline{X} \cup \tilde{Y}} \cdot \overline{E}_X = (K_{\overline{X}} + \overline{E}_X) \cdot \overline{E}_X = -2 + 1 - \tfrac{1}{3} + 1 - \tfrac{1}{22} + 1 - \tfrac{1}{2} = \tfrac{4}{33} > 0.\]

\begin{figure}[h]
\begin{tikzcd}[row sep=huge]
\hspace{1in} {\resizebox{.4\textwidth}{!}{\begin{tabular}{c}
\begin{tikzpicture}[gren0/.style = {draw, circle,fill=greener!80,scale=.7},gren/.style ={draw, circle, fill=greener!80,scale=.4},blk/.style ={draw, circle, fill=black!,scale=.2},plc/.style ={draw, circle, color=white!100,fill=white!100,scale=0.02},smt/.style ={draw, circle, color=gray!100,fill=gray!100,scale=0.02},lbl/.style ={scale=.2}] 

\node[smt] at (-2.5+0,0) (1){};
\node[smt] at (-2.5+0,2) (2){};
\node[smt] at (-2.5+1.2,2.7) (3){};
\node[smt] at (-2.5+1.2,0.7) (4){};
\draw [black] plot [smooth cycle, tension=.1] coordinates { (1) (2) (3) (4) };

\draw [thick,blue] plot [smooth, tension=1] coordinates {(-2.5+.4,2) (-2.5+.2,1.4) (-2.5+.8,2) (-2.5+.3,0.8) (-2.5+1.1,1.3) (-2.5+.65,0.5)};
\node[below right,node font=\tiny,text=blue] at (-2.5+0.7,0.8) {smooth octic};

\node[smt] at (0,0) (5){};
\node[smt] at (0,2) (6){};
\node[smt] at (1.2,2.7) (7){};
\node[smt] at (1.2,0.7) (8){};
\draw [black] plot [smooth, tension=.1] coordinates { (8) (5) (6) (7) };

\node[smt] at (2.5,2.5) (11){};
\node[smt] at (2.5,1) (12){};
\draw [thick,violet] plot [smooth cycle, tension=.1] coordinates { (8) (12) (11) (7) };
\node[node font=\tiny,color=violet] at (3,2.4) {exceptional $Y \cong \bP(1,3,22)$};
\filldraw[color=teal, fill=teal](1.18,2.2) circle (0.05);
\filldraw[color=teal, fill=teal](1.18,1.2) circle (0.05);
\draw[teal,->] (2,.55) -- (1.25,2.1);
\draw[teal,->] (2,.55) -- (1.25,1.1);
\node[below right,node font=\tiny,color=teal] at (2,.55) {singular points from};
\node[below right,node font=\tiny,color=teal] at (2,.25) {weighted blow-up};

\draw [thick,blue] plot [smooth, tension=1] coordinates { (0.2,1.6)  (.8,1.9) (1.18,1.7) };
\draw [thick,blue] plot [smooth, tension=1] coordinates {  (1.18,1.7) (1.8,1.3) (2.3,1.6) };
\draw[blue,->] (.6,1.5) -- (.8,1.7);
\node[node font=\tiny,color=blue] at (.4,1.4) {$D_X$};
\node[node font=\tiny,color=blue] at (.3,1.1) {$(D_X)^2= -2$};
\draw[blue,->] (2.6,1.4) -- (2.35,1.5);
\node[right,node font=\tiny,color=blue] at (2.6,1.4) {$D_Y$};

\draw[thick,color=gray] (-2.5-0.5,-.5) to (1.8+0.5,-.5);
\node[blk] at (-2.5+0.6,-.5) (9){};
\node[below right, node font=\tiny] at (9) {$t$};
\node[blk] at (1,-.5) (10){};
\node[below right, node font=\tiny] at (10) {$0$};
\node[right, node font=\tiny] at (1.8+0.5,-.5) {$\bA^1_t$};

\node[below right,node font=\tiny] at (-2.5,0.1) {$\bP^2_t$};
\node[above right, node font=\tiny] at (0,0.2) {$X$};
\node[node font=\huge, text=gray] at (-2.5+1.8,1.5) {$\rightsquigarrow$};
\end{tikzpicture}
\end{tabular}}} \arrow[d,swap,start anchor={[xshift=-2ex]},end anchor={[xshift=-8ex]},"\text{weighted blow up}"] \arrow[rd, end anchor={[xshift=-10ex,yshift=2ex]}, "\text{small contraction}"] \arrow[rr, dashed, start anchor={[xshift=-5ex]}, end anchor={[xshift=-23ex]}, "\text{flip } D_X"]  && \hspace{-1.5in} {\resizebox{.4\textwidth}{!}{\begin{tabular}{c}
\begin{tikzpicture}[gren0/.style = {draw, circle,fill=greener!80,scale=.7},gren/.style ={draw, circle, fill=greener!80,scale=.4},blk/.style ={draw, circle, fill=black!,scale=.2},plc/.style ={draw, circle, color=white!100,fill=white!100,scale=0.02},smt/.style ={draw, circle, color=gray!100,fill=gray!100,scale=0.02},lbl/.style ={scale=.2}] 

\node[smt] at (-2.5+0,0) (1){};
\node[smt] at (-2.5+0,2) (2){};
\node[smt] at (-2.5+1.2,2.7) (3){};
\node[smt] at (-2.5+1.2,0.7) (4){};
\draw [black] plot [smooth cycle, tension=.1] coordinates { (1) (2) (3) (4) };

\draw [thick,blue] plot [smooth, tension=1] coordinates {(-2.5+.4,2) (-2.5+.2,1.4) (-2.5+.8,2) (-2.5+.3,0.8) (-2.5+1.1,1.3) (-2.5+.65,0.5)};
\node[below right,node font=\tiny,text=blue] at (-2.5+0.7,0.8) {smooth octic};

\node[smt] at (0,0) (5){};
\node[smt] at (0,2) (6){};
\node[smt] at (1.2,2.7) (7){};
\node[smt] at (1.2,0.7) (8){};
\draw [black] plot [smooth, tension=.1] coordinates { (8) (5) (6) (7) };

\node[smt] at (2.5,2.5) (11){};
\node[smt] at (2.5,1) (12){};
\draw [thick,violet] plot [smooth cycle, tension=.1] coordinates { (8) (12) (11) (7) };
\node[node font=\tiny,color=violet] at (2.8,2.6) { $\tilde{Y}$};
\filldraw[color=teal, fill=teal](1.18,2.2) circle (0.05);
\filldraw[color=teal, fill=teal](1.18,1.2) circle (0.05);
\filldraw[color=blue, fill=blue](1.18,1.7) circle (0.05);

\draw [thick,dotted,blue] plot [smooth, tension=1] coordinates {  (1.18,1.7) (1.8,1.8) (2.3,1.7) };
\draw [thick,blue] plot [smooth, tension=1] coordinates {  (1.7,2.2) (1.9,1.9) (1.9,1.3) (2.1,1) };

\draw[blue,->] (2.6,1.9) -- (2.35,1.75);
\node[right,node font=\tiny,color=blue] at (2.6,2.1) {$D_X^+$};
\node[right,node font=\tiny,color=blue] at (2.6,1.8) {flip of $D_X$};

\draw[blue,->] (2.6,1.1) -- (2.15,1.2);
\node[right,node font=\tiny,color=blue] at (2.6,1.1) {$\tilde{D}_Y$};
\node[right,node font=\tiny,color=blue] at (2,0.8) {strict transform of $D$};

\draw[thick,color=gray] (-2.5-0.5,-.5) to (1.8+0.5,-.5);
\node[blk] at (-2.5+0.6,-.5) (9){};
\node[below right, node font=\tiny] at (9) {$t$};
\node[blk] at (1,-.5) (10){};
\node[below right, node font=\tiny] at (10) {$0$};
\node[right, node font=\tiny] at (1.8+0.5,-.5) {$\bA^1_t$};


\node[below right,node font=\tiny] at (-2.5,0.1) {$\bP^2_t$};
\node[above right, node font=\tiny] at (0,0.2) {$\overline{X}$};
\node[node font=\huge, text=gray] at (-2.5+1.8,1.5) {$\rightsquigarrow$};
\end{tikzpicture}
\end{tabular}}} \arrow[ld,end anchor={[xshift=-2ex,yshift=-2ex]}, "\text{small contraction}"]  \\
\hspace{-1in} {\resizebox{.4\textwidth}{!}{\begin{tabular}{c}
\begin{tikzpicture}[gren0/.style = {draw, circle,fill=greener!80,scale=.7},gren/.style ={draw, circle, fill=greener!80,scale=.4},blk/.style ={draw, circle, fill=black!,scale=.2},plc/.style ={draw, circle, color=white!100,fill=white!100,scale=0.02},smt/.style ={draw, circle, color=gray!100,fill=gray!100,scale=0.02},lbl/.style ={scale=.2}] 

\node[smt] at (-2.5+0,0) (1){};
\node[smt] at (-2.5+0,2) (2){};
\node[smt] at (-2.5+1.2,2.7) (3){};
\node[smt] at (-2.5+1.2,0.7) (4){};
\draw [black] plot [smooth cycle, tension=.1] coordinates { (1) (2) (3) (4) };

\draw [thick,blue] plot [smooth, tension=1] coordinates {(-2.5+.4,2) (-2.5+.2,1.4) (-2.5+.8,2) (-2.5+.3,0.8) (-2.5+1.1,1.3) (-2.5+.65,0.5)};
\node[below right,node font=\tiny,text=blue] at (-2.5+0.7,0.8) {smooth octic};

\node[smt] at (0,0) (5){};
\node[smt] at (0,2) (6){};
\node[smt] at (1.2,2.7) (7){};
\node[smt] at (1.2,0.7) (8){};
\draw [black] plot [smooth cycle, tension=.1] coordinates { (5) (6) (7) (8) };

\draw [thick,blue] plot [smooth, tension=1] coordinates { (0.2,1.6) (0.5,1.6) (0.4,.8) };
\draw [thick,blue] plot [smooth, tension=1] coordinates {  (0.4,.8) (.7,1.1) (1,1) };
\draw[blue,->] (.9,.3) -- (.45,.75);
\filldraw[color=violet, fill=violet](0.4,0.8) circle (0.02);
\node[right,color=blue,node font=\tiny] at (.9,.3) {$C_0$: octic curve with};
\node[right,color=blue,node font=\tiny] at (.9,0) {cusp singularity $y^3 +x^{22}$};

\draw[thick,color=gray] (-2.5-0.5,-.5) to (1.2+0.5,-.5);
\node[blk] at (-2.5+0.6,-.5) (9){};
\node[below right, node font=\tiny] at (9) {$t$};
\node[blk] at (0.6,-.5) (10){};
\node[below right, node font=\tiny] at (10) {$0$};
\node[right, node font=\tiny] at (1.2+0.5,-.5) {$\bA^1_t$};

\node[below right,node font=\tiny] at (-2.5,0.1) {$\bP^2_t$};
\node[below right, node font=\tiny] at (7) {$X = \bP^2_0$};
\node[node font=\huge, text=gray] at (-2.5+1.8,1.5) {$\rightsquigarrow$};
\end{tikzpicture}
\end{tabular}}} & [-2em] \hspace{-1in} {\resizebox{.35\textwidth}{!}{\begin{tabular}{c}
\begin{tikzpicture}[gren0/.style = {draw, circle,fill=greener!80,scale=.7},gren/.style ={draw, circle, fill=greener!80,scale=.4},blk/.style ={draw, circle, fill=black!,scale=.2},plc/.style ={draw, circle, color=white!100,fill=white!100,scale=0.02},smt/.style ={draw, circle, color=gray!100,fill=gray!100,scale=0.02},lbl/.style ={scale=.2}] 

\node[smt] at (-2.5+0,0) (1){};
\node[smt] at (-2.5+0,2) (2){};
\node[smt] at (-2.5+1.2,2.7) (3){};
\node[smt] at (-2.5+1.2,0.7) (4){};
\draw [black] plot [smooth cycle, tension=.1] coordinates { (1) (2) (3) (4) };

\draw [thick,blue] plot [smooth, tension=1] coordinates {(-2.5+.4,2) (-2.5+.2,1.4) (-2.5+.8,2) (-2.5+.3,0.8) (-2.5+1.1,1.3) (-2.5+.65,0.5)};
\node[below right,node font=\tiny,text=blue] at (-2.5+0.7,0.8) {smooth octic};

\node[smt] at (0,0) (5){};
\node[smt] at (0,2) (6){};
\node[smt] at (1.2,2.7) (7){};
\node[smt] at (1.2,0.7) (8){};
\draw [black] plot [smooth, tension=.1] coordinates { (8) (5) (6) (7) };

\node[smt] at (2.5,2.5) (11){};
\node[smt] at (2.5,1) (12){};
\draw [thick,violet] plot [smooth cycle, tension=.1] coordinates { (8) (12) (11) (7) };
\filldraw[color=teal, fill=teal](1.18,2.2) circle (0.05);
\filldraw[color=teal, fill=teal](1.18,1.2) circle (0.05);

\draw [thick,blue] plot [smooth, tension=1] coordinates {  (1.18,1.7) (1.8,1.3) (2.3,1.6) };
\draw[blue,->] (.6,1.9) -- (1,1.7);
\node[node font=\tiny,color=blue] at (.15,2.4) {singular point};
\node[node font=\tiny,color=blue] at (.2,2.1) {image of $D_X$};
\draw[blue,->] (2.6,1.4) -- (2.35,1.5);
\node[right,node font=\tiny,color=blue] at (2.6,1.4) {$D_Y$};
\filldraw[color=blue, fill=blue](1.18,1.7) circle (0.05);


\draw[thick,color=gray] (-2.5-0.5,-.5) to (1.8+0.5,-.5);
\node[blk] at (-2.5+0.6,-.5) (9){};
\node[below right, node font=\tiny] at (9) {$t$};
\node[blk] at (1,-.5) (10){};
\node[below right, node font=\tiny] at (10) {$0$};
\node[right, node font=\tiny] at (1.8+0.5,-.5) {$\bA^1_t$};


\node[below right,node font=\tiny] at (-2.5,0.1) {$\bP^2_t$};
\node[above right, node font=\tiny] at (0,0.2) {$\overline{X}$};
\node[node font=\huge, text=gray] at (-2.5+1.8,1.5) {$\rightsquigarrow$};
\end{tikzpicture}
\end{tabular}}} &  \\
\end{tikzcd}
\vspace{-.5in}
\caption{Replacement of the rational octic curve with cusp $y^3 = x^{22}$.}
\label{f:octics}
\end{figure}
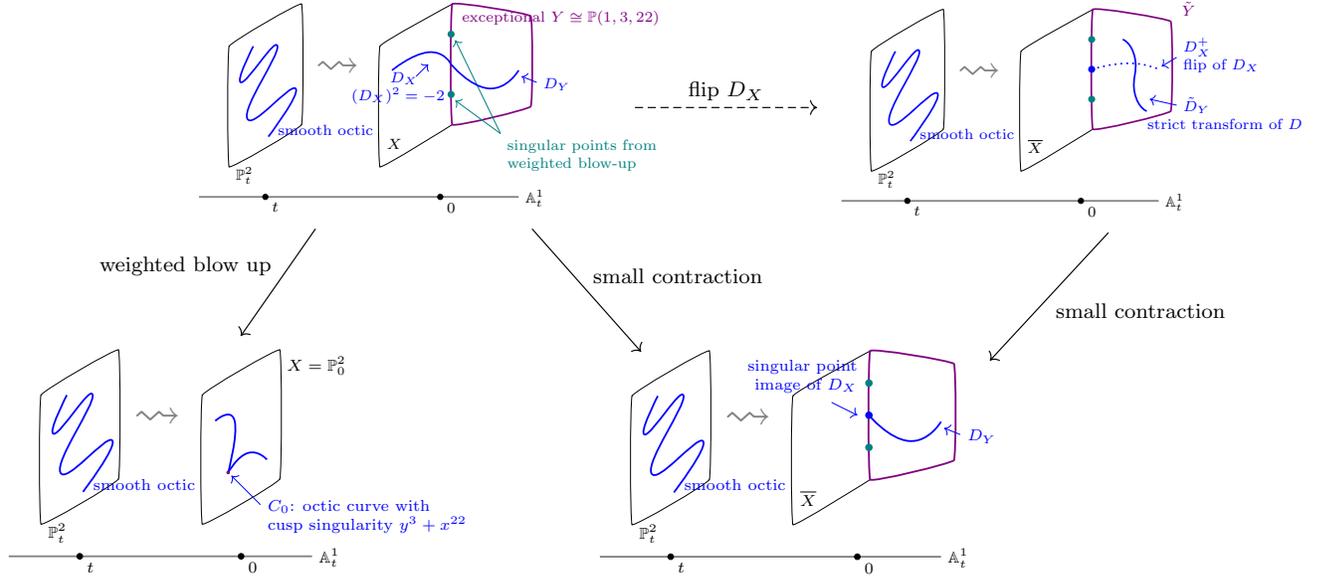

        The wall crossing on this family is summarized in the Figure \ref{f:octics}. In summary, we have produced stable limits $(X, cD)$ of pairs $(\bP^2, cC)$ where $C$ is an octic curve such that, for $c = \frac{1}{2}$, $K_X$ and $D$ are not individually $\bQ$-Cartier, and for $c \ge \frac{1}{2}$, $X$ has two components but the curve $D$ only has one component and is fully contained in only one piece of the surface.  In particular, $rK_X + D \not \sim 0$ and $D \notin |-rK_X|$ despite this relationship being satisfied for general fibers.  
\end{example}

\subsection{Moduli of boundary polarized Calabi Yau pairs}\label{ss:bpcy}

The previous sections develop the theory of wall crossing for the moduli space compactifying the locus of klt pairs $(X,cD)$ with $D \in |- rK_X|$ when $c < \min\{1, r^{-1}\}$ or $c > r^{-1}$.  It is then natural to ask for a moduli space of pairs $(X,cD)$ when $c = r^{-1}$.  In this case, the given pair is log Calabi Yau.  However, there exists a natural polarization as either the divisor $-K_X$ or the divisor $D$ is ample on each pair.  Using ideas from both K-stability and KSBA stability, in \cite{BABWILD}, a moduli stack of \textit{boundary polarized Calabi Yau} pairs is constructed.  Furthermore, this stack fits naturally between the KSBA- and K-moduli stacks.  The stack is not of finite type in general and thus cannot admit a good moduli space, but in \cite{BABWILD, BL}, it is shown that if $\dim X = 2$, this stack admits an \textit{asymptotically good moduli space} (essentially, a stabilization of good moduli spaces) completing the diagram: 

 \[
 \cM^\KSBA_{n,v,r,r^{-1}+\epsilon}\xhookrightarrow{}
 \cM^{\rm CY}_{n,v,r}\xhookleftarrow{}\cM^\K_{n,v,r,r^{-1}-\epsilon}
 \]
 with induced projective morphisms of asymptotically good moduli spaces
 \[
 M^\KSBA_{n,v,r,r^{-1}+\epsilon}\xrightarrow{}
 M^{\rm CY}_{n,v,r}\xleftarrow{}M^\K_{n,v,r,r^{-1}-\epsilon}.
 \]

Precisely, the objects parametrized by this stack are \textit{boundary polarized Calabi Yau pairs}. In \cite{BABWILD}, the coefficient $c$ below was absorbed into the divisor $D$, but for compatibility with the framework introduced in the previous section, we consider pairs $(X,cD)$.

\begin{defn}\label{d:bpcy}
For $c \in \bQ^{\ge 0}$, a \emph{boundary polarized CY pair} is a projective slc pair $(X,cD)$ such that 
\begin{enumerate}
	\item $D$ is an effective $\bQ$-divisor,
	\item $K_{X}+cD\sim 0$, and
	\item $D$ is $\mathbb{Q}$-Cartier and ample.
\end{enumerate}
Note that (3) is equivalent to the condition that $X$ is Fano. 
\end{defn}

In \cite{BABWILD}, it is shown, for fixed rational $c = r^{-1}$, $n = \dim X$, and $v = (-K_X)^n$, there exists an Artin stack $\cM^{\rm CY}_{n,v,r}$, locally of finite type, parametrizing all boundary polarized Calabi Yau pairs with these invariants.  Note that, if $(X,cD)$ is a klt boundary polarized Calabi Yau pair, then $(X,(c+\epsilon)D)$ is KSBA stable for $\epsilon \ll 1$ because it is log canonical and $K_X + (c+\epsilon )D \sim e D$ is ample.  Similarly, $(X,(c-\epsilon)D)$ is K-semistable for $\epsilon \ll 1$ by \cite[Lemma 5.3]{Zhou}.  Conversely, if $(X, (c\pm\epsilon)D)$ is KSBA or K-semistable for all $\epsilon \ll 1$, then $(X,cD)$ is a boundary polarized Calabi Yau pair.  This gives open immersions of stacks 
\[ \cM^K_{n,v,r,r^{-1}-\epsilon} \hookrightarrow \cM^{\rm CY}_{n,v,r} \hookleftarrow \cM^\KSBA_{n,v,r,r^{-1}+\epsilon}.\]

As part of the K-moduli and KSBA moduli theorems, the stacks on the left and right admit good and coarse moduli spaces, respectively.  It is natural to ask if there exists a good moduli space for $\cM^{\rm CY}_{n,v,r}$ completing the diagram.  Ultimately, the answer is no: 

\begin{example}
    Let $n = 2, v = 9, r = 1$ and consider the moduli spaces compactifying the locus of pairs $(\bP^2, cC)$ where $C$ is a cubic plane curve.  The stack $\cM^{\rm CY}_{2,9,1}$ parametrizes all smoothable slc pairs $(X,D)$ with $(-K_X)^2 = 9$ such that $K_X + D \sim 0$ and $D, -K_X$ are both $\bQ$-Cartier and ample.  There are infinitely many toric pairs satisfying this condition: $(\bP(a^2, b^2,c^2), (xyz = 0))$ (c.f. Theorem \ref{t:kltdegens}).  The pair $(\bP^2, (xyz=0))$ isotrivially specializes to all of these, so all such pairs are contained in the closure of the point $(\bP^2, (xyz = 0))$.  Furthermore, there is no closed point $[(X,D)] \in \cM^{\rm CY}_{2,9,1}$ that all such pairs specialize to: the Cartier index of $\bP(a^2,b^2,c^2)$ is unbounded.  A good moduli space must identify all of these points as their closures intersect, but as there is no closed point in this equivalence class, this does not lead to a good moduli space.
\end{example}

To remedy this situation, in dimension $2$, we define an \textit{asymptotically good moduli space}.  Stratifying the moduli space $\cM^{\rm CY}_{n,v,r}$ by substacks 
\[ \calU_1 \subset \calU_2 \subset \calU_3 \dots \]
where each $\calU_n$ represents the locus of pairs $(X,cD)$ for which $nK_X$ is Cartier, we prove that for $n \gg 0 $, each $\calU_n$ admits a good moduli space $U_n$ and these spaces $U_n$ are isomorphic.  The asymptotically good moduli space $M^{\rm CY}_{2,v,r}$ is then defined to be this stabilization.  With this definition, there is a diagram 

\begin{center}
\begin{tikzcd}
    \cM^K_{2,v,r,r^{-1}-\epsilon}  \arrow[r] \arrow[d] & \cM^{\rm CY}_{2,v,r}   \arrow[d] & \cM^\KSBA_{2,v,r,r^{-1}+\epsilon} \arrow[d] \arrow[l] \\
    M^K_{2,v,r,r^{-1}-\epsilon}  \arrow[r]  & M^{\rm CY}_{2,v,r}  & M^\KSBA_{2,v,r,r^{-1}+\epsilon} \arrow[l]
\end{tikzcd}
\end{center}
where the upper arrows are open immersions of stacks and the lower arrows are projective morphisms of varieties.  An essential technical piece needed is that the stacks $\cM^{\rm CY}_{2,v,r}$ and $\cU_n$ are $S$-complete and $\Theta$-reductive, which are valuative criterion over punctured surfaces.  By \cite{AHLH18}, these are necessary and sufficient for finite type stacks with affine diagonal to have a separated good moduli space.

This provides a complete picture of wall crossing when $\dim X = 2$.

\section{Wall crossing for quartic plane curves}\label{sec:quartics}

The aim of the next few sections is to study wall crossing in the specific case of degree $d$ plane curves.  We use the quartic case to describe a \textit{complete} example, highlighting essential techniques throughout the example. 

\noindent \textbf{Notation.}  Fix $n = 2$, $v = 9$ and $r = \frac{d}{3}$ for some integer $d > 3$\footnote{When $d \le 3$, these spaces are described in \cite{ADLpub, BABWILD}.}.  For $c \in (0, \frac{3}{d})$, define $P^K_{d,c} \subset M^K_{2,9, \frac{d}{3},c}$ to be the irreducible component of the K-moduli space such that the general point is $(\bP^2, cC)$ where $C$ is a degree $d$ plane curve.  When $c = \frac{3}{d}$, define $P^{\rm CY}_{d} \subset M^{\rm CY}_{2,9, \frac{d}{3}}$ to be the irreducible component of the boundary polarized Calabi Yau moduli space such that the general point is $(\bP^2, cC)$ where $C$ is a degree $d$ plane curve.  Finally, when $c \in (\frac{3}{d}, 1]$, define $P^\KSBA_{d,c} \subset M^\KSBA_{2,9, \frac{d}{3},c}$ to be the irreducible component of the KSBA space such that the general point is $(\bP^2, cC)$ where $C$ is a degree $d$ plane curve.  We will study the wall crossing on these components as $c$ varies and may omit the superscript $\K, \KSBA, \rm CY$ as this is clear from choice of $c$.  

\subsection{Toolbox}

We first present key results that are used repeatedly in explicit computations. 

When constructing moduli spaces of pairs $(X,D)$, it is useful to have an idea of potential degenerations of $X$ that may appear in the moduli problem.  If $X = \bP^2$, we have a full understanding of the normal, log canonical degenerations, stated here in the klt case: 

\begin{theorem}[\cite{hp}]\label{t:kltdegens}
    Let $X$ be a klt degeneration of $\bP^2$, Then $X$ is isomorphic to one of the following surfaces: 
        \begin{enumerate}
            \item a weighted projective space $\bP(a^2,b^2,c^2)$ such that $a^2 + b^2 + c^2 = 3abc$, or 
            \item a partial smoothing of a surface in (1). 
        \end{enumerate}
\end{theorem}

It is also useful, if possible, to control the allowed singularities of pairs $(X, cD)$ appearing in the moduli problem in terms of the coefficient $c$. In K-moduli, this can be done using another invariant called the normalized volume. 

\begin{defn}
Let $(X,D)$ be a klt log pair of dimension $n$. Let $x\in X$ be a closed point. A \emph{valuation $v$ on $X$ centered at $x$} is a valuation of $\bC(X)$ such that $v\geq 0$ on $\cO_{X,x}$ and $v>0$ on $\fm_x$. The set of such valuations is denoted by $\Val_{X,x}$.
The \emph{volume} is a function $\vol_{X,x}:\Val_{X,x}\to \bR_{\geq 0}$ defined as 
\[
\vol_{X,x}(v):=\lim_{k\to\infty}\frac{\dim_{\bC}\cO_{X,x}/\{f\in\cO_{X,x}\mid v(f)\geq k\}}{k^n/n!}.
\]

The \emph{log discrepancy}  is a function $A_{(X,D)}:\Val_{X,x}\to \bR_{>0}\cup\{+\infty\}$ defined in \cite{JM}. Note that if $v=a\cdot\ord_E$ where $a\in\bR_{>0}$ and $E$ is a prime divisor over $X$ centered at $x$, then 
\[
A_{(X,D)}(v)=a(1+\ord_E(K_{Y}-\pi^*(K_X+D))),
\]
where $\pi:Y\to X$ is a birational model of $X$ containing $E$ as a divisor.

The \emph{normalized volume} defined in \cite{Li18} is a function $\hvol_{(X,D),x}:\Val_{X,x}\to \bR_{>0}\cup\{+\infty\}$ defined as
\[
\hvol_{(X,D),x}(v):=\begin{cases}
A_{(X,D)}(v)^n\cdot\vol_{X,x}(v) & \textrm{ if } A_{(X,D)}(v)<+\infty\\
+\infty & \textrm{ if }A_{(X,D)}(v)=+\infty
\end{cases}
\]

The \emph{local volume} of a klt singularity $x\in (X,D)$ is defined as 
\[
\hvol(x,X,D):=\min_{v\in\Val_{X,x}}\hvol_{(X,D),x}(v).
\]
Note that the existence of a $\hvol$-minimizer is proven in \cite{Blu18}.
\end{defn}

The following theorem from \cite{LL19} generalizing \cite[Theorem 1.1]{Fuj15} and \cite[Theorem 1.2]{Liu18}, known as the \textit{Local-to-Global principal} can be a key tool in explicit computations. 

\begin{thm}[{\cite[Proposition 4.6]{LL19}}]\label{thm:local-vol-global}
Let $(X,D)$ be a K-semistable log Fano pair of dimension $n$. Then for any closed point $x\in X$, we have 
\[
(-K_X-D)^n\leq \left(1+\frac{1}{n}\right)^n\hvol(x,X,D).
\]
\end{thm}

\begin{property}\label{propsofvol}
We list several important properties of the normalized volume.
\begin{enumerate}
    \item Assuming $K_X$ is $\bQ$-Cartier (which holds in the K-moduli spaces of pairs in this survey by Remark \ref{rmk:fiberwiseDinKX}), by definition we have $A_X(v) \ge A_{X,D}(v)$ so $\widehat{\vol}(x,X) \ge \widehat{\vol}(x,X,D)$.
    \item \cite{dFEM,Li18} If $X$ has dimension $n$ and $x \in X$ is smooth, then 
    \[ \widehat{\vol}(x,X) = n^n \]
    \item \cite{LXcubics} If $X$ has dimension $n$, then for any $x \in X$, \[ \widehat{\vol}(x,X) \le n^n \] and equality holds if and only if $x$ is smooth.  
    \item \cite{Liu18} If $x \in X = (0 \in \mathbb{A}^n/G)$ is a quotient singularity where $G \subset \GL_n(\bC)$ acts freely in codimension 1, then \[\widehat{\vol}(x,X) = \frac{n^n}{|G|}.\] If $x \in X$ is a quotient of an arbitrary variety by $G \subset \GL_n(\bC)$, then \[\widehat{\vol}(x,X) \le \frac{n^n}{|G|}.\] Combining this with Theorem \ref{thm:local-vol-global} and (1) gives: if $X$ is a K-semistable log Fano pair, then for any quotient singularity $x \in X = (0 \in \mathbb{A}^n/G)$, 
    \[ (-K_X)^n \le \frac{(n+1)^n}{|G|}. \]
    \item Combining the previous points, the normalized volume gives an \textit{index bound} on the canonical divisor in a K-semistable pair $(X,cD)$ assuming $K_X$ is $\bQ$-Cartier.  Let $m$ be the minimal positive integer such that $mK_X$ is Cartier at $x \in X$.  Let $Y$ be the canonical cyclic cover, so $Y \to X$ is a quotient by $\mu_m$.  Then, 
    \[ (-K_X-cD)^n \le \left( 1 + \frac{1}{n}\right)^n \widehat{\vol}(x,X,D) \le  \left( 1 + \frac{1}{n}\right)^n \widehat{\vol}(x,X) \le   \left( 1 + \frac{1}{n}\right)^n \frac{n^n}{m} = \frac{(n+1)^n}{m}. \]
    For example, if $(-K_X - cD)^n > \frac{(n+1)^n}{2}$, this implies $-K_X$ must be Cartier.  In particular, if $n = 2$ and $D = 0$, this says any K-semistable Fano surface with $(-K_X)^2 > \frac{9}{2}$ must be Gorenstein and hence have only canonical singularities. 
    \item \cite{LXcubics} \textit{Gap Conjecture:} if $x\in X$ is \textit{not} a smooth point, then 
    \[ \widehat{\vol}(x,X) \le 2(n-1)^n \] and equality holds if and only if $x$ is an ordinary double point. This is known if $X$ has dimension $n = 2 $ or $n = 3$.
    \item Suppose $\dim X = 2$ or $3$ and $(X,cD)$ is K-semistable.  Then, $(X,cD)$ is klt.  Assume $K_X$ is $\bQ$-Cartier (which holds if $\dim X = 2$, see \cite[Chapter 4]{KM98}).  If $x \in X$ is not a smooth point, then (6) implies
    \[ (-K_X-cD)^n \le \left( 1 + \frac{1}{n}\right)^n \widehat{\vol}(x,X,D) \le  \left( 1 + \frac{1}{n}\right)^n \widehat{\vol}(x,X) \le   \left( 1 + \frac{1}{n}\right)^n 2(n-1)^n = 2 \left(n - \frac{1}{n} \right)^n. \]
    When $n = 2$, this says $(-K_X - cD)^2 \le \frac{9}{2}$ and when $n = 3$ this says $(-K_X - cD)^3 \le \frac{1024}{27} < 38$.  In particular, if $X$ is a K-semistable Fano surface or threefold with $\vol(X) \ge 5$ or $\vol(X) \ge 38$, it must be smooth.
\end{enumerate}
\end{property}

In the proportional case, we also have a key result known as \textit{interpolation} allowing us to describe K-stability of pairs $(X,cD)$ for a range of coefficients $c$. 

\begin{prop}{\cite[Proposition 2.13]{ADLpub}}\label{prop:k-interpolation}
Let $X$ be a $\bQ$-Fano variety. Let $D_1$ and $D_2$ be effective $\bQ$-divisors on $X$ satisfying the following properties:
\begin{itemize}
    \item Both $D_1$ and $D_2$ are rational multiples of $-K_X$ under $\bQ$-linear equivalence.
    \item $-K_X-D_1$ is ample, and $-K_X-D_2$ is nef.
    \item The log pairs $(X,D_1)$ and $(X,D_2)$ are K-(poly/semi)stable and K-semistable, respectively.
\end{itemize}
Then we have
\begin{enumerate}
    \item If $D_1\neq 0$, then $(X,tD_1+(1-t)D_2)$ is K-(poly/semi)stable for any $t\in (0,1]$.
    \item If $D_1=0$, then $(X,(1-t)D_2)$ is K-semistable
    for any $t\in (0,1]$.
\end{enumerate}
\end{prop}

In the case that $D \in |-rK_X|$ explored in this survey, we will use this to determine K-stability of $(X,r^{-1}D)$.  While not mentioned in here, one can define the notion of K-stability more generally for any polarized pair, and for log Calabi-Yau pairs, the K-stability is straightforward to determine.  By \cite{Odaka13}, a pair $(X,cD)$ with $-K_X - cD \sim 0$ is K-semistable if and only if it is slc.  Therefore, if a pair $(X, r^{-1}D)$ in our setting is slc, it is K-semistable.  As $-K_X - r^{-1} D \sim 0$ is nef, if for example $X$ is K-semistable, we may conclude by (2) above that $(X, cD)$ is K-semistable for any $c \in [0, r^{-1})$.  We state a precise version below.

\begin{cor}
    Suppose $(X,cD)$ is a log Fano pair such that $D \sim -rK_X$. If $(X,c_0D)$ is K-(poly/semi)stable for some $c_0 \le r^{-1}$ and the log canonical threshold $\lct(X,D) \ge r^{-1}$, then $(X,cD)$ is K-(poly/semi)stable for any $c \in (c_0, r^{-1})$.  
\end{cor}

\begin{proof}
    Apply Proposition \ref{prop:k-interpolation} for $D_1 = c_0 D$ and $D_2 = r^{-1}D$, using \cite{Odaka13} to conclude $(X,D_2)$ is K-semistable.  
\end{proof}

To produce pairs $(X,cD)$ that are K-semistable for small coefficient, we relate K-stability to Geometric Invariant Theory. We present the case just for $X = \bP^n$, but this holds more generally, see \cite[Theorem 2.22]{ADLpub} or \cite[Theorem 1.11]{LZwalls}. 

\begin{theorem}[\cite{ADLpub}]\label{thm:k=git}
    For $c \ll 1$, the K-moduli stack (space) parameterizing K-semi(poly)stable limits of pairs $(\bP^n, cD)$, where $D$ is a degree $d$ hypersurface, is isomorphic to the GIT moduli stack (space).  
\end{theorem}

The idea of the proof comes in two parts: first, one-parameter subgroups in GIT induce test configurations in K-stability, and in this setting the Futaki invariant is proportional to the Hilbert Mumford weight.  This will prove that the K-stability implies the GIT stability.  The converse relies on K-stability of the generic pair of this form, properness of the K- and GIT-moduli spaces, and the fact that $\bP^n$ itself is K-polystable so admits no nontrivial K-semistable degenerations.

\subsection{K-moduli wall crossing for quartic curves}

We now apply the tools from the previous toolbox to consider moduli of pairs $(\bP^2, cD)$ where $0 < c < 3/d$, $D$ is a degree $d$ curve, and try to understand all K-semistable pairs of this form and their K-semistable degenerations.  

We know that if $(X,cD)$ is a K-semistable object in this moduli space, it is klt by Theorem \ref{thm:kssisklt}, so $X$ is log terminal.  Additionally, by Theorem \ref{t:kltdegens}, the possible singularities on $X$ are all of the form $\frac{1}{n^2}(1,nl-1)$ where $\gcd(l,n) = 1$, and $n$ is one of the elements of a Markov triple $(a,b,c)$ satisfying $a^2 + b^2 + c^2 = 3abc$.  The Cartier index of the singularity $\frac{1}{n^2}(1,na-1)$ is $n$.

In this case, using a more refined estimate comparing $\widehat{\vol}(x,X,D)$ and $\widehat{\vol}(x,D)$ in the computation in Property \ref{propsofvol}(5), in \cite{ADLpub} it is shown that for $d$ is not divisible by 3, 
\[ \ind (x) \le \min \left\{ \lfloor \frac{3}{3-dc} \rfloor, d \right\} .\]  

For $c \ll 1$, this implies $\ind(x) = 1$, so $X$ is Gorenstein.  The only Gorenstein surface in Theorem \ref{t:kltdegens} is $X = \bP^2$.  Therefore, for $c \ll 1$, all K-semistable pairs must be of the form $(\bP^2, cD)$ for some degree $d$ plane curve.  We can in fact strengthen the result in Theorem \ref{thm:k=git} to conclude in that in this case, the K-moduli space coincides with the GIT moduli space.  We summarize the consequences of the index bound for $d = 4$.

\begin{prop}
        When $d = 4$ and $c < \frac{3}{4}$, the bound above implies $\ind (x) \le 4$.  By the description of the Markov triples in Theorem \ref{t:kltdegens}, the only possible surfaces appearing in any K-moduli space $P_{4,c}$ are $\bP^2$ and $\bP(1,1,4)$.  When $d = 4$ and $c < \frac{3}{8}$, the bound above implies $\ind(x) < 2$, so the only possible surface appearing has index $1$ and must be $\bP^2$.  For $c < \frac{3}{8}$, the K-moduli stack and space must coincide with the GIT moduli stack and space. 
\end{prop}

For low degree curves in $\bP^2$, the GIT moduli space is well described.  Therefore, to completely understand the K-moduli spaces, we can start with the GIT moduli space, increase the coefficient $c$ until something ``destabilizes'' (which will give a wall crossing), find the K-semistable replacement, and continue.  

\begin{example}
We will work this out completely for degree 4 curves, following \cite{ADLpub}.  From above, for  $c < \frac{3}{8}$, we have $\ind(x) < 2$, so only $\bP^2$ appears and the K-moduli space of pairs $(\bP^2,cD)$ must be isomorphic to the GIT moduli space.  In other words, for $c < \frac{3}{8}$ we start with the GIT moduli space of quartic plane curves.  

In the GIT moduli space, there is a unique polystable point $[2Q]$ corresponding to a doubled smooth conic, which is the unique non-reduced curve that is GIT semistable.  For quartic curves, part (2) of Proposition \ref{prop:k-interpolation} and the given corollary says that as long as the log canonical threshold of the pair $(\bP^2,D)$ is at least $\frac{3}{4}$, then $(\bP^2, cD)$ is K-polystable for all $c\in(0,\frac{3}{4})$.  Every polystable curve $D$ in the GIT moduli space other than the double conic has this property, so for $D$ other than the double conic, the pair $(\bP^2, cD)$ is K-polystable for all $c$.  However, the pair $(\bP^2,c(2Q))$ is not K-polystable for all $c$: we compute $\delta_{\bP^2,c(2Q)} (Q) \ge 1$ if and only if $c \le \frac{3}{8}$.  This pair is therefore unstable for $c> \frac{3}{8}$.

When $c = \frac{3}{8}$, we do the following.  Consider a family $D$ of smooth quartic curves degenerating to the double conic, inside $X = \bP^2 \times \bA^1$.  In $X$, blow up the conic.  This produces a threefold $Y$ with exceptional divisor $E \cong \mathbb{F}_4$.  Let $D_Y$ be the strict transform of $D$ in $Y$.  Now, the surface that was the original central fiber of $X$ is contractible, and we can contract it to produce a family $Z$ of $\bP^2$ degenerating to $\bP(1,1,4)$.  As we cross the wall at $c = \frac{3}{8}$, we can verify that the new central fiber $(\bP(1,1,4), (c+ \epsilon)D')$ is K-semistable by direct computation.  This is illustrated in Figure \ref{f:doubleconic}.

\begin{figure}[h]
\begin{tikzcd}[row sep=huge]
&[-10em] \begin{tabular}{c}
\begin{tikzpicture}[gren0/.style = {draw, circle,fill=greener!80,scale=.7},gren/.style ={draw, circle, fill=greener!80,scale=.4},blk/.style ={draw, circle, fill=black!,scale=.2},plc/.style ={draw, circle, color=white!100,fill=white!100,scale=0.02},smt/.style ={draw, circle, color=gray!100,fill=gray!100,scale=0.02},lbl/.style ={scale=.2}] 

\node[smt] at (-2.5+0,0) (1){};
\node[smt] at (-2.5+0,2) (2){};
\node[smt] at (-2.5+1.2,2.7) (3){};
\node[smt] at (-2.5+1.2,0.7) (4){};
\draw [black] plot [smooth cycle, tension=.1] coordinates { (1) (2) (3) (4) };

\draw [thick,blue] plot [smooth, tension=1] coordinates { (-2.5+.2,1.4) (-2.5+1,2.1) (-2.5+.3,0.8) (-2.5+.8,1.1)};
\node[below right,node font=\tiny,text=blue] at (-2.5+0.3,0.8) {smooth quartic};

\node[smt] at (0,0) (5){};
\node[smt] at (0,2) (6){};
\node[smt] at (1.2,2.7) (7){};
\node[smt] at (1.2,0.7) (8){};
\draw [black] plot [smooth cycle, tension=.1] coordinates { (5) (6) (7) (8) };

\draw [thick,violet] plot [smooth cycle, tension=1] coordinates { (0.35,1) (0.35,1.4) (0.85,1.6) (0.85,1.2) };
\draw [thick,violet] plot [smooth cycle, tension=1] coordinates { (0.2,1+1.5) (0.2,1.4+1.5) (1,1.7+1.5) (1,1.3+1.5) };
\draw [thick,violet] plot [smooth cycle, tension=1] coordinates { (0.2,1-1.5) (0.2,1.4-1.5) (1,1.7-1.5) (1,1.3-1.5) };
\draw [thick,violet] plot [smooth, tension=1] coordinates { (0.05,2.54) (0.18,2.2) (0.25,1.1)};
\draw [thick,densely dotted,violet] plot [smooth, tension=1] coordinates { (0.25,1.1) (0.21,0.3) (0.2,0.1)};
\draw [thick,violet] plot [smooth, tension=1] coordinates { (0.2,0.1) (0.15,-0.05) (0.03,-0.35)};
\draw [thick,densely dotted,violet] plot [smooth, tension=1] coordinates { (0.96,1.44) (1,0.58)};
\draw [thick,violet] plot [smooth, tension=1] coordinates { (1,0.58) (1.02,0.35) (1.1,0.2) (1.17,0.1)};
\draw [thick,violet] plot [smooth, tension=1] coordinates { (1.15,2.96) (1.02,2.5) (0.96,1.44)};


\node[smt] at (0.2,1.85) (A){};
\node[smt] at (1,2.35) (B){};
\node[smt] at (0.8,2.45) (C){};
\node[smt] at (0.18,2.2) (D){};
\node[smt] at (0.22,2.15) (E){};
\node[smt] at (1.05,2.65) (F){};
\draw [thick,blue,-] (A) to [bend right=30] (B) to [bend right=30] (C) to [bend right=30] (D) to [bend right=30] (E) to [bend right=30] (F) ;
\node[below right, node font=\tiny, text=blue] at (1.05,2.6) {strict transform of family of curves};
\node[below right, node font=\tiny, text=violet] at (1.05,0) {exceptional divisor $E \cong \mathbb{F}_4$};

\draw[thick,color=gray] (-2.5-0.5,-.7) to (1.2+0.5,-.7);
\node[blk] at (-2.5+0.6,-.7) (9){};
\node[below right, node font=\tiny] at (9) {$t$};
\node[blk] at (0.6,-.7) (10){};
\node[below right, node font=\tiny] at (10) {$0$};
\node[right, node font=\tiny] at (1.2+0.5,-.7) {$\bA^1_t$};

\node[above left, node font=\tiny] at (2) {$Y_t= \bP^2_t$};
\node[node font=\huge, text=gray] at (-2.5+1.8,1.5) {$\rightsquigarrow$};
\end{tikzpicture}
\end{tabular} \arrow[ld, xshift=-2ex, swap, "\text{blow up conic in central fiber}"] \arrow[rd, xshift=-5ex, "\text{contract } \bP^2 \text{ in central fiber}"] &[-17em] \\
\begin{tabular}{c}
\begin{tikzpicture}[gren0/.style = {draw, circle,fill=greener!80,scale=.7},gren/.style ={draw, circle, fill=greener!80,scale=.4},blk/.style ={draw, circle, fill=black!,scale=.2},plc/.style ={draw, circle, color=white!100,fill=white!100,scale=0.02},smt/.style ={draw, circle, color=gray!100,fill=gray!100,scale=0.02},lbl/.style ={scale=.2}] 

\node[smt] at (-2.5+0,0) (1){};
\node[smt] at (-2.5+0,2) (2){};
\node[smt] at (-2.5+1.2,2.7) (3){};
\node[smt] at (-2.5+1.2,0.7) (4){};
\draw [black] plot [smooth cycle, tension=.1] coordinates { (1) (2) (3) (4) };

\draw [thick,blue] plot [smooth, tension=1] coordinates { (-2.5+.2,1.4) (-2.5+1,2.1) (-2.5+.3,0.8) (-2.5+.8,1.1)};
\node[below right,node font=\tiny,text=blue] at (-2.5+0.3,0.8) {smooth quartic};

\node[smt] at (0,0) (5){};
\node[smt] at (0,2) (6){};
\node[smt] at (1.2,2.7) (7){};
\node[smt] at (1.2,0.7) (8){};
\draw [black] plot [smooth cycle, tension=.1] coordinates { (5) (6) (7) (8) };

\draw [very thick,blue] plot [smooth cycle, tension=1] coordinates { (0.35,1) (0.35,1.4) (0.85,1.6) (0.85,1.2) };
\node[below right,node font=\tiny,text=blue] at (0.6,1.1) {2(smooth conic)};

\draw[thick,color=gray] (-2.5-0.5,-.5) to (1.2+0.5,-.5);
\node[blk] at (-2.5+0.6,-.5) (9){};
\node[below right, node font=\tiny] at (9) {$t$};
\node[blk] at (0.6,-.5) (10){};
\node[below right, node font=\tiny] at (10) {$0$};
\node[right, node font=\tiny] at (1.2+0.5,-.5) {$\bA^1_t$};

\node[above left, node font=\tiny] at (2) {$X_t = \bP^2_t$};
\node[below right, node font=\tiny] at (7) {$X_0 = \bP^2_0$};
\node[node font=\huge, text=gray] at (-2.5+1.8,1.5) {$\rightsquigarrow$};
\end{tikzpicture}
\end{tabular} & &\begin{tabular}{c}
\begin{tikzpicture}[gren0/.style = {draw, circle,fill=greener!80,scale=.7},gren/.style ={draw, circle, fill=greener!80,scale=.4},blk/.style ={draw, circle, fill=black!,scale=.2},plc/.style ={draw, circle, color=white!100,fill=white!100,scale=0.02},smt/.style ={draw, circle, color=gray!100,fill=gray!100,scale=0.02},lbl/.style ={scale=.2}] 

\node[smt] at (-2.5+0,0) (1){};
\node[smt] at (-2.5+0,2) (2){};
\node[smt] at (-2.5+1.2,2.7) (3){};
\node[smt] at (-2.5+1.2,0.7) (4){};
\draw [black] plot [smooth cycle, tension=.1] coordinates { (1) (2) (3) (4) };

\draw [thick,blue] plot [smooth, tension=1] coordinates { (-2.5+.2,1.4) (-2.5+1,2.1) (-2.5+.3,0.8) (-2.5+.8,1.1)};
\node[below right,node font=\tiny,text=blue] at (-2.5+0.3,0.8) {smooth quartic};

\draw [thick,violet] plot [smooth cycle, tension=1] coordinates { (0.2,1+1) (0.2,1.4+1) (1,1.7+1) (1,1.3+1) };
\draw [thick,violet] plot [smooth cycle, tension=1] coordinates { (0.2,1-1) (0.2,1.4-1) (1,1.7-1) (1,1.3-1) };
\draw[thick,violet] (0.05,2.07) -- (1.15,0.62);
\draw[thick,violet] (1.17,2.53) -- (0.03,0.17);
\node[below right,node font=\tiny,text=violet] at (1.1,0.6) {new central fiber $Z_0\cong \bP(1,1,4)$};

\node[smt] at (0.4,1.6) (A){};
\node[smt] at (.78,1.75) (B){};
\node[smt] at (.76,1.9) (C){};
\node[smt] at (0.24,1.85) (D){};
\node[smt] at (0.28,1.8) (E){};
\node[smt] at (0.94,2.05) (F){};
\draw [thick,blue,-] (A) to [bend right=30] (B) to [bend right=20] (C) to [bend right=30] (D) to [bend right=30] (E) to [bend right=30] (F) ;
\node[below right, node font=\tiny, text=blue] at (0.9,1.9) {hyperelliptic curve};

\draw[thick,color=gray] (-2.5-0.5,-.5) to (1.2+0.5,-.5);
\node[blk] at (-2.5+0.6,-.5) (9){};
\node[below right, node font=\tiny] at (9) {$t$};
\node[blk] at (0.6,-.5) (10){};
\node[below right, node font=\tiny] at (10) {$0$};
\node[right, node font=\tiny] at (1.2+0.5,-.5) {$\bA^1_t$};

\node[above left, node font=\tiny] at (2) {$Z_t= \bP^2_t$};
\node[node font=\huge, text=gray] at (-2.5+1.8,1.5) {$\rightsquigarrow$};
\end{tikzpicture}
\end{tabular} \\
\end{tikzcd}
\vspace{-.5in}
\caption{Replacement of the double conic.}
\label{f:doubleconic}
\end{figure}
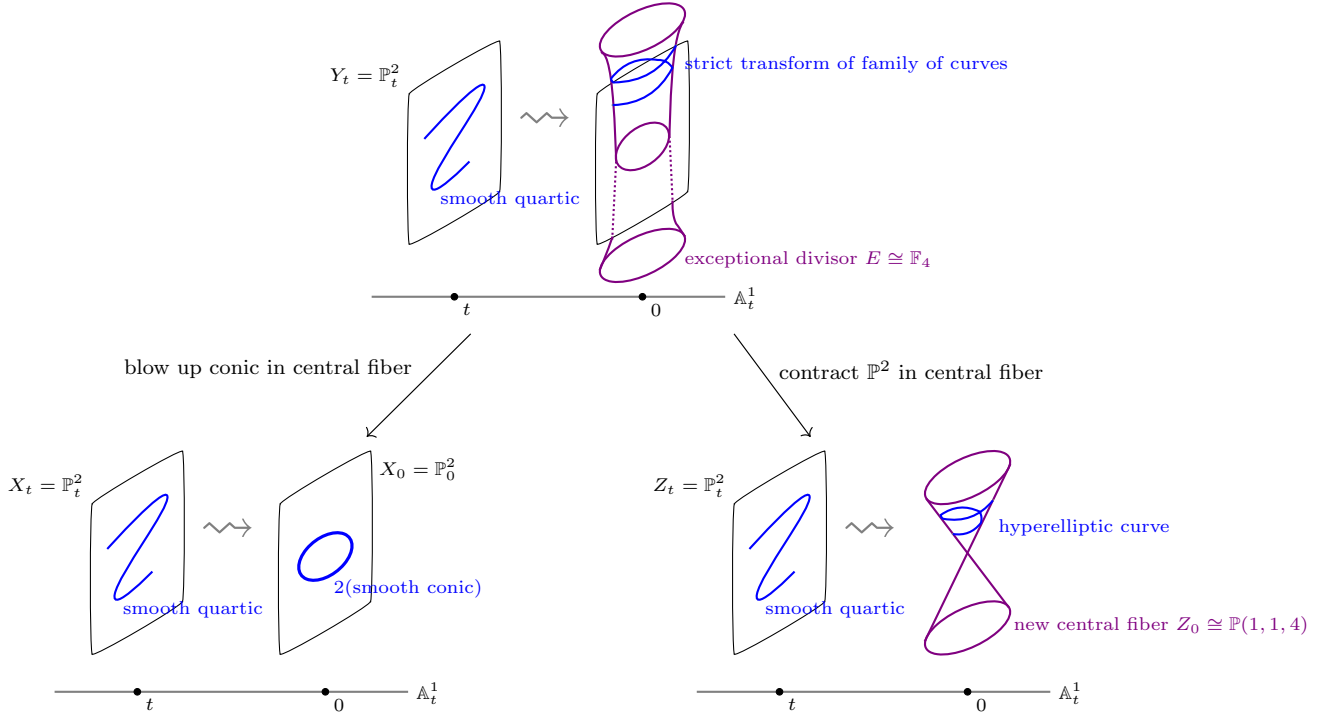

This previous discussion gives a wall crossing at $c = \frac{3}{8}$, where it suggests that the doubled conic curve should be replaced by hyperelliptic curves on $\bP(1,1,4)$.  Indeed, pairs $(\bP(1,1,4),cD)$ where $D$ is smooth and hyperelliptic are K-semistable for $c > \frac{3}{8}$.  To verify that this is the structure of the wall crossing, we prove that $(\bP(1,1,4),\frac{3}{8}(2Q'))$ is K-polystable where $Q'$ is a section of the cone $\bP(1,1,4)$.  By openness of K-semistability, for any hyperelliptic curve $D \subset \bP(1,1,4)$, the pair $(\bP(1,1,4),cD)$ admits a special degeneration to $(\bP(1,1,4), c(2Q'))$, so we conclude $(\bP(1,1,4),cD)$ are K-semistable for $c = \frac{3}{8}$.  It is computationally straightforward to show for a K-polystable pair $(\bP(1,1,4), c D)$ and $c> \frac{3}{8}$, $D$ has log canonical threshold at least $\frac{3}{4}$, so again by interpolation these are K-polystable for all $c\in(\frac{3}{8},\frac{3}{4})$.  Therefore, we know there are no other wall crossings in this range of coefficient.

We can furthermore understand the structure of the wall crossing.  In \cite{ADLpub}, we prove that there is a morphism $P_{4,c} \to P_{4,\frac{3}{8} - \epsilon}$ which is the blow up of the point corresponding to the double conic, and this is a Kirwan partial desingularization. 

\end{example}

All wall crossings for quartics will be summarized in Figure \ref{f:quarticwalls}.

\subsection{Boundary polarized Calabi Yau wall crossing}

For the wall crossing when $c \ge \frac{3}{4}$ in the log Calabi Yau and general type region; we use the theory of the boundary polarized Calabi Yau wall crossing and the wall crossings for KSBA stable pairs.  In this section, we consider the wall at $c = \frac{3}{4}$.  Suppose $d$ is not a multiple of $3$ and let $P_d^{\rm klt} $ be the locus of pairs $(X,D)$ such that $X$ is a $\bQ$-Gorenstein smoothable Fano surface, $K_X^2 = 9$, $D \in |-\frac{d}{3}K_X|$, and $(X,\frac{3}{d})$ is klt.  Then, by the discussion in Section \ref{ss:bpcy}, for $\epsilon \ll1$, $P_d^{\rm klt} \subset P^K_{d, \frac{3}{d} - \epsilon}$ and $P_d^{\rm klt} \subset P^{\rm CY}_{d}$ and $P_d^{\rm klt} \subset P^{\rm KSBA}_{d, \frac{3}{d} + \epsilon}$.  Furthermore, if $[(X,(\frac{3}{d}-\epsilon)D)] \in P^K_{d, \frac{3}{d} - \epsilon} $ is such that $(X,\frac{3}{d}D)$ is not klt, it is necessarily log canonical and $(X,(\frac{3}{d} + \epsilon)D)$ is not log canonical, so $\frac{3}{d} = \lct (D)$. If $[(X,(\frac{3}{d}+\epsilon)D)] \in P^{\rm KSBA}_{d, \frac{3}{d} + \epsilon} $ is such that $(X,\frac{3}{d}D)$ is not klt (using that $K_X$ and $D$ are both $\bQ$-Cartier divisors and $X$ has plt normalization by \cite[Theorem 7.1]{hacking}), $X$ is necessarily not normal and therefore $(X,(\frac{3}{d} - \epsilon)D)$ is not K-semistable.  

In particular, this says the wall crossing at $c=\frac{3}{d}$ must interchange the K-semistable pairs $(X,(c-\epsilon)D)$ such that $\frac{3}{d} = \lct(D)$ with the slc pairs $(X,(c+\epsilon)D)$ such that $X$ is non-normal.  By the classification in the previous section, the curves in $P^K_{4,c -\epsilon}$ with log canonical threshold $c$ are precisely the curves on $\bP^2$ or $\bP(1,1,4)$ with tacnodes.  By \cite[Section 11]{hacking}, the pairs in $P^{\KSBA}_{4,c+\epsilon}$ such that the surface is non-normal are those with $X = \bP(1,1,2)  \cup \bP(1,1,2)$.

In \cite{BABWILD}, we prove the following: 

\begin{thm}
	The polystable pairs parametrized by $P_4^{\rm CY}$ are the following:
	\begin{enumerate}
		\item $(\bP^2,\frac{3}{4}C)$ where $C$ is a plane quartic curve with at worst cuspidal singularities;
		\item $(\bP(1,1,4), \frac{3}{4}C)$ where $C$ is a degree $8$ curve not passing through the cone point with at worst cuspidal singularities;
		\item $(\bP(1,1,2)\cup \bP(1,1,2), \frac{3}{4}C)$ where $C$ is of degree $4$ on each component, and has a tacnodal singularity on each component, illustrated in Figure \ref{f:polystablequartics}. 
	\end{enumerate}
    The pairs in (1) and (2) are the points of $P_4^{\klt}$.  Every pair $(X,cD)$ in $P^{\K}_{4,c-\epsilon}$  or $P^{\KSBA}_{4,c+\epsilon}$ that is not klt when $c = \frac{3}{4}$ admits an isotrivial specialization to the pair in (3).
\end{thm}

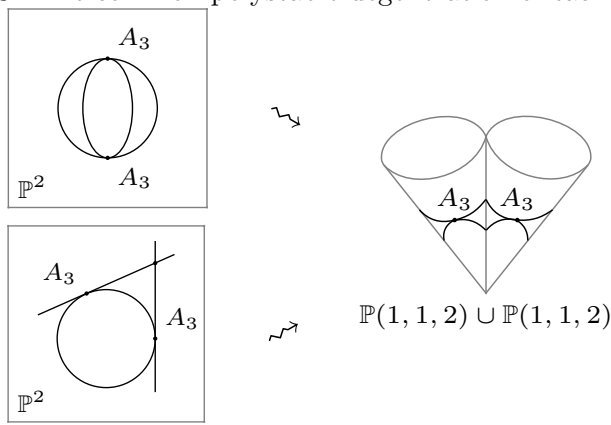
\begin{figure}[h]
\caption{The common polystable degeneration of tacnodal curves.}
\label{f:polystablequartics}
\begin{tabular}{ccc}
{\resizebox{.2\textwidth}{!}{\begin{tabular}{c}
\begin{tikzpicture}[gren0/.style = {draw, circle,fill=greener!80,scale=.7},gren/.style ={draw, circle, fill=greener!80,scale=.4},blk/.style ={draw, circle, fill=black!,scale=.08},plc/.style ={draw, circle, color=white!100,fill=white!100,scale=0.02},smt/.style ={draw, circle, color=gray!100,fill=gray!100,scale=0.02},lbl/.style ={scale=.2}] 
\node[smt] at (0,0) (1){};
\node[smt] at (0,2) (2){};
\node[smt] at (2,2) (3){};
\node[smt] at (2,0) (4){};
\draw [-,color=gray] (1) to (2) to (3) to (4) to (1);
\draw (1,1) circle (0.5);
\draw (1,1) ellipse (0.25 and 0.5);
\node[blk] at (1,0.5) {};
\node[blk] at (1,1.5) {};

\node[above right, node font=\tiny] at (0,0) {$\bP^2$};
\node[below right, node font=\tiny] at (1,0.5) {$A_3$};
\node[above right, node font=\tiny] at (1,1.5) {$A_3$};

\end{tikzpicture}
\end{tabular}}} & \rotatebox[origin=c]{-30}{\Large{$\rightsquigarrow$}} 
& \multirow{2}{*}{\resizebox{.26\textwidth}{!}{\begin{tabular}{c}
\begin{tikzpicture}[gren0/.style = {draw, circle,fill=greener!80,scale=.7},gren/.style ={draw, circle, fill=greener!80,scale=.4},blk/.style ={draw, circle, fill=black!,scale=.08},plc/.style ={draw, circle, color=white!100,fill=white!100,scale=0.02},smt/.style ={draw, circle, color=gray!100,fill=gray!100,scale=0.02},lbl/.style ={scale=.2}] 
\node[smt] at (0,0) (1){};
\node[smt] at (0,1.5) (2){};
\node[plc] at (1,1.25) (3){};
\node[plc] at (0.98,1.23) (3r){};
\node[plc] at (-1,1.25) (4){};
\node[plc] at (-0.98,1.23) (4r){};
\node[blk] at (0.3,0.7) (5){};
\node[plc] at (0,0.9) (6){};
\node[plc] at (0,0.6) (7){};
\node[plc] at (0.64,0.8) (8){};
\node[plc] at (0.4,0.5) (9){};
\node[blk] at (-0.3,0.7) (10){};
\node[plc] at (-0.64,0.8) (11){};
\node[plc] at (-0.4,0.5) (12){};
\draw [-,color=gray] (1) to (2);
\draw [-,color=gray] (1) to (3r); 
\draw [-,color=gray] (1) to (4r);
\draw [-,color=gray,rounded corners=1pt] (2) to [bend left=80] (3) to [bend left = 60] (2);
\draw [-,color=gray,rounded corners=1pt] (2) to [bend left=60] (4) to [bend left = 80] (2);
\draw [-] (6) to [bend right=30] (5) to [bend right=20] (8);
\draw [-] (7) to [bend left=30] (5) to [bend left=30] (9);
\draw [-] (11) to [bend right=30] (10) to [bend right=20] (6);
\draw [-] (12) to [bend left=37] (10) to [bend left=30] (7);

\node[above, node font=\tiny] at (5) {$A_3$};
\node[above, node font=\tiny] at (10) {$A_3$};
\node[below, node font=\tiny] at (1) {$\bP(1,1,2) \cup \bP(1,1,2)$};

\end{tikzpicture}
\end{tabular}}} \\ {\resizebox{.2\textwidth}{!}{\begin{tabular}{c}
\begin{tikzpicture}[gren0/.style = {draw, circle,fill=greener!80,scale=.7},gren/.style ={draw, circle, fill=greener!80,scale=.4},blk/.style ={draw, circle, fill=black!,scale=.08},plc/.style ={draw, circle, color=white!100,fill=white!100,scale=0.02},smt/.style ={draw, circle, color=gray!100,fill=gray!100,scale=0.02},lbl/.style ={scale=.2}] 
\node[smt] at (-.3,-.15) (1){};
\node[smt] at (-.3,1.85) (2){};
\node[smt] at (1.7,1.85) (3){};
\node[smt] at (1.7,-.15) (4){};
\node[blk] at (1.2,.7) (5){};
\node[blk] at (0.5,1.16) (6){};
\node[blk] at (1.2,1.47) (7){};
\node[plc] at (1.2, 1.7) (8){};
\node[plc] at (1.4,1.56) (9){};
\node[plc] at (1.2, 0.15) (10){};
\node[plc] at (0,0.94) (11){};

\draw [-,color=gray] (1) to (2) to (3) to (4) to (1);
\draw (0.7,0.7) circle (0.5);
\draw [-] (10) to (8);
\draw [-] (11) to (9);

\node[above right, node font=\tiny] at (5) {$A_3$};
\node[above left, node font=\tiny] at (6) {$A_3$};
\node[above right, node font=\tiny] at (1) {$\bP^2$};

\end{tikzpicture}
\end{tabular}}} & \rotatebox[origin=c]{30}{\Large{$\rightsquigarrow$}}  &
\end{tabular}
\end{figure}

On the level of the moduli spaces themselves, the wall crossing at $c = \frac{3}{4} $ is a flip of the locus of curves with tacnodes, replacing it with the locus of curves on $\bP(1,1,2) \cup \bP(1,1,2)$.  This is illustrated in Figure \ref{f:quarticwalls}.  

We can also identify the moduli space $P^{\rm CY}_4$ as a moduli space of K3 surfaces. In \cite{kondok3}, Kond\={o}  uses the observation 
that a degree 4 cyclic cover of $\bP^2$ along a smooth quartic curve is a degree K3 surface to construct a compactification $P_4\subset P_4^*$, which is the Baily-Borel compactification of a period domain parametrizing 
degree 4 K3 surfaces with a $\mathbb{Z}/4\mathbb{Z}$ symmetry.  It is shown in \cite{BABWILD} that this space is isomorphic to $P^{\rm CY}_4$.

\begin{prop}\label{prop:quartics}
The birational map $P_4^{\rm CY} \dashrightarrow P_4^{*}$ is an isomorphism.
\end{prop}

This gives the Baily-Borel compactification the structure of a good moduli space of a moduli functor.  This result is strengthened in \cite{BL} to moduli spaces of K3 surfaces that are branched covers of other del Pezzo surfaces.

\subsection{KSBA wall crossing}

To complete the wall crossing for quartics, we must now consider the case of coefficient $c \in (\frac{3}{4}, 1]$. We can explicitly describe pairs parametrized by the moduli space $P^\KSBA_{4, \frac{3}{4} +\epsilon}$ due to Hacking in \cite{hacking} (alternatively, using the description of the boundary polarized Calabi Yau wall crossing above): 

\begin{prop}
    If $[(X,cD)] \in P^\KSBA_{4, c}$ where $c = \frac{3}{4} + \epsilon$ for $\epsilon \ll 1$, then $X = \bP^2, \bP(1,1,4)$, or $\bP(1,1,2) \cup \bP(1,1,2)$ and $D \in |-\frac{4}{3}K_X|$ such that $D$ has at worst $A_2$ (cuspidal) singularities. 
\end{prop}

An explicit description of the configuration of the allowed cusps can be found in \cite[Lemma 4.2]{hassett}.  By definition in this moduli space, $K_X + cD$ is ample and $K_X + \frac{3}{4} D \sim 0$, we conclude that $D$ is ample and hence $K_X + c'D$ is ample for all $c' \in (\frac{3}{4},1]$.  Because curves with cusps have log canonical threshold equal to $\frac{5}{6}$, this implies every such pair $(X,cD)$ is KSBA stable for $c \in (\frac{3}{4}, \frac{5}{6}]$.  The first wall crossing in the canonically polarized region is therefore at $\frac{5}{6}$ when cuspidal curves must undergo a replacement.  To determine which surfaces and curves appear on the other side of the wall, we must at least partially resolve these singularities and compute canonical models.  This wall crossing was worked out by Hassett in \cite{hassett}.  Let $[(X, cD)] \in P^{\KSBA}_{4, \frac{3}{4}+\epsilon}$ be a stable pair such that $D$ has a cusp.  Following \cite{hassett}, define a surface $Y$ as follows: 
\begin{itemize}
    \item for each cusp $p \in D$ with local equation $x^2  = y^3$, perform the $(3,2)$ weighted blow up of $X$ at $p$, resulting in a surface $\tilde{X} \to X$ such that $\tilde{X}$ has a $\frac{1}{2}(1,1)$ and a $\frac{1}{3}(1,1)$ singularity along the exceptional divisor $E$.
    \item for each weighted blow-up, glue a copy of $\bP(1,2,3)$ to $\tilde{X}$ along the exceptional divisor $E$ in $\tilde{X}$ and a section of $\calO(1)$ in $\bP(1,2,3)$ so the total space has singularities $(xy = 0) \subset \frac{1}{2}(1,1,1)$ and $(xy = 0) \subset \frac{1}{3}(1,2,1)$. 
    \item define $Y$ as the union of $\tilde{X}$ and the surfaces $\bP(1,2,3)$.
\end{itemize}
Note that $Y$ could be alternatively constructed in the following way.  First, observe that each such $X$ is a degree 2 hypersurface in $\bP(1,1,1,2)$.  Denoting the coordinates on $\bP(1,1,1,2)$ by $[x:y:z:w]$, $\bP^2 $ is given by $ (w = f(x,y,z))$ for any degree 2 polynomial in $x,y,z$; $\bP(1,1,4) $ is given by $ (xy = z^2)$; and $\bP(1,1,2) \cup \bP(1,1,2) $ is given by $ (xy = 0)$. Taking an appropriate pencil of these hypersurfaces gives a smoothing of $X$ to $\bP^2$ over $\bA^1_t$.  Performing the $(2,3,1)$ weighted blow-up of this family at any point corresponding to a cusp in $X$ with local coordinates $(x,y,t) = (0,0,0)$ yields the same surface $Y$ and demonstrates that $Y$ is smoothable to $\bP^2$. 

Hassett proves that, for any such $Y$ and curve $D_Y \in |-\frac{4}{3}K_Y|$ with at worst nodal singularities, $K_Y + cD_Y$ is ample and $(Y,cD_Y)$ is slc for $c>\frac{5}{6}$ (c.f. \cite[Lemma 4.3]{hassett}) and these give all elements of $P^\KSBA_{4, c}$ for $c\in(\frac{5}{6},1]$.  Furthermore, every $D_Y \subset Y$ has an \textit{elliptic tail} in each component $\bP(1,2,3)$, i.e. an arithmetic genus one curve meeting the rest of the curve at one point.  For $c\in(\frac{5}{6},1]$, the map $P^\KSBA_{4, c} \to P^{\KSBA}_{4,\frac{5}{6}}$ contracts each $\bP(1,2,3)$ component of $Y$, and on the curve $D_Y$, contracts the elliptic tail. 

There is a natural forgetful map $P^{\KSBA}_{4,c} \to \overline{M}_3$, the moduli space of Deligne-Mumford stable genus 3 curves, given by $[(X,D)] \mapsto [D]$.  The main theorem in \cite{hassett} is that this is an isomorphism\footnote{These types of forgetful maps are not expected to be isomorphisms in general.  For example, if $D$ is a quintic plane curve, the associated forgetful morphism is not surjective to $\oM_6$ (not every genus 6 curve is planar) and it is also not injective (it is finite of degree $14$ on the locus of smooth hyperelliptic curves, see e.g. \cite[Corollary 7.8]{ADLpub}).}.  

In summary, we have the following description of the wall crossing for moduli of quartic plane curves on the level of moduli spaces.  There is a similar description on the level of moduli stacks, see \cite{ADLpub} for more details.  We suppress the superscript on $P_{4,c}$. 

\begin{theorem}
    Let $P_{4,c}$ be the moduli space compactifying the locus of pairs $(\bP^2, cC)$ where $C$ is a smooth quartic plane curve.
        \begin{itemize}
            \item When $c \in (0, \frac{3}{8})$, $P_{4,c}$ is isomorphic to the GIT quotient of quartic plane curves, and parametrizes pairs $(X,cD)$ where $X = \bP^2$ and $D$ is either a doubled conic curve or a reduced quartic curve with at worst $A_3$ (tacnodal) singularities. 
            \item When $c = \frac{3}{8}$, $P_{4,c}$ is again isomorphic to the GIT quotient, but the polystable point representing the double conic curve corresponds to the pair $(\bP(1,1,4), \frac{3}{8} (2Q))$ where $Q = (z =0)$. 
            \item When $c \in (\frac{3}{8}, \frac{3}{4})$, $P_{4,c}$ is a blow up of the GIT quotient at the point corresponding to the doubled conic curve.  It parametrizes pairs $(X,cD)$ such that $X = \bP^2$ or $= \bP(1,1,4)$ and $D$ has at worst $A_3$ (tacnodal) singularities. 
            \item When $c = \frac{3}{4}$, the moduli space $P_{4, \frac{3}{4}}$ is the contraction of the locus of tacnodal curves in $P_{4, \frac{3}{4} - \epsilon}$ to the point described in Figure \ref{f:polystablequartics}.
            \item When $c \in (\frac{3}{4}, \frac{5}{6}]$, $P_{4,c}$ parametrizes genus 3 curves with at worst cusps on the surfaces $\bP^2, \bP(1,1,4)$, or $\bP(1,1,2) \cup \bP(1,1,2)$. 
            \item When $c \in (\frac{5}{6},1]$, $P_{4,c}$ parametrizes genus 3 curves with at worst nodes on the surfaces $\bP^2, \bP(1,1,4)$, $\bP(1,1,2) \cup \bP(1,1,2)$, or a surface $Y$ as in the construction above obtained by blowing up and gluing to $\bP(1,2,3)$.  Furthermore, $P_{4,c} \cong \overline{M}_3$.
        \end{itemize}
\end{theorem}

This is summarized in Figure \ref{f:quarticwalls}.

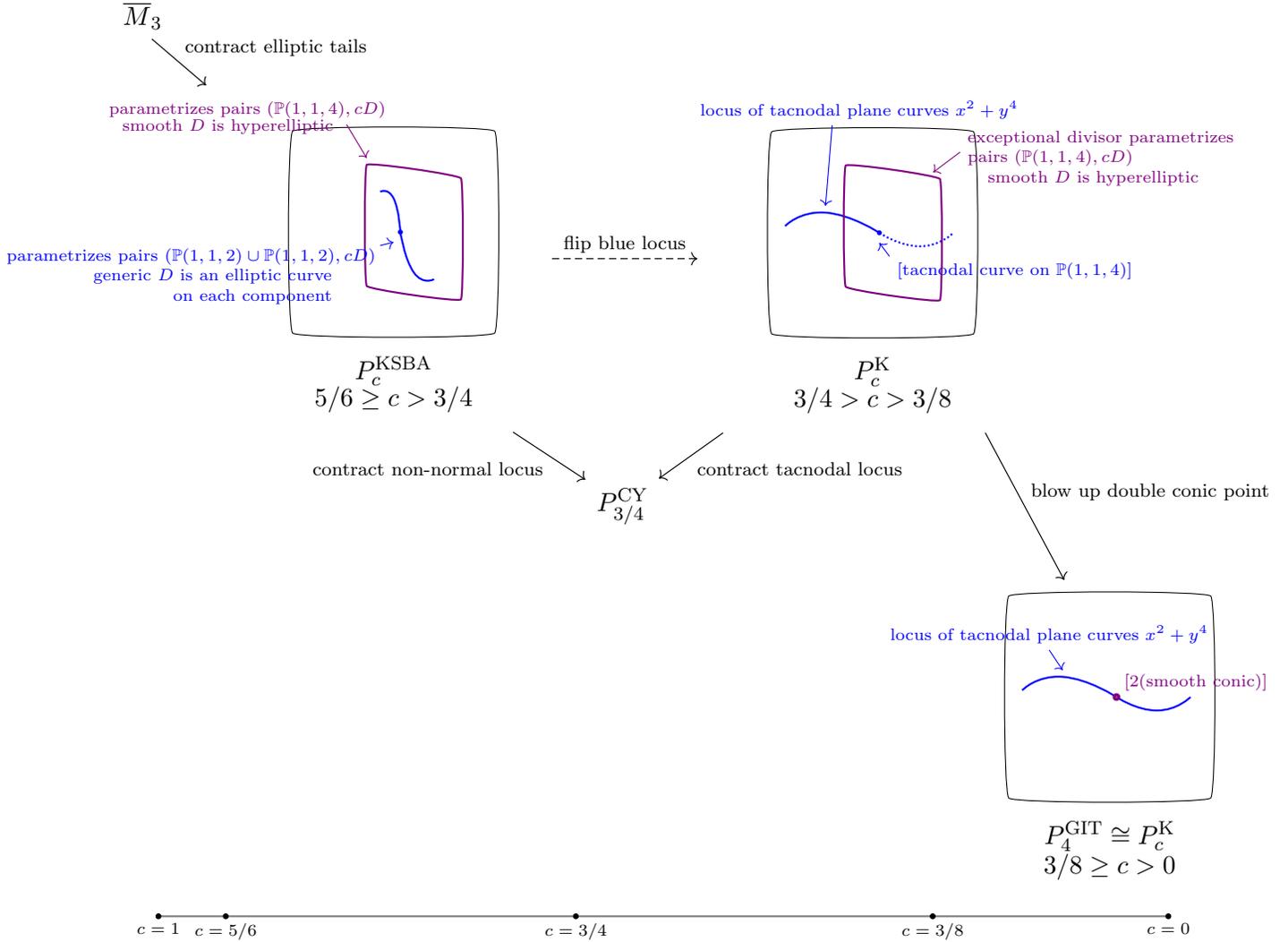
\begin{figure}[h]
\begin{tikzcd}[row sep=normal]
\overline{M}_3 \arrow[rd, end anchor={[xshift=3ex]}, "\text{contract elliptic tails}"]  &[-10em] & &[-2em] &[-18em] \\
& \begin{tabular}{c}
\begin{tikzpicture}[gren0/.style = {draw, circle,fill=greener!80,scale=.7},gren/.style ={draw, circle, fill=greener!80,scale=.4},blk/.style ={draw, circle, fill=black!,scale=.2},plc/.style ={draw, circle, color=white!100,fill=white!100,scale=0.02},smt/.style ={draw, circle, color=gray!100,fill=gray!100,scale=0.02},lbl/.style ={scale=.2}] 

\draw [black] plot [smooth cycle, tension=.1] coordinates { (0,0) (0,3) (3,3) (3,0) };

\draw [thick,blue] plot [smooth, tension=1] coordinates { (1.3,2.1) (1.5,2) (1.6,1.5) };
\draw [thick,blue] plot [smooth, tension=1] coordinates {  (1.6,1.5) (1.8,.9) (2.1,.8) };

\draw [thick,violet] plot [smooth cycle, tension=.1] coordinates { (1.1,.7) (1.1,2.5) (2.5,2.3) (2.5,.5) };


\filldraw[color=blue, fill=blue] (1.6,1.5) circle (0.03);

\draw[color=blue, ->] (1.3,1.3) -- (1.5,1.35);
\node[below left, node font=\tiny, color=blue] at (1.35,1.35) {parametrizes pairs $(\bP(1,1,2) \cup \bP(1,1,2), cD)$};
\node[below left, node font=\tiny, color=blue] at (0.7,1.05) {generic $D$ is an elliptic curve};
\node[below left, node font=\tiny, color=blue] at (0.7,0.75) {on each component};
\draw[color=violet, ->] (.8,3.1) -- (1.1,2.6);
\node[above left, node font=\tiny, color=violet] at (1.5,3.1) {parametrizes pairs $(\bP(1,1,4),cD)$};
\node[left, node font=\tiny, color=violet] at (.72,3.1) {smooth $D$ is hyperelliptic};
\node[below] at (1.5,-.2) {$P_c^\mathrm{\KSBA}$};
\node[below] at (1.5,-.65) { $5/6 \ge c > 3/4$};

\end{tikzpicture}
\end{tabular} \arrow[rd, swap, "\text{contract non-normal locus}"] \arrow[rr,dashed,xshift=2.5ex,"\text{flip blue locus}"] & & \begin{tabular}{c}
\begin{tikzpicture}[gren0/.style = {draw, circle,fill=greener!80,scale=.7},gren/.style ={draw, circle, fill=greener!80,scale=.4},blk/.style ={draw, circle, fill=black!,scale=.2},plc/.style ={draw, circle, color=white!100,fill=white!100,scale=0.02},smt/.style ={draw, circle, color=gray!100,fill=gray!100,scale=0.02},lbl/.style ={scale=.2}] 

\draw [black] plot [smooth cycle, tension=.1] coordinates { (0,0) (0,3) (3,3) (3,0) };

\draw [thick,blue] plot [smooth, tension=1] coordinates { (.2,1.6) (.8,1.8) (1.6,1.5) };
\draw [thick,densely dotted,blue] plot [smooth, tension=1] coordinates {  (1.6,1.5) (2.2,1.3) (2.7,1.5) };

\draw [thick,violet] plot [smooth cycle, tension=.1] coordinates { (1.1,.7) (1.1,2.5) (2.5,2.3) (2.5,.5) };


\filldraw[color=blue, fill=blue] (1.6,1.5) circle (0.03);

\draw[color=blue, ->] (.9,3.1) -- (.8,1.9);
\node[above, node font=\tiny, color=blue] at (1.3,3.1) {locus of tacnodal plane curves $x^2 + y^4$};
\draw[color=blue, ->] (1.75,1.15) -- (1.6,1.35);
\node[below right, node font=\tiny, color=blue] at (1.75,1.15) {[tacnodal curve on $\bP(1,1,4)$]};
\draw[color=violet, ->] (2.8,2.7) -- (2.4,2.4);
\node[above right, node font=\tiny, color=violet] at (2.8,2.7) {exceptional divisor parametrizes};
\node[above right, node font=\tiny, color=violet] at (2.8,2.4) {pairs $(\bP(1,1,4),cD)$};
\node[above right, node font=\tiny, color=violet] at (3.1,2.1) {smooth $D$ is hyperelliptic};
\node[below] at (1.5,-.2) {$P_c^\mathrm{K}$};
\node[below] at (1.5,-.65) { $3/4 > c > 3/8$};

\end{tikzpicture}
\end{tabular} \arrow[ld,"\text{contract tacnodal locus}"] \arrow[rdd, start anchor={[xshift=-2ex]}, end anchor={[xshift=2ex]}, "\text{blow up double conic point}"] & \\
& & P_{3/4}^{\rm CY}  & & \\
& & & & \begin{tabular}{c}
\begin{tikzpicture}[gren0/.style = {draw, circle,fill=greener!80,scale=.7},gren/.style ={draw, circle, fill=greener!80,scale=.4},blk/.style ={draw, circle, fill=black!,scale=.2},plc/.style ={draw, circle, color=white!100,fill=white!100,scale=0.02},smt/.style ={draw, circle, color=gray!100,fill=gray!100,scale=0.02},lbl/.style ={scale=.2}] 

\draw [black] plot [smooth cycle, tension=.1] coordinates { (0,0) (0,3) (3,3) (3,0) };

\draw [thick,blue] plot [smooth, tension=1] coordinates { (.2,1.6) (.8,1.8) (1.6,1.5) };
\draw [thick,blue] plot [smooth, tension=1] coordinates {  (1.6,1.5) (2.2,1.3) (2.7,1.5) };


\filldraw[color=violet, fill=violet](1.6,1.5) circle (0.05);

\node[smt] at (1.6,1.5) (A){};
\node[above right, node font=\tiny, color=violet] at (A) {[2(smooth conic)]};
\draw[color=blue, ->] (.6,2.2) -- (.8,1.9);
\node[above, node font=\tiny, color=blue] at (.6,2.2) {locus of tacnodal plane curves $x^2 + y^4$};
\node[below] at (1.5,-.2) {$P_4^{\mathrm{GIT}} \cong P_c^\mathrm{K}$};
\node[below] at (1.5,-.68) {$3/8 \ge c > 0$};

\end{tikzpicture}
\end{tabular} \\
\end{tikzcd}

\vspace{-.2in}
\begin{tabular}{c}
\begin{tikzpicture}[gren0/.style = {draw, circle,fill=greener!80,scale=.7},gren/.style ={draw, circle, fill=greener!80,scale=.4},blk/.style ={draw, circle, fill=black!,scale=.2},plc/.style ={draw, circle, color=white!100,fill=white!100,scale=0.02},smt/.style ={draw, circle, color=gray!100,fill=gray!100,scale=0.02},lbl/.style ={scale=.2}] 

\draw[thick,color=gray] (0,0) to (15,0);
\node[blk] at (0,0) (1){};
\node[below, node font=\tiny] at (1) {$c=1$};
\node[blk] at (15,0) (2){};
\node[below, node font=\tiny] at (2) {$c=0$};
\node[blk] at (11.5,0) (3){};
\node[below, node font=\tiny] at (3) {$c=3/8$};
\node[blk] at (6.2,0) (4){};
\node[below, node font=\tiny] at (4) {$c=3/4$};
\node[blk] at (1,0) (5){};
\node[below, node font=\tiny] at (5) {$c=5/6$};
\end{tikzpicture}
\end{tabular}
\caption{Wall crossings for moduli of quartic curves.}
\label{f:quarticwalls}
\end{figure}

\section{Wall crossing for quintic plane curves and applications to higher degree}\label{sec:higherdegree}

\subsection{Quintic plane curves}

For quintics, the situation is much more complicated as there exist many walls.  We describe them briefly here.  Full details for the wall crossings for $c \in (0, \frac{3}{5}]$ can be found in \cite{ADLpub, BABWILD}, and the description when $c = \frac{3}{5} + \epsilon$ for $\epsilon \ll 1$ can be found in \cite{hacking}.

\begin{theorem}
For $c \le \frac{3}{7}$, the moduli space $P_{5,c}$ is isomorphic to the GIT moduli space of plane quintics.  For $c \in (0, \frac{3}{5})$, there are five wall crossings for K-moduli spaces of plane quintics. Among them, the first two are weighted blow-ups while the last three are flips. 

After the log Calabi Yau wall crossing, the moduli space $P_{5,c}$ has an explicit description given in \cite[Section 11]{hacking}.
 \end{theorem}

   In Table \ref{table:quintic}, we summarize the behavior of all wall crossings for plane quintics.  The $E_i^-$ column describes the general points that appear in $P_{5, c_i - \epsilon}$ but destabilize for $c_i + \epsilon$ and the $E_i^+$ column describes the general points that appear in $P_{5, c_i + \epsilon}$ but destabilize for $c_i - \epsilon$.  More detailed descriptions of curves in the $E_i^+$ column can be found in \cite{ADLpub}; we only include here those with a simple geometric description.  The justification for the seventh wall can be found in the following section.

\begin{table}
 \begin{tabular}{| p{.05\textwidth} | p{.05\textwidth}| p{.3\textwidth} | p{.5\textwidth}|}\hline
       $i$ & $c_i$ & $E_i^-$ & $E_i^+$ \\ \hline
       1 & $\frac{3}{7}$ & $(\bP^2, Q_5)$ & $(\bP(1,1,4), D)$\\
       2 & $\frac{8}{15}$ & $(\bP^2, A_{12}\textrm{-quintic})$ & $(X_{26}, D): \ D \text{ hyperelliptic}$ \\
       3 & $\frac{6}{11}$ &$ (\bP^2, A_{11}\textrm{-reducible quintics})$ &       $(\bP(1,1,4), D)$ \\
       4 & $\frac{63}{115}$ & $(\bP^2, A_{11}\textrm{-irreducible quintics})$ &
        $(\bP(1,4,25), D)$\\
       5 & $\frac{54}{95}$ & $(\bP^2, A_{10}\textrm{-quintics})$ & $(\bP(1,4,25), D)$\\
       6 & $\frac{3}{5}$ & $(\bP^2, A_{9}\textrm{-quintics})$ & $(\bP(1,1,5) \cup X_6, D):\ D $ union of genus 4 and genus 1 component meeting at 2 points 
       \\
       6 & $\frac{3}{5}$ & $(\bP^2, D_{6}\textrm{-quintics})$ &  $(\bP(1,1,2) \cup \bP(1,1,2), D): \ D $ union of 2 genus 2 components meeting at 3 points \\
       7 & $\frac{11}{18}$ & $(\bP^2, A_{8}\textrm{-quintics})$ &  
       $(Bl_{2,9}\bP^2 \cup \bP(1,2,9), D): \ D $ union of a genus 2 and a genus 4 component meeting at 1 point\\
       \hline
    \end{tabular}
    
    \caption{Wall crossings for K-moduli spaces of plane quintics}
    \label{table:quintic}
\end{table}

\subsection{KSBA wall crossing toward moduli of surfaces of general type}

Consider the KSBA moduli space $P^\KSBA_{5, \frac{4}{5}}$ of pairs compactifying the locus of pairs $(\bP^2, \frac{4}{5} C_5)$ where $C$ is a quintic plane curve.  By taking a cyclic $\bZ_5$ cover of $\bP^2$ branched over the curve $C_5$, we obtain a surface $S$ of general type (in fact, a quintic surface in $\bP^3$). Indeed, if $\pi: S \to \bP^2$ is the cyclic cover, then
\[ K_S = \pi^*(K_{\bP^2} + \tfrac{4}{5} C_5)\]
where $K_{\bP^2} + \frac{4}{5} C_5 \sim L$ where $L$ is a line in $\bP^2$, so $K_S$ is ample. Furthermore, $(K_S)^2 = 5$.  From the study above, we have a complete understanding of the moduli space compactifying the of pairs $(\bP^2, (\frac{3}{5} + \epsilon)C_5)$, and to understand the moduli space of canonically polarized surfaces arising as cyclic covers branched over a quintic in $\bP^2$, we wish to increase the coefficient of $c$ up to $\frac{4}{5}$. Increasing the coefficient and sequentially computing the walls is a promising technique to explicitly understand moduli spaces of canonically polarized surfaces.  We illustrate this by computing the first wall in the KSBA region below, and will explore the full range of coefficients in forthcoming work. 

\begin{theorem}
    The first wall crossing in the KSBA moduli $P_{5,c}^{\KSBA}$ occurs at $c = \frac{11}{18}$, and resolves the locus of curves with an $A_8$ singularity, replacing a general such curve with a curve on the surface $Bl_{(2,9)} \bP^2 \cup \bP(1,2,9)$.
\end{theorem}

\begin{proof}
    In the moduli space compactifying the locus of pairs $(\bP^2, (\frac{3}{5} + \epsilon)C_5)$, the surfaces appearing are: $\bP^2, \bP(1,1,4), X_{26}, \bP(1,4,25), \bP(1,1,2) \cup \bP(1,1,2), \bP(1,1,5) \cup X_6, \bP(1,1,5) \cup \bP(1,4,5)$ by \cite[Section 11.2]{hacking}. Here, $X_{26}$ is the partial smoothing of $\bP(1,4,25)$ given as the weighted hypersurface $ (xw = y^{13} + z^2) \subset \bP(1,2,13,25)$ with weighted coordinates $[x:y:z:w]$, and $X_6$ is the partial smoothing of $\bP(1,4,5)$ given as the weighted hypersurface $ (ut = v^3 + w^2) \subset \bP(1,2,3,5)$ with weighted coordinates $[u:v:w:t]$.
    
    Each component of each of these surfaces has Picard rank 1, and thus for any $(X,c D)$ appearing in this moduli space, as $D$ is effective, it is ample.  Therefore, ampleness of $K_X + (\frac{3}{5} + \epsilon) D$ implies ampleness of $K_X +cD$ for any $c > \frac{3}{5}$.  Therefore, the only way the moduli space can undergo a wall crossing is if a pair $(X, cD)$ becomes unstable because $c > \lct ( D)$.  The problem then becomes a question of existence of curve singularities.  In \cite[Section 11.2]{hacking}, the possible curve singularities are enumerated and, by direct computation, the smallest log canonical threshold occurring is that of an $A_8$ singularity: $\frac{1}{2} + \frac{1}{9} = \frac{11}{18}$. So, for any $c \in (\frac{3}{5},\frac{11}{18}]$, stability of the pair $(X, (\frac{3}{5}+ \epsilon)D)$ implies stability of the pair $(X,c D)$.  However, any $(X,cD)$ such that $D$ has an $A_8$ singularity is unstable for $c> \frac{11}{18}$ so we must compute the replacement. 

    The description of the generic replacements follows from similar arguments to the quartic case: suppose $(\bP^2, cC)$ is a pair such that $C$ has an $A_8$ singularity at one point and no other singularities.  Taking a generic smoothing of $C$ over $\bA^1_t$ so that the $A_8$ singularity is (in local coordinates) given by $t = x^2 - y^9 = 0$, we compute the replacement in $P^{\KSBA}_{5,\frac{11}{18}+\epsilon}$ by performing a $t \to t^{18}$ base change and then a $(9,2,1)$ blow up in these coordinates.  The new central fiber of this family is $Bl_{(2,9)} \bP^2 \cup \bP(1,2,9)$, and it is straightforward to compute that the strict transform of the family of curves on this central fiber satisfies $K_X + cD$ is ample for $c > \frac{11}{18}$ (see Lemma \ref{lem:ampleforc>c0} for a more general version of this computation).  Therefore, this is a stable pair. 
\end{proof}

While the previous result describes the generic behavior, the moduli space $P^{\KSBA}_{5, \frac{11}{18}+\epsilon}$ will parametrize replacements of curves with $A_8$ singularities appearing on any of the surfaces in \cite[Section 11.2]{hacking}.  A full analysis of the possible replacements yields the moduli space compactifying the locus of pairs $(\bP^2, (\frac{11}{18} + \epsilon)C_5)$ and provides the explicit description in Theorem \ref{introthm:quintics}.  In what follows, the notation $Bl_{(2,9)}X$ indicates a blowing up of the surface $X$ at a smooth point with weights $(2,9)$, and the resulting surface is glued to the copy of $\bP(1,2,9)$ along the exceptional divisor so the singularities of index 2 and 9 coincide.

\begin{theorem}
    For $c \in (\frac{11}{18}, \frac{5}{8})$, the surfaces appearing in pairs parametrized by the moduli space $P_{5,c}$ are 
    \[ \bP^2, \bP(1,1,4), X_{26}, \bP(1,4,25), \bP(1,1,2) \cup \bP(1,1,2), \bP(1,1,5) \cup X_6, \bP(1,1,5) \cup \bP(1,4,5) \]
    \[Bl_{(2,9)} \bP^2 \cup \bP(1,2,9), Bl_{(2,9)} X_{26} \cup \bP(1,2,9), Bl_{(2,9)} \bP(1,4,25) \cup \bP(1,2,9), Bl_{(2,9)} \bP(1,1,4) \cup \bP(1,2,9) \]
    \[ \bP(1,2,9) \cup Bl_{(2,9)} \bP(1,1,5) \cup X_6  , \bP(1,2,9) \cup Bl_{(2,9)} \bP(1,1,5) \cup \bP(1,4,5)).\]
    For any such surface $X$, the pairs $(X,c D)$ appearing are those for which $D \in |-\frac{5}{3} K_X|$ such that $D$ has $A_n$, $2 \le n \le 7$, or $D_n$, $4 \le n \le 5$, singularities. 
\end{theorem}

\begin{proof}
    An $A_8$ singularity is analytically irreducible and the $\delta$-invariant of such a singularity (the local contribution to the difference between the arithmetic and geometric genus of the curve) is $4$.  All curves $D$ in pairs $[(X,cD)] \in P_{5,c}$ have arithmetic genus $6$, so there can be at most one $A_8$ singularity on any curve $D$, and it must occur on a component of arithmetic genus at least 4.  If $X = \bP(1,1,2) \cup \bP(1,1,2)$, the curves $D \in |-\frac{5}{3}K_X|$ are the unions of two curves of genus 2 meeting at 3 points, so there are no singularities of type $A_8$ in this linear system of curves.  For $X = \bP(1,1,5) \cup X_6$ or $\bP(1,1,5) \cup \bP(1,4,5)$, curves $D \in |-\frac{5}{3}K_X|$ are unions of arithmetic genus 4 curves on $\bP(1,1,5)$ and arithmetic genus 1 curves on $\bP(1,4,5)$ or $X_6$ meeting at two points, so an $A_8$ singularity can only occur on the $\bP(1,1,5)$ component. Finally, from the analysis in \cite[Section 11]{hacking}, for any $c\in (\frac{3}{5}, \frac{11}{18})$, if a pair $[(X,cD)] \in P_{5,c}$ is such that $D$ has an $A_8$ singularity at a point $p$, $p$ must be in the smooth locus of $X$.  Writing the $A_8$ singularity in local coordinates as $x^2 = y^9$, the $(9,2)$ weighted blow up of $X$ of the point $p$ resolves the singular point $p$ and Lemma \ref{lem:ampleforc>c0}, so after blowing up, the resulting surface $Y \to X$ with exceptional divisor $E$ satisfies $(Y, E_Y+cD_Y)$ is stable for all $c \in (\frac{11}{18}, \lct(D_Y))$. 

    Let $X$ denote one of the surfaces $\bP^2, \bP(1,1,4), X_{26}, \bP(1,4,25), \bP(1,1,5) \cup X_6,$ or $\bP(1,1,5) \cup \bP(1,4,5)$ and $D = (f = 0)$ be a curve on $X$ with an $A_8$ singularity.  Here, the notation $(f = 0)$ refers to the weighted equation of $D$ as described in the following Lemma \ref{lem:classificationdeg5}.  Let $(g = 0)$ be a general curve on $X$ which has at worst nodal singularities. Let $(X \times \bA^1_s,\calD) \to \bA^1_s$ be the pencil of divisors on $X$ given by $f = s g$.  Perform the base change $s \to s^{18}$ and consider the pencil $f = s^{18} g$.  Consider the local coordinates $(x,y,s)$ such that $D$ has an $A_8$ singularity $x^2 = y^9$ at $(0,0,0)$, so $\calD$ has local equation $x^2 - y^9 = s^{18} g$.

    Take the $(9,2,1)$ weighted blow-up of $(X \times \bA^1_s,\calD) \to \bA^1_s$ to form a new family $(\calY, \calD_{\calY}) \to \bA^1_s$.  The strict transform of the family of curves $\calD$ in the exceptional divisor $\bP(1,2,9)$ has at worst nodal singularities by the genericity assumption on $g$.  From Lemma \ref{lem:ampleforc>c0}, we conclude the central fiber $(\calY_0, c{\calD_{\calY}}_0)$ is a stable pair for $c = \frac{11}{18} + \epsilon$ and $\epsilon \ll 1$.  This proves that all such surfaces appear in this moduli space, and the statement on curve singularities follows from the fact that the log canonical threshold $> \frac{11}{18}$. 

    Now, we prove that every stable pair $[(X,cD)] \in P_{5,c}$ is one of these pairs for $c \in (\frac{11}{18},\frac{11}{18}+\epsilon)$.  Let $[(X,cD)]$ be a stable pair and consider a smoothing $(\calX, c \calD) \to T$ over a DVR with closed point $0 \in T$ such that $(\calX^\circ, c\calD^\circ) \to T^\circ$ is a family of smooth plane quintics and the fiber $(\calX_0, c\calD_0) = (X,cD)$.  By hypothesis, $(\calX^\circ, \frac{11}{18}\calD^\circ) \to T^\circ$ is a stable family, so induces a map $T^\circ \to P_{5,\frac{11}{18}}$.  By properness of KSBA moduli, this extends to a map $T \to P_{5,\frac{11}{18}}$.  Up to finite base change, this corresponds to a stable family of stable pairs $(\calX', \frac{11}{18} \calD') \to T'$ with same generic fiber as $(\calX, c \calD) \to T$, where $\calX'_0 = \bP^2, \bP(1,1,4), X_{26}, \bP(1,4,25), \bP(1,1,2) \cup \bP(1,1,2), \bP(1,1,5) \cup X_6,$ or $\bP(1,1,5) \cup \bP(1,4,5)$.  If $(\calX', (\frac{11}{18}+ \epsilon) \calD') \to T'$ is stable for $\epsilon \ll 1$, by uniqueness of canonical models and properness of $P_{5,c}$, we must have $(X',cD')= (\calX'_0, c\calD'_0) \cong (X,cD)$. Assuming instead $(\calX', (\frac{11}{18}+ \epsilon) \calD') \to T'$ is not stable for $\epsilon \ll 1$, as $K_Y+(\frac{11}{18}+\epsilon)D_Y$ is ample for all $[(Y, \frac{11}{18}D_Y)] \in P_{5,\frac{11}{18}}$, we must have $\lct(D') = \frac{11}{18}$.  By the classification in \cite[Section 11]{hacking}, we must have $D'$ is a curve with an $A_8$ singularity.  If $t$ is a uniformizing parameter for $T'$, in suitable coordinates $\calD'$ is given by $x^2 - y^9 = t^n$ for some $n > 0$.  Let $d = \frac{18}{\gcd(18,n)}$ and perform a base change $t \to t^{d}$. In these coordinates, consider the $(9k,2k,1)$ weighted blow-up where $k = \frac{n}{\gcd(18,n)}$ $(\calY, \frac{11}{18} \calD_{\calY}) \to T'$. The resulting central fiber is one of the surfaces listed in the statement and, by construction, the family is a KSBA stable family.  The central fiber $(\calY_0, \frac{11}{18}{\calD_\calY}_0)$ has slc singularities and ${\calD_\calY}_0$ no longer has $A_8$ singularities by construction.  The $A_8$ singularity on the strict transform of $X$ has been resolved, and by Lemma \ref{lem:p129}, curves on the exceptional divisor $\cong \bP(1,2,9)$ that do not have $A_8$ singularities have at worst $A_7$ singularities.  Therefore, ${\calD_\calY}_0$ has singularities as claimed: at worst $A_7$ on $\bP(1,2,9)$, and whatever remaining singularities exist on the strict transform of $X$.  Because $(\calY_0, c{\calD_\calY}_0)$ is then KSBA stable for all $c \in (\frac{11}{18}, \frac{5}{8})$, by uniqueness of canonical models and properness of the moduli space, we conclude $(\calY_0, c{\calD_\calY}_0) \cong (X,cD)$.  
\end{proof}

The previous proof uses the following description of pairs $[(X,cD)] \in P_{5,c}$ for $c \in(\frac{3}{5}, \frac{11}{18})$ from Hacking's classification in \cite{hacking}.

\begin{lemma}\label{lem:classificationdeg5}
    Let $X$ denote one of the surfaces $\bP^2, \bP(1,1,4), X_{26}, \bP(1,4,25), \bP(1,1,2) \cup \bP(1,1,2),$ $ \bP(1,1,5) \cup X_6,$ or $\bP(1,1,5) \cup \bP(1,4,5)$.  For an integral divisor $D \in |-\frac{5}{3}K_X|$, $D$ can be given by a single (weighted) equation $f_d$ on $X$ of degree $d$ of the following forms: 
        \begin{enumerate}
            \item $X = \bP^2$ with coordinates $[x:y:z]$: $D = (f_5(x,y,z) = 0)$.  The general curve in this linear system is smooth.
            \item $X = \bP(1,1,4)$ with weighted coordinates $[x:y:z]$: $D = (f_{10}(x,y,z) = 0)$.  The general curve in this linear system has only nodes as singularities.
            \item $X = X_{26} \subset \bP(1,2,13,25)$ with weighted coordinates $[x:y:z:w]$, given by $ xw = y^{13} + z^2 $: $D = (f_{25}(x,y,z,w) = 0)$.  The general curve in this linear system is smooth.
            \item $X = \bP(1,4,25)$ with weighted coordinates $[x:y:z]$: $D = (f_{50}(x,y,z) = 0)$. The general curve in this linear system has only nodes as singularities.
            \item $X = \bP(1,1,2) \cup \bP(1,1,2) \subset \bP(1,1,1,2)$ with weighted coordinates $[x:y:z:w]$ given by $xy = 0$: $D = (f_5(x,y,z,w) = 0)$.   The general curve in this linear system has only nodes as singularities.
            \item $X = \bP(1,1,5) \cup \bP(1,4,5) \subset \bP(1,1,4,5)$ with weighted coordinates $[x:y:z:w]$ given by $yz = 0$: $D = (f_{10}(x,y,z,w) = 0)$. The general curve in this linear system has only nodes as singularities.
            \item $\bP(1,1,5) \cup X_6$, where $\bP(1,1,5)$ has weighted coordinates $[x:y:z]$ and $X_6 \subset \bP(1,2,3,5)$ with weighted coordinates $[u:v:w:t]$ given by $ut = v^3 + w^2$, glued along $(y = 0) \subset \bP(1,1,5)$ and $(v = 0) \subset X_6$.  On $\bP(1,1,5)$, $D$ is given by $(f_{10}(x,y,z) = 0)$ and on $X_6$, $D$ is given by $(f_5(u,v,w,t) = 0)$ such that the intersection points of $f_{10}$ with $(y = 0)$ and the intersection points of $f_5$ with $(v = 0)$ agree via the gluing. The general curve in this linear system has only nodes as singularities.
        \end{enumerate}
\end{lemma}

We also include a short lemma regarding curves on $\bP(1,2,9)$.

\begin{lemma}\label{lem:p129}
    Let $C \subset \bP(1,2,9)$ be a curve of degree 18 meeting the section of $\calO(1)$ transversely at one point.  If $\lct(C) \ge \frac{11}{18}$, then $C$ has at worst $A_8$ singularities, and has an $A_8$ singularity if and only if equality holds.
\end{lemma}

\begin{proof}
    First, observe $\lct(C) > \frac{11}{18}$ implies $C$ is reduced and $C$ can have multiplicity at most 3 at any point.  Up to a change of coordinates, assume $C$ has a singularity at $[1:0:0]$. 
    
    Writing the coordinates of $\bP(1,2,9)$ as $[x:y:z]$ and $C$ as a monomial in $x,y,z$, multiplicity at most 3 at every point implies the coefficient of $z^2$ is nonzero.  By completing the square, we may assume $C$ is given by an equation of the form 
    \[ z^2 = f_{18}(x,y).\]
    By the transverse assumption, we must have the coefficient of $y^9$ also nonzero, so there exists some $g(x,y)$ such that $C$ is given by
    \[ z^2 = y^9 + xg_{17}(x,y).\]

    Let $x^{18-2k}y^{k}$ be a monomial in $xg_{17}(x,y)$ with nonzero coefficient and $k$ minimal.  Because $C$ is singular at $[1:0:0]$, $k \ge 1$.  Then, on the coordinate chart $x \ne 0$, $C$ is given by 
    \[ z^2 = \sum_{i = k}^9 b_i y^i. \]
    The singularity at $(0,0)$ is therefore an $A_{k-1}$ singularity with $k-1 \le 8$.  By direct computation, the log canonical threshold of $C$ is $\frac{11}{18}$ if and only if $k = 9$.
\end{proof}

Now, we have a complete description of $P_{5,c}$ for $c \in (\frac{11}{18},\frac{5}{8})$, and at $\frac{5}{8}$, must replace curves with log canonical threshold equal to $\frac{5}{8}$.  Algorithmically, we can sequentially increase the coefficient in this manner and compute explicit wall crossings to determine the moduli space compactifying the locus of pairs  $(\bP^2, cC_5)$ for any $c \in (\frac{3}{5},1]$.  This is a massive computational undertaking, but forthcoming work will study the case $c \le \frac{4}{5}$ and construct moduli of quintic surfaces arising as cyclic covers of $\bP^2$.

\subsection{Using wall crossing to construct degenerations}

In this section, we outline some additional applications of wall crossing.  For example, in moduli spaces of certain surface pairs, we can use the theory of wall crossing to produce surfaces with many components that admit smoothings. 

\begin{theorem}\label{thm:manycomps}
    Let $X$ be a $\bQ$-Gorenstein slc surface and $D \subset X$ a curve such that 
        \begin{enumerate}
            \item $D$ has an $A_n$ singularity at a point $p \in X$ for some $n \ge 1$;
            \item away from $p$, $D$ has at worst nodes as singularities; and
            \item if some component $C$ of $D$ has rational normalization, it either meets $\overline{D \setminus C}$ in at least three points or meets $\overline{D \setminus C}$ in at least one point (respectively, two points) if $n$ is odd (respectively, even) and passes through $p$.
        \end{enumerate} 
    Assume $K_X + c D$ is ample for $c \ge \frac{1}{2} + \frac{1}{n+1}$.  Assume that $[(X,cD)]$ is a point of an irreducible component of a KSBA moduli space whose general points $(X',cD')$ are stable for all $c \in [\frac{1}{2}+\frac{1}{n+1},1]$.

    Let $Y$ be the surface constructed from $X$ by performing the $(a,b)$ weighted blow up of $p$, where $(a,b) =  (1,\frac{n+1}{2})$ if $n$ is odd or $(a,b) = (2, n+1)$ if $n$ is even, and gluing $Y$ along the exceptional divisor $E$ to the weighted projective space $\bP(1,a,b)$ along a section of $\calO(1)$.  Then, $Y$ is a $\bQ$-Gorenstein degeneration of $X$, and for a divisor $D_Y$ in the linear system corresponding to the specialization of $D$, the divisor $K_Y + c D_Y$ is ample for $c > \frac{1}{2} + \frac{1}{n+1}$.  A general choice of $D_Y$ gives $[(Y, cD_Y)]$ is a point of the same irreducible component of the KSBA moduli space as $(X',cD')$ for all $c \in (\frac{1}{2} + \frac{1}{n+1},1]$.  Furthermore, there exists a particular divisor $D_{Y_0}$ in this linear system on $Y$ with an $A_{n-1}$ singularity at a smooth point $p \in \bP(1,a,b)$ such that $D_{Y_0}$ satisfies conditions (1), (2), and (3) of the theorem statement.
\end{theorem}

From this, we derive the following Corollary (which proves Theorem \ref{introthm:manycomps}).

\begin{corollary}\label{cor:manycomps}
    Let $X$ be $\bQ$-Gorenstein slc surface with $k$ components and $D \subset X$ a curve with $j$ components and an $A_n$ singularity satisfying the conditions in the previous theorem.  Then, for any $c \in (\frac{1}{2} + \frac{1}{n+1}, 1]$ and $\ell$ such that $c <\frac{1}{2} + \frac{1}{\ell+1}$, there exists a specialization $Y$ of $X$ with $k+n -\ell$ components containing a divisor $D_Y$ with $j + n - \ell$ components such that $D_Y$ has an $A_\ell$ singularity and $(Y, c D_Y)$ is a stable pair.

    If $c = 1$, and $\ell = 1$, this produces a stable pair $(Y,D_Y)$ with $k+n -1$ components. 
\end{corollary}

For the idea of the proof, suppose $(X,cD)$ as above is a stable pair for $c_0 = \frac{1}{2} + \frac{1}{n+1}$.  As $c_0 = \lct(D)$, for any larger coefficient of $D$ this pair does not have log canonical singularities.  As it is the limit of a stable pair for any $c \ge c_0$, it must undergo a replacement in the wall crossing.  Successively replacing $A_n$ singularities with $A_{n-1}$ singularities will produce the surface. 

We begin with a preliminary lemma. 

\begin{lemma}\label{lem:ampleforc>c0}
    Let $(X,\Delta +c_0D)$ be a stable surface pair with $K_X + \Delta + cD$ ample for all $c \in [c_0,1]$ and assume the following: 
        \begin{enumerate}
            \item $X$ is normal and $D$ has exactly one singular point $p \in X$ and $p \notin \Delta$;
            \item $c_0 = \lct((X, \Delta), D)$ is the log canonical threshold of $D$;
            \item there exists a birational morphism $\pi: Y \to X$ with one exceptional divisor $E$ over $p$ such that
            \[K_Y + \Delta_Y + c_0 D_Y + E = \pi^*(K_X + \Delta + cD) \] and
            $(Y,\Delta_Y + cD_Y+E)$ is dlt for all $c \in [c_0,1]$, where $D_Y$ is the strict transform of $D$ and $\Delta_Y$ is the strict transform of $\Delta$;
            \item $D_Y$ is a union of divisors with positive genus or, if a component has genus $0$, then its intersection with $E + \Delta$ is at least three smooth points.
        \end{enumerate}
    Then, $K_Y + \Delta_Y+ cD_Y + E$ is ample for all $c \in [c_0,1]$ and $(Y,\Delta_Y + cD_Y + E)$ is a stable pair. 
\end{lemma}

\begin{proof}
    Let $\pi: Y \to X$ be the birational morphism extracting $E$, so \[K_Y + \Delta_Y+ c_0D_Y + E = \pi^*(K_X + \Delta+ c_0D).\]  Let $c \in [c_0,1]$.  Then, $K_Y + \Delta_Y+ cD_Y + E = \pi^*(K_X + \Delta+ cD) - (c-c_0)d E$ for $d = \ord_{\pi^{-1}(D)}(E)$.  Because $E$ is exceptional, $E^2 < 0$, so $(K_Y + \Delta_Y + cD_Y + E) \cdot E > 0$ for any $c > c_0$. 

    Now, writing $K_Y + \Delta_Y+ cD_Y + E = K_Y + \Delta_Y + c_0D_Y + E +(c-c_0)D_Y = \pi^*(K_X + \Delta + c_0D) + (c-c_0)D_Y$ and using that $K_X + \Delta+ c_0D$ is ample, we see that for any $C\ne E \in \overline{NE(Y)}$, \[(K_Y + \Delta_Y + cD_Y + E) \cdot C = (\pi^*(K_X + \Delta+ c_0D) + (c-c_0)D_Y) \cdot C = (K_X + \Delta+ c_0D)\cdot \pi(C) + (c-c_0)D_Y \cdot C.\]  As long as $C \not \subset \Supp \ D_Y$ or $D_Y \cdot C \ge 0$, we conclude $(K_Y + \Delta_Y+ cD_Y + E) \cdot C > 0$.  Finally, if $C \subset \Supp \ D_Y$ and $D_Y \cdot C < 0$, which necessarily implies $C^2 < 0$, then write $D_Y = C + C'$, where $C'$ are the other components of $D_Y$.  Then,
    \[ K_Y + \Delta + cD_Y + E  = K_Y + \Delta+ c(C+C') + E = (K_Y + C) + \Delta + E - (1-c) C + cC'  \]
    and intersecting with $C$, as $(K_Y + \Delta + C+E) \cdot C  \ge 0$ and $C^2 < 0$ by assumption, we conclude $(K_Y + \Delta_Y + cD_Y + E) \cdot C > 0$. 

    This proves that $K_Y + \Delta_Y + cD_Y + E$ is ample for any $c \in [c_0,1]$. The stability of the pair then follows by assumption (3). 
\end{proof}

Now, we prove Theorem \ref{thm:manycomps}. 

\begin{proof}
Let $(\calX, \calD) \to T$ be a generic smoothing of the stable pair $(X,D)$ over the germ of a curve, so $(X,D) = (\calX_0, \calD_0)$ is the closed fiber over $0 \in T$ and choose a local parameter $t$ for the coordinates of $T$.   By assumption, $D$ has a singularity of type $x^2 +y^{n+1}$ at a point $p \in X$ in some local coordinate system $(x,y)$ on $X$.  

Suppose first $n$ is even, and consider the $(n+1,2,1)$ weighted blow up in the coordinates $(x,y,t)$ on $\calX$.  This produces a family $\calY \to T$ where the central fiber is $\calY_0 = \tilde{X} \cup \bP(1,2,n+1)$, where $\tilde{X}$ is the $(n+1,2)$ blow up of $X$ and these surfaces are glued along the rational curve $\Delta \in |\calO_{\bP(1,2,n+1)}(1)|$.  The strict transform of $\calD$ in $\calY$, denoted $\calD_\calY$, has central fiber $(\calD_\calY)_0 = \tilde{D} \cup D_1$ the union of two curves meeting at a single smooth point $p \in \Delta$, where $\tilde{D}$ is the strict transform of $D$ in $\tilde{X}$ and $D_1$ is a curve in $\bP(1,2, n+1)$ of degree $2(n+1)$.  In the component $\tilde{X}$, $\tilde{D}$ is a smooth curve of genus $g > 0$ by assumption, and using genericity of the smoothing, $D_1$ may be assumed to be smooth.  By Lemma \ref{lem:ampleforc>c0}, the pair $(\tilde{X}, c\tilde{D} + \Delta)$ is a stable pair for all $c \in [c_0,1]$.

Observe that there exists a curve $D_2$ of degree $2(n+1)$ on $\bP(1,2,n+1)$ meeting $\Delta$ only at $p$, with an $A_{n-1}$ singularity.  If the coordinates on $\bP(1,2,n+1)$ are $[u:v:w]$, one such curve is given by $w^2 = v^{n}q_2(u,v)$ where $q_2$ is a generic polynomial of degree 2 in $u$ and $v$.  Because $D_2$ is Cartier, the pair $(\bP(1,2,n+1), D_2)$ admits a smoothing to $(\bP(1,2,n+1), D_1)$ for generic $D_1$ as above. Because $\tilde{X} \cup \bP(1,2,n+1)$ is smoothable and $D_2$ is Cartier, this implies $(\tilde{X}\cup \bP(1,2,n+1), \tilde{D} + D_2)$ is smoothable.

Finally, we note that the geometric genus of $D_2$ is equal to $0$, and $D_2$ meets $\tilde{D}$ at one point and has an $A_{n-1}$ singularity, where $n-1$ is odd. For any curve $D' \in |\calO_{\bP(1,2,n+1)}(2(n+1))|$, we have $K_{\bP(1,2,{n+1})} + cD' + \Delta$ has degree $-n-4 + c(2n+2) + 1$, which for $c \ge \frac{1}{2}+ \frac{1}{n}$ and $n \ge 2$, is positive.  Therefore, the pair $(\tilde{X}\cup \bP(1,2,n+1), \tilde{D} + D_2)$ satisfies the assumptions of Theorem \ref{thm:manycomps} where $\tilde{D} + D_2$ has an $A_{n-1}$ singularity. 

If $n$ instead was odd, so the singular point of $D$ has local equation $x^2 + y^{n+1}$, we perform the $(\frac{n+1}{2}, 1, 1)$ weighted blow up in the family resulting in a pair $(\tilde{X}\cup \bP(1,1,\frac{n+1}{2}), \tilde{D} + D_1)$ the intersection of $\tilde{X}$ and $\bP(1,1,\frac{n+1}{2})$ is a divisor $\Delta \in |\calO_{\bP(1,1,\frac{n+1}{2})}(1)|$ and $D_1$ is a divisor of degree $n+1$ on $\bP(1,1, \frac{n+1}{2})$ meeting $\Delta$ in two points (which are the intersection points with $\tilde{D}$).  Applying Lemma \ref{lem:ampleforc>c0} to $\tilde{X} \to X$ and denoting by $E$ the exceptional divisor, we conclude as above $K_{\tilde{X}} + c\tilde{D} + E$ is ample for all $c \in [c_0,1]$. Again as above, we observe that there exists a curve $D_2$ of degree $n+1$ on $\bP(1,1,\frac{n+1}{2})_{[u:v:w]}$ with an $A_{n-1}$ singularity given by $w^2 = v^nu$.  This curve has geometric genus zero and meets $\tilde{D}$ in two points and has an $A_{n-1}$ singularity, where $n-1$ is even.  It also satisfies $K_{\bP(1,1,\frac{n+1}{2})} + cD_2 + \Delta$ is ample for $c \ge \frac{1}{2}+ \frac{1}{n}$. Therefore, it satisfies the assumptions of Theorem \ref{thm:manycomps}.  
\end{proof}

To prove Corollary \ref{cor:manycomps}, we simply use induction.

\begin{proof}
    Let $(X,D)$ be as in the statement such that $D$ has an $A_n$ singularity.  By Theorem \ref{thm:manycomps}, we may produce a surface $(Y,D_Y)$ satisfying the hypotheses of Theorem \ref{thm:manycomps} such that both $Y$ and $D_Y$ have one more component than $X$ and $D_X$, and $D_Y$ has an $A_{n-1}$ singularity.  We are done by induction.
\end{proof}

We apply this to moduli of plane curves for $d \ge 6$ and prove Theorem \ref{introthm:degensP2}.

\begin{corollary}
    In the KSBA moduli space compactifying the locus pairs $(\bP^2, D)$ where $D \in |\calO(d)|$ and $d \ge 6$, there exists a surface with at least $d^2/2 - 1$ components if $d$ is even and $d(d-1)/2  - 1$ components if $d$ is odd.  
\end{corollary}

\begin{proof}
    First, assume $d$ is even.  We construct a curve of degree $d$ with an $A_n$ singularity where $n = d^2/2 - 1$.  To do this, consider the plane curve $C$ given by $f = 0$, where
    \[ f(x,y,z) = (xz^{d/2-1} - y^{d/2})^2 - x^d. \]
    By direct computation, this curve has an isolated singular point at $[0:0: 1]$.  In the coordinate chart where $z \ne 0$, we may write 
    \[ f(x,y,1) = (x - y^{d/2})^2 - x^d\]
    and consider the change of coordinates given by $x' = x - y^{d/2}$, so this becomes 
    \[ (x')^2 - (x'+y^{d/2})^d = (x')^2 - y^{d^2/2} + h.o.t.\]
    where the higher order terms are with respect to the weighting of $(x', y)$ by $(d^2/2, 2)$.  This defines an $A_n$ singularity where $n = d^2/2 - 1$.  By construction, this curve is the union of two smooth degree $d/2$ curves as $f$ factors as 
    \[ f(x,y,z) = (xz^{d/2-1} - y^{d/2})^2 - x^d = (xz^{d/2-1} - y^{d/2} +x^{d/2})(xz^{d/2-1} - y^{d/2} -x^{d/2}).\]
    Because $d/2 \ge 6/2 = 3$, the genus of each component of $C$ is at least one. 

    When $d$ is odd, we construct a plane curve with an $A_n$ singularity where $n = d(d-1)/2  - 1$ in the same way.  This curve $C$ is given by the vanishing of the equation 
    \[ f(x,y,z) = (xz^{(d-1)/2-1} - y^{(d-1)/2})^2z - x^d. \]
    As above, this has an $A_n$ singularity at $[0:0:1]$ as claimed.  This curve is irreducible and its geometric genus is  
    \[ \frac{(d-1)(d-2)}{2} - \lfloor \frac{d(d-1)}{4}\rfloor \ge \frac{(d-1)(d-2)}{2} - \frac{d(d-1)}{4} = \frac{(d-1)(d-4)}{4} > 0  \]
    as $d \ge 6$.

    For any $c > \frac{3}{d} $, $K_{\bP^2} + cC$ is ample, and for any degree $d \ge 6$, $c = \frac{1}{2} + \frac{1}{n} > \frac{1}{2} \ge \frac{3}{d}$, so the hypothesis $K_{\bP^2} + cC$ is ample in the previous theorem is satisfied.  The hypotheses of Corollary \ref{cor:manycomps} are therefore all satisfied, so the statement holds. 
\end{proof}

The curves constructed in these proofs consist of the normalization of $C$ glued to a union of $\frac{d^2}{2} - 2$ or $\frac{d(d-1)}{2} - 2$ (if $d$ is even or odd, respectively) rational curves.  Each rational curve meets the rest of the curve in at least three points, so these are stable curves.  The normalization of $C$ has two components if $d$ is even and one if $d$ is odd.  This observation proves the following:

\begin{corollary}
    For $d \ge 6$, the closure of the locus of planar curves of degree $d$ in the moduli space of Deligne-Mumford stable curves $\oM_g$, $g = \frac{(d-1)(d-2)}{2}$, includes curves with $\frac{d^2}{2}$ components if $d$ is even and $\frac{d(d-1)}{2} - 1$ components if $d$ is odd.
\end{corollary}

While this section has used wall crossing to produce singular limits of plane curves, in \cite{DS}, the idea of wall crossing is applied to study the closure of the locus of \textit{smooth} plane curves in $M_g$.  This perspective merits future exploration and will be studied in upcoming papers.

\section{Wall crossing in the non-proportional case}\label{sec:nonprop}

In the non-proportional case, the K-moduli wall crossing is not well understood. There is a general theory recently developed in \cite{LZreal,LZwalls} but few examples have been worked out.  One example, relating certain K-moduli to GIT moduli and computing non-proportional wall crossing, can be found in \cite{218}.  

In this final section, we present one additional example of non-proportional K-moduli to indicate the methods and computations involved. 

\subsection{K-moduli of hyperelliptic curves on $\bF_1$}

Let $\bF_1$ be the first Hirzebruch surface, the blow up of $\bP^2$ at one point.  Let $s, f$ respectively be the exceptional divisor of the blow up and the strict transform of a line through the point blown up.  We recall some standard facts about $\bF_1$.  The curves are the extremal rays on the Mori cone $\overline{NE}(\bF_1)$, and any prime divisor $D$ on $\bF_1$ is numerically equivalent to $as + b f$ for some $a,b \in \bZ$.  A divisor $D = as + bf$ is ample (respectively, nef) if and only if $b > a > 0$ (respectively, $b \ge a \ge 0$), and it is big if and only if $a, b > 0$.  The volume of such a divisor $D = as + bf$ is given by 
\[  \vol(D) = \left\{ \begin{array}{cc}
    a(2b-a) & \text{ if } b \ge a \ge 0 \\
    b^2 & \text{ if } a > b > 0 \\
    0 & \text{ if } b < 0 \text{ or } a <0 . 
\end{array} \right. \]

The canonical divisor on $\bF_1$ is given by $2s + 3f$.  

Let $D \in | 2s + 4f|$.  We discuss the K-moduli spaces compactifying the locus of pairs $(\bF_1, cD)$ to prove the following theorem, which proves Theorem \ref{introthm:F1}.

\begin{theorem}\label{thm:kmodF1}
    The K-moduli space compactifying the locus of pairs $(\bF_1, c D)$, where $D \in |2s + 4f|$, is nonempty if and only if $c \in [ \frac{1}{10}(4 - \sqrt{6}), \frac{1}{2}).$  When $c = \frac{1}{10}(4 - \sqrt{6})$, the moduli space is a single point.  If $D$ is smooth and meets the negative section of $\bF_1$ transversely, $(\bF_1, cD)$ is K-stable for every $c \in ( \frac{1}{10}(4 - \sqrt{6}), \frac{1}{2})$.
\end{theorem}

The first result is that the K-moduli space is empty if $c \ll 1$. 

\begin{lemma}\label{lem:betaofsection}
    For any $D \in |2s + 4f|$, $\delta_{\bF_1, cD} (s) \ge 1$ if and only if $\ord_D(s) = 0$ and $c \ge c_0 = \frac{1}{10}(4 - \sqrt{6})$.  In particular, if $(\bF_1, cD)$ is K-semistable, then $c \ge c_0$.  Note that $c_0 = \frac{1}{10}(4 - \sqrt{6}) \approx 0.155$ is not rational.
\end{lemma}

\begin{proof}
    We compute $\delta_{\bF_1, cD} (s) = \frac{1 - c\ord_D(s)}{ \frac{(1-c)(7-10c)}{6-9c}} $.  If $(\bF_1, cD)$ is K-semistable, then $\delta{\bF_1, cD} (s) \ge 1$, and solving for $c$ yields $c \ge c_0$ and $\ord_D(s) = 0$.
\end{proof}

\begin{lemma}
    Let $s_{\infty} \in |s + f|$ be a smooth section.  Then, for any $D \in |2s+4f|$, $\delta_{\bF_1, cD}(s_\infty) \ge 1$ implies $\ord_D(s_\infty) \le 2$, and if $\ord_D(s_\infty) = 2$, then $c \le c_0$. 
\end{lemma}

\begin{proof}
    We compute $\delta_{\bF_1, cD} (s_\infty) = \frac{1 - c\ord_D(s_\infty)} {\frac{(28c^2-38c+13)}{12-18c}} $.  If $(\bF_1, cD)$ is K-semistable, then $\delta_{\bF_1, cD} (s_\infty) \ge 1$, and solving for $c$ yields and $\ord_D(s_\infty) \le 2$ and if $\ord_D(s_\infty) = 2$, then $c \le c_0$.
\end{proof}

\begin{lemma}
    Let $f_1, f_2 \in |f|$ be distinct fibers and $s_\infty$ a smooth section as above.  Then, the pair $(\bF_1, c(2s_\infty + f_1 + f_2))$ is K-semistable if and only if $c = c_0$.
\end{lemma}

\begin{proof}
    The previous lemmas tell us $c = c_0$ is necessary.  To prove it is sufficient, consider the action of $GL_2$ on $\bF_1$ given on $\bP^2$ by the subgroup 
    \[ \left\{ \begin{bmatrix} a & b & 0 \\ c & d & 0 \\ 0 & 0 & 1 \end{bmatrix} \right\} \subset \PGL_3.  \]
    Blowing up the point $[0:0:1] \in \bP^2$, which is fixed by the action, induces an action of $GL_2$ on $\bF_1$.  This acts transitively on $|f|$ as this is the (strict transform) of the linear system of lines through the point $[0:0:1]$.  Therefore, there are no fixed points of the action, and the only divisors fixed by the action are the exceptional divisor of the blow up $s$ and the strict transform of $z = 0$, which we may assume is our given section $s_{\infty}$.  By \cite[Theorem 1.2]{Zhuang}, the $\delta$-invariant of this pair can be computed by the infimum of the $\delta$-invariants of invariant divisors.  Therefore, from the computation above, for $c > c_0$ and $c \in \bQ$, 
    \[ \delta(\bF_1, c(2s_\infty+f_1+f_2)) = \delta_{\bF_1, cD} (s_\infty) = \frac{1 - 2c} {\frac{(28c^2-38c+13)}{12-18c}} \] and for $c < c_0$, $c \in \bQ$, we have 
    \[ \delta(\bF_1, c(2s_\infty+f_1+f_2)) = \delta_{\bF_1, cD} (s) = \frac{1}{ \frac{(1-c)(7-10c)}{6-9c}} .\]
    By \cite{LZreal}, for irrational $c$, $\delta(\bF_1, c(2s_\infty+f_1+f_2)) = \lim_{c' \to c} \delta(\bF_1, c'(2s_\infty+f_1+f_2))$ where $c'$ is rational, and the limit of either computation above yields 
    \[ \delta(\bF_1, c(2s_\infty+f_1+f_2)) = 1 \iff c = c_0.\]
    We conclude that $(\bF_1,c_0(2s_\infty + f_1 + f_2)) $ is K-semistable if and only if $c = c_0$. 
\end{proof}

Now, we wish to prove that $(\bF_1, c_0(2s_\infty + f_1 + f_2))$ is K-polystable, which we will ultimately use to show that it is the only point of the associated K-moduli space when $c = c_0$.

\begin{lemma}\label{lem:F1}
    Let $(X,cD)$ be a K-semistable specialization of a K-semistable pair $(\bF_1, cD')$ for $c < \frac{1}{12}(10 - \sqrt{58}) \approx 0.2$.  Then, $X\cong \bF_1$.
\end{lemma}

\begin{proof}
    We use the normalized volume to first show that $X$ must be smooth.  Let $x \in X$ be any point.  Then, by the properties in \ref{propsofvol},
    \[ \widehat{\vol}(X,x) \ge \widehat{\vol}((X,cD),x) \ge \frac{4}{9} (-K_X - cD)^2 = \frac{16}{9}(1-c)(2-3c) \] For  $c < \frac{1}{12}(10 - \sqrt{58})$, we have 
    \[ \widehat{\vol}(X,x) \ge \widehat{\vol}((X,cD),x) \ge \frac{4}{9} (-K_X - cD)^2 = \frac{16}{9}(1-c)(2-3c) > 2. \] 
    By the Gap Conjecture which is known in dimension 2 (see \ref{propsofvol}), the only point of a surface with $\widehat{\vol}(X,x) > 2$ is a smooth point.  Therefore, we conclude $x$ is a smooth point. This proves that $X$ is smooth. 

    Because $X$ is smooth, we have $K_X^2 = K_{\bF^1}^2 = 8$.  By hypothesis, $-K_X - cD$ is ample, and we next prove that $-K_X$ is nef.  Suppose $-K_X$ was not nef, so there exists a curve $C \subset X$ such that $-K_X \cdot C < 0$.  Because $-K_X - cD$ is ample, this implies $-D \cdot C > 0$, which implies $C \subset \Supp \ D$.  Writing 
    $-K_X - cD = -(K_X + C) + (1-c)C - c(\overline{D \setminus C})$, because $D \cdot C < 0$ implies $C \cdot C < 0$ and $c(\overline{D \setminus C}) \cdot C > 0$, we must have $-(K_X + C) \cdot C > 0$ by ampleness of $-K_X - cD$.  This implies $C$ is a smooth rational curve and $-(K_X + C) \cdot C = 2$.  Because $1-c \ge 0.8$, $(1-c) C < 2$ if and only if $C^2 = -1$ or $C^2 = -2$.  In either case, $-K_X \cdot C \ge 0$, so we have derived a contradiction and therefore $-K_X$ is nef.  Because $-K_X = (-K_X - cD) + cD$ is the sum of an ample and an effective divisor, $-K_X$ is also big.  Therefore, $-K_X$ is nef and big.  Because $(-K_X)^2 = 8$, by classification of weak Fano surfaces, we conclude $X = \bP^1 \times \bP^1, \bF_1$, or $\bF_2$.  

    Finally, we show that $X$ cannot be $\bP^1 \times \bP^1$ or $\bF_2$.  This follows as there exists a smooth big and nef divisor $\Delta$ on each (either a section of $\calO(1,1)$ on $\bP^1 \times \bP^1$, or a section missing the $-2$ curve on $\bF_2$) with self intersection $2$.  By assumption, $X$ is smoothable to $\bF_1$, but the divisor $\Delta$ deforms in the family as $H^1(X, \calO(-\Delta)) = 0$, but there exist no smooth curves on $\bF_1$ with self intersection $2$, so this is impossible.  Therefore, $X \cong \bF_1$. 
\end{proof}

Finally, we study specializations of the divisor on $(\bF_1, c_0(2s_\infty + f_1 + f_2))$.

\begin{lemma}\label{lem:transverse}
    If $(\bF_1, cD)$ is K-semistable for $D \in |2s + 4f|$ and $c$ sufficiently close to $c_0$, then $D$ must intersect $s$ transversely with multiplicity 1.
\end{lemma}

\begin{proof}
    Suppose first that $D$ is singular at a point $p \in s$.  Let $e$ be the exceptional divisor of the blow up of $p$.  Then, we compute $\delta_{\bF_1, cD}(e) < 1$ for any $c$.  Next, suppose $D$ meets $s$ non-transversally at a smooth point.  Let $e'$ be the exceptional divisor of the $(1,2)$ weighted blow up of $p$.  Then, we compute $\delta_{\bF_1, cD}(e') < 1$ for any $c \in [c_0, c_0 + \epsilon)$ where $\epsilon \ll 1$.    
\end{proof}

\begin{corollary}
    The pair $(\bF_1, c_0(2s_\infty + f_1 + f_2))$ is K-polystable.
\end{corollary}

\begin{proof}
    Suppose $(X, cD)$ is any K-semistable specialization of $(\bF_1, c_0(2s_\infty + f_1 + f_2))$.  Then, by Lemma \ref{lem:F1}, $X = \bF_1$.  Next, $D$ must be non-reduced with components that are specializations of $s_\infty$ and $f_i$.  This implies $D = 2s_\infty + f_1 + f_2$, or $2s_\infty + 2f$, or $D = 2s + f_i + f_j + f_k + f_{\ell}$ where $f_i,f_j,f_k,f_\ell$ are not necessarily distinct.  By Lemmas \ref{lem:betaofsection} and \ref{lem:transverse}, $D$ cannot contain $s$ and cannot contain a multiple fiber, so we must have $D = 2s_\infty + f_1 + f_2$.  Therefore, we have shown that any K-semistable specialization of $(\bF_1, c_0(2s_\infty + f_1 + f_2))$ is itself, so the pair is K-polystable. 
\end{proof}

\begin{corollary}
    For $\epsilon \ll 1$, every K-semistable pair $(\bF_1, cD)$ for $c \in [c_0, c_0 + \epsilon)$ is the strict transform of a nodal quartic curve in $\bP^2$.
\end{corollary}

\begin{proof}
    The blowup of a curve $C \subset \bP^2$ at a point $p$ meets the negative section $s \subset \bF_1$ at two distinct points if and only if $C$ has a node at $p$.  By the previous lemmas, any curve $D \in |2s + 4f|$ such that $(\bF_1, cD)$ is K-semistable meets $s$ at two distinct points, so contracting the section $s$ realizes $D$ as a nodal quartic curve. 
\end{proof}

\begin{lemma}
    Any nodal quartic curve $C \subset \bP^2$ admits an isotrivial specialization to the curve $xyz^2$.  In particular, the strict transform $D$ of any nodal quartic satisfies $(\bF_1, c_0D)$ is K-semistable.
\end{lemma}

\begin{proof}
    Suppose $C$ has a node at $p = [0:0:1]$, so $C$ can be written as the vanishing locus of an equation 
    \[ f(x,y,z) = g_4(x,y) + g_3(x,y)z + xyz^2\]
    up to change of coordinates, where $g_3, g_4$ are homogeneous polynomials of degrees 3 and 4, respectively.  Consider the action of $\bG_m$ given by $\mathrm{diag} (t,t,t^{-1})$, which acts on $f$ by
    \[  t \cdot f(x,y,z) = t^4g_4(x,y) + t^2g_3(x,y)z + xyz^2.\]  This gives a specialization of $C$ to the curve defined by $xyz^2$. 

    Taking the induced specialization on strict transforms yields an isotrivial specialization of divisors in $|2s + 4f|$ meeting $s$ at two distinct points to $2s_\infty + f_1 + f_2$.  Therefore, by openness of K-semistability, all such pairs are K-semistable. 
\end{proof}

To complete the proof of Theorem \ref{thm:kmodF1}, we must show that the moduli space is nonempty for any $c > c_0$.  To do this, we use the Abban-Zhuang method of admissible flags. 

\begin{theorem}
    For any $c \in (c_0, \frac{1}{2})$ and smooth curve $D \in |2s + 4f |$ that meets $s$ transversely, the pair $(\bF_1, cD)$ is K-stable. 
\end{theorem} 

\begin{proof}
    Using the Abban-Zhuang method of admissible flags with the divisor $s$ or a generic section of $s_\infty$ which does not meet $D$ tangentially, it is straightforward to compute that $\delta(\bF_1, cD) > 1$ for all $c \in (c_0, \frac{1}{2})$.  
\end{proof}

With calculation using the theory of admissible flags, it is also possible to determine all elements parametrized by the K-moduli space for $ c \in [c_0, \frac{1}{2})$.  This is left for forthcoming work.

\bibliography{proposal}{}
\bibliographystyle{alpha}

\end{document}